\documentclass[10pt, oneside]{article}

\usepackage{fancyhdr}
\fancyheadoffset{ 0.5 cm}
\fancyfootoffset{ 0.5 cm}

\usepackage{titlesec}
\titlelabel{\thetitle.\quad}

\usepackage{times}
\usepackage{geometry}
\usepackage{graphicx}
\usepackage{amssymb, amscd}
\usepackage{amsmath}
\makeatletter
\@addtoreset{equation}{section}
\makeatother
\numberwithin{equation}{section}
\usepackage{amsthm} 
\usepackage{mathrsfs}
\usepackage{epstopdf}
\usepackage{enumerate}
\usepackage{url}
\usepackage{graphicx}
\usepackage{mathtools}
\usepackage{tikz}							
\usepackage{setspace}
\usepackage{amsfonts, amsmath, wasysym}

\usepackage{appendix}

\usepackage{relsize}

\usepackage{shuffle}

\usepackage{float}
\restylefloat{table}

\usepackage{hyperref}
\hypersetup{
    colorlinks=true,
    linkcolor=red,
    urlcolor=purple,
    linktoc=all,
    citecolor=blue
           }
\usepackage{cite}

\usepackage{bbm}

\usepackage[draft]{optional}

\usepackage{color}

\usepackage{xcolor}
\newtagform{red}{\color{red}(}{)}

\usepackage{enumitem}

\usepackage{mathtools}

\usepackage{amsmath}
\makeatletter
\@addtoreset{equation}{section}
\makeatother
\numberwithin{equation}{section}

\theoremstyle{plain}
\newtheorem{theorem}{Theorem}[section]
\newtheorem{lemma}[theorem]{Lemma}

\newtheorem{corollary}[theorem]{Corollary}

\newtheorem{claim}[theorem]{Claim}

\theoremstyle{definition}
\newtheorem{definition}[theorem]{Definition}

\theoremstyle{remark}

\newtheorem{remark}[theorem]{\bf{Remark}}

\def\DataSig{DataS{\i}g }

\newcommand{\e}{\ensuremath{{\cal E}}}

\newcommand{\cd}{\ensuremath{{\cal D}}}

\newcommand{\cl}{\ensuremath{{\cal L}}}

\newcommand{\cH}{\ensuremath{{\cal H}}}

\newcommand{\al}{\alpha}

\newcommand{\de}{\delta}

\renewcommand{\th}{\theta}

\newcommand{\vph}{\varphi}
\newcommand{\ep}{\varepsilon}

\newcommand{\R}{\ensuremath{{\mathbb R}}}

\newcommand{\B}{\ensuremath{{\mathbb B}}}

\newcommand{\Z}{\ensuremath{{\mathbb Z}}}

\newcommand{\ovB}{\ensuremath{{ \overline{ \mathbb B }}}}

\newcommand{\bA}{{\bf A }}
\newcommand{\bB}{{\bf B}}
\newcommand{\bC}{{\bf C}}
\newcommand{\bD}{{\bf D}}
\newcommand{\bE}{{\bf E}}

\newcommand{\bI}{{\bf I}}
\newcommand{\bII}{{\bf II}}
\newcommand{\bIII}{{\bf III}}
\newcommand{\bIV}{{\bf IV}}
\newcommand{\bV}{{\bf V}}

\DeclareMathOperator{\inj}{inj}
\DeclareMathOperator{\proj}{proj}

\DeclareMathOperator{\dist}{dist}

\DeclareMathOperator{\Diam}{diam}

\DeclareMathOperator{\pack}{pack}
\DeclareMathOperator{\cov}{cov}

\newcommand{\beq}{\begin{equation}}
\newcommand{\eeq}{\end{equation}}
\newcommand{\beqa}{\begin{equation}\begin{aligned}}
\newcommand{\eeqa}{\end{aligned}\end{equation}}
\newcommand{\brmk}{\begin{rmk}}
\newcommand{\ermk}{\end{rmk}}
\newcommand{\partref}[1]{\hbox{(\csname @roman\endcsname{\ref{#1}})}}

\newcommand{\Lip}{{\mathrm{Lip}}}

\newcommand{\bX}{{ \textbf{X} }}

\usepackage{soul}

\newcommand{\twopartdef}[4]
{
	\left\{
		\begin{array}{ll}
			#1 & \mbox{if } #2 \\ 
			#3 & \mbox{if } #4
		\end{array}
	\right.
}

\newcommand{\threepartdef}[6]
{ 
	\left\{
		\begin{array}{lll}
			#1 & \mbox{if } #2 \\ 
			#3 & \mbox{if } #4 \\ 
			#5 & \mbox{if } #6
		\end{array}
	\right.
}

\makeatletter
\def\dual#1{\expandafter\dual@aux#1\@nil}
\def\dual@aux#1/#2\@nil{\begin{tabular}{@{}c@{}}#1\\#2\end{tabular}}
\makeatother

\makeatletter
\def\three#1{\expandafter\three@aux#1\@nil}
\def\three@aux#1/#2/#3\@nil{\begin{tabular}{@{}c@{}c@{}}#1\\#2\\#3\end{tabular}}
\makeatother

\makeatletter
\def\four#1{\expandafter\four@aux#1\@nil}
\def\four@aux#1/#2/#3/#4\@nil{\begin{tabular}{@{}c@{}c@{}c@{}}#1\\#2\\#3\\#4\end{tabular}}
\makeatother

\title{Higher Order Lipschitz Sandwich Theorems}
\author{Terry Lyons and Andrew D. McLeod}
\date{\today}

\begin{document}
\usetagform{red}
\maketitle

\begin{abstract}
We investigate the consequence of two 
$\Lip(\gamma)$ functions, in the sense of Stein, being 
close throughout a subset of their domain.
A particular consequence of our results is the following.
Given $K_0 > \ep > 0$ and $\gamma > \eta > 0$ there is a 
constant $\de = \de(\gamma,\eta,\ep,K_0) > 0$ for which 
the following is true. 
Let $\Sigma \subset \R^d$ be closed and  
$f , h : \Sigma \to \R$ be $\Lip(\gamma)$ functions whose 
$\Lip(\gamma)$ norms are both bounded above by $K_0$. 
Suppose $B \subset \Sigma$ is closed and that $f$ and $h$ 
coincide 
throughout $B$. Then over the set of points in $\Sigma$ whose 
distance to $B$ is at most $\de$ we have that the 
$\Lip(\eta)$ norm of the difference $f-h$ is bounded above 
by $\ep$.
More generally, we establish that this phenomenon remains 
valid in a less restrictive Banach space setting under the 
weaker hypothesis that the two $\Lip(\gamma)$ functions
$f$ and $h$ are only close in a pointwise sense 
throughout the closed subset $B$.
We require only that the subset $\Sigma$ be closed; 
in particular, the case that $\Sigma$ is finite is covered
by our results.
The restriction that $\eta < \gamma$ is sharp in the sense
that our result is false for $\eta := \gamma$.
\end{abstract}

\begin{center}
\textsl{Mathematics Subject Classification 2020}:
\textbf{26B35} (primary);
\textbf{26D07}, \textbf{46A32}, \textbf{46B28},
\textbf{46M05} (secondary).
\end{center}

\tableofcontents

\section{Introduction}
\label{intro}
The notion of $\Lip(\gamma)$ functions, for $\gamma > 0$, 
introduced in \cite{Ste70} provides an extension 
of $\gamma$-H\"{o}lder regularity that is both non-trivial 
and meaningful even when $\gamma > 1$.
This notion of regularity is the appropriate one for the 
study of \textit{rough paths} instigated by the first author 
in \cite{Lyo98}; an introductory overview to this theory
may be found in \cite{CLL04}, for example. 
Moreover, $\Lip(\gamma)$ regularity underpins the efforts 
made to extend the theory of rough paths to the setting of 
manifolds \cite{CLL12,BL22}.
Further the flow of $\Lip(\gamma)$ vector fields is utilised
to investigate the \textit{accessibility problem} regarding
the use of classical ODEs to obtain the terminal solution to 
a \textit{rough differential equations} driven by 
\textit{geometric rough paths} in \cite{Bou15,Bou22}.
The notion of $\Lip(\gamma)$ regularity is well-defined 
for functions defined on arbitrary closed subsets including,
in particular, finite subsets.

The origin of $\Lip(\gamma)$ functions go back at least as 
far as the original extension work of Whitney in \cite{Whi34-I}.
This work considered the following extension problem. 
Given an integer $m \in \Z_{\geq 1}$ and a closed subset 
$A \subset \R^d$, when can a real-valued function 
$f : A \to \R$ be extended to a $C^m(\R^d)$ function 
$F : \R^d \to \R$ such that $F \rvert_A \equiv f$.
Whitney introduces a definition of $f$ being $m$-times 
continuously differentiable on an arbitrary closed subset 
$A \subset \R^d$, which we denote by $f \in C^m(A)$,
and subsequently establishes that this 
condition is sufficient to ensure that $f$ admits an extension
to an element $F \in C^m(\R^d)$ (see Section 3 and Theorem I in \cite{Whi34-I}).
A variant of this extension result with quantified estimates may be found in \cite{Whi44}.

Whitney's definition of $f \in C^m(A)$ involves assigning a family of ``derivatives" for $f$ on $A$. 
Hence applying Whitney's extension theorem requires one to first fix such an assignment of derivatives. 
Numerous works have subsequently considered the extension problem proposed by Whitney with the additional 
constraint of avoiding such an assignment, i.e. using only the values of the function $f$ throughout $A$.
For the case that $d=1$ Whitney himself provided an answer to this problem using the \textit{divided differences}
of $f$; see \cite{Whi34-II}.

The general case that $d \geq 1$ was fully resolved by Fefferman in \cite{Fef06,Fef07}.
His resolution builds upon the reformulation of Whitney's result in \cite{Whi34-II} as a \textit{finiteness principle}
by Brudnyi and Shvartsman \cite{BS94,BS01}.
The key point is that the finiteness principle no longer involves the divided differences of $f$.
The subsequent finiteness principle established by Fefferman 
in \cite{Fef05} underlies his resolution to the Whitney extension problem (see also \cite{Fef09-I}).
Subsequent algorithmic approaches to computing an extension 
have been considered by Fefferman et al. in \cite{FK09-I,FK09-II,Fef09-II,FIL16,FIL17}.

Returning to Whitney's original extension theorem in \cite{Whi34-I}, analogous results have been established 
in non-Euclidean settings where the domain of $f$ is \textit{not} a subset of $\R^d$ for some $d \in \Z_{\geq 1}$.
A $C^1$ version of Whitney's extension theorem was 
established for real valued mappings defined on subsets of 
the sub-Riemannian Heisenberg group in \cite{FSS01}.
Mappings taking their values in the Heisenberg group have 
also been considered; a version of Whitney's extension theorem 
has been established for horizontal $C^m$ curves in the 
Heisenberg group \cite{Zim18,PSZ19}.
Moreover, a finiteness principle for horizontal curves in
the Heisenberg group is proven in \cite{Zim21}.
Whitney-type extensions for horizontal $C^1$ curves
in general Carnot groups and sub-Riemannian manifolds 
have been considered in \cite{JS17,SS18}.

In this article we focus on Stein's notion of $\Lip(\gamma)$
functions in the Euclidean setting.
Motivated by Whitney's original definition
of $C^m(A)$ for a closed subset $A \subset \R^d$ in 
\cite{Whi34-I},
Stein's definition of a $\Lip(\gamma)$ is an 
extension of the classical notion of 
Lipschitz (or H\"{o}lder) continuity. 
Indeed, for $\gamma \in (0,1]$ Stein's definition 
coincides with the classical notion of a function being 
bounded and $\gamma$-H\"{o}lder continuous.
For $\gamma > 1$ Stein's definition provides a non-trivial 
extension of H\"{o}lder regularity to higher orders.

Stein's definition of a $\Lip(\gamma)$ function is a
refined weaker assignment of a family of ``derivatives" 
to a function $f$ than the assignment proposed in 
Whitney's definition of $C^m(A)$.
As in Whitney's original work \cite{Whi34-I}, Stein's 
definition of $\Lip(\gamma)$ requires, for each point
$x \in A$, the prescription of a polynomial $P_x$ based at 
$x$. 
These polynomials act as proposals for how the function 
should look at points distinct from $x$, and, similarly to 
Whitney's definition of $C^m(A)$, 
these polynomials are required to satisfy Taylor-like 
expansion properties. 
In particular, it is required that for every $x \in A$ we have 
that $P_x(x) = f(x)$, and that the remainder term 
$R(x,y) := f(x) - P_x(y)$ is bounded 
above by $C|y-x|^{\gamma}$ for some constant $C>0$ 
and every pair $x,y \in A$.
This remainder term bound is weaker than the 
remainder term bound imposed in Whitney's definition of
a $C^m(A)$ function in \cite{Whi34-I}.

To illustrate the sense in which $\Lip(\gamma)$ is weaker
than $C^m(A)$ let $\gamma \in \Z_{\geq 1}$ and $O \subset \R^d$
be a non-empty open subset.
If $f : O \to \R$ is $C^{\gamma}(O)$ in the sense of Whitney
\cite{Whi34-I}
then $f$ is $C^{\gamma}(O)$ in the 
classical sense. 
However, if $f : O \to \R$ is $\Lip(\gamma)$ in the sense 
of Stein \cite{Ste70} then $f$ is $C^{\gamma-1}(O)$ in the 
classical sense with its $(\gamma-1)^{\text{th}}$ derivative
$D^{\gamma-1}f$ only guaranteed to be Lipschitz rather than 
$C^1$.

Analogously to Whitney's extension theorem in \cite{Whi34-I}, 
Stein proves an extension theorem establishing that,
given any $\gamma > 0$, if $A \subset \R^d$
is closed and $f : A \to \R$ is $\Lip(\gamma)$, then 
there exists a function $F : \R^d \to \R$ that is $\Lip(\gamma)$
and satisfies that $F \rvert_A \equiv f$; see Theorem 4 in 
Chapter VI of \cite{Ste70}.
Recall that for $\gamma \in (0,1]$ Stein's notion of a 
$\Lip(\gamma)$ function coincides with the classical notion
of a bounded $\gamma$-H\"{o}lder continuous function.
Consequently, for $\gamma \in (0,1]$, Stein's extension theorem
recovers earlier extension results independently established 
by Whitney in \cite{Whi34-I} and McShane in \cite{McS34}.
Stein's extension theorem allows us to make the following 
conclusion regarding the original Whitney extension problem.
Given an integer $k \in \Z_{\geq 0}$ and any 
$\ep > 0$, the assumption that $f : A \to \R$ is $\Lip(k+\ep)$
ensures that $f$ admits an extension to $F \in C^k(\R^d)$.
Finally, Stein's extension theorem is essential in 
the first author's work on rough paths in \cite{Lyo98}. 

Whilst our focus within this article is on Stein's notion of
$\Lip(\gamma)$ functions within a Euclidean framework, 
it is worth mentioning some recent works proposing 
generalisations of Stein's definition to non-Euclidean 
frameworks. A definition of $\Lip(\gamma)$ functions for 
functions on \textit{Carnot groups} that is consistent 
with Stein's original definition in \cite{Ste70} is proposed 
in \cite{PV06}; further, the authors establish extension 
theorems analogous to Stein's extension theorem for this 
notion of $\Lip(\gamma)$ functions on \textit{Carnot groups} 
in \cite{PV06}.
More generally, building on the rough integration theory 
of cocyclic one-forms developed in \cite{LY15} and 
\cite{Yan16} generalising the work of the first author in 
\cite{Lyo98}, a definition of $\Lip(\gamma)$ functions 
on unparameterised paths is proposed in \cite{Nej18}.
It is established, in particular, in \cite{Nej18} that
an extension property analogous to the Stein-Whitney extension 
property is valid for this notion of $\Lip(\gamma)$ functions
on unparameterised path space.

It is clear that, in general, there is no uniqueness 
associated to Stein's extension theorem; for example 
the mapping $[0,1] \to \R$ given by $t \mapsto |t|$ 
is $\Lip(1)$ (i.e. Lipschitz in the classical sense),
and there are numerous distinct ways it can be extended
to a bounded Lipschitz continuous mapping $\R \to \R$.
Nevertheless, it seems intuitively clear that any two 
extensions must be, in some sense, \textit{close} for points
$x \in \R$ that are not, in some sense, too far away from 
the interval $[0,1]$.

The aim of this article is to establish precise results 
realising this intuition for $\Lip(\gamma)$ functions 
for any $\gamma > 0$.
If $\gamma \in (0,1]$, which we recall means that
Stein's notion of $\Lip(\gamma)$ functions \cite{Ste70} 
coincides with the classical notion of bounded 
$\gamma$-H\"{o}lder continuous functions, then such results 
are well-known. In particular, they arise as
immediate consequences of the maximal and minimal 
extensions of Whitney \cite{Whi34-I} and McShane 
\cite{McS34} respectively.
The main novelty of this paper is establishing our main 
results in the case that $\gamma > 1$ 
(cf. the \textit{Lipschitz Sandwich Theorem} \ref{lip_k_cover_thm},
the \textit{Single-Point Lipschitz Sandwich} Theorem 
\ref{lip_k_ball_estimates_thm}, and the
\textit{pointwise Lipschitz sandwich theorem} 
\ref{lip_k_cover_C0_thm} in Section \ref{main_results_sec}), 
for which we 
have been unable to locate formal statements of such 
properties within the existing literature.
Moreover, our results are presented in a more general setting 
than restricting to working within $\R^d$ for some 
$d \in \Z_{\geq 1}$. In particular, we consider the 
Banach space framework 
utilised in both \cite{Lyo98} and \cite{Bou15}, for example.

A particular consequence of our results is that $\Lip(\gamma)$ 
functions can be cost-effectively approximated.
Loosely, if $0 < \eta < \gamma$, then on compact sets 
the $\Lip(\eta)$-behaviour of a $\Lip(\gamma)$ function 
is determined, up to an arbitrarily small error $\ep > 0$, 
via knowledge
of an upper bound for the $\Lip(\gamma)$ norm of the function
on the entire compact set, and the knowledge of the 
value of the $\Lip(\gamma)$ function at a finite number of 
points.
The number of points at which it is required to know 
the value of the $\Lip(\gamma)$ function depends on 
the upper bound of its $\Lip(\gamma)$ norm on the entire 
compact set, the regularity parameters $\eta$ and $\gamma$, 
the desired error size $\ep > 0$, and the compact subset
(cf. Corollary \ref{lip_k_finite_subset_corollary} in Section 
\ref{sec:cost-effect_approx}).

The remainder of the paper is structured as follows. 
In Section \ref{sec:notation} we provide Stein's definition of a 
$\Lip(\gamma)$ function within a Banach space framework
(cf. Definition \ref{lip_k_def}) and fix the notation and 
conventions that will be used throughout the article.
In Section \ref{main_results_sec} we 
state our main results; the 
\textit{Lipschitz Sandwich Theorem} \ref{lip_k_cover_thm},
the \textit{Single-Point Lipschitz Sandwich Theorem} 
\ref{lip_k_ball_estimates_thm}, and the
\textit{Pointwise Lipschitz Sandwich Theorem} 
\ref{lip_k_cover_C0_thm}.
In Section \ref{sec:cost-effect_approx} we 
illustrate how the 
\textit{Lipschitz Sandwich Theorem} \ref{lip_k_cover_thm}
and the 
\textit{Pointwise Lipschitz Sandwich Theorem} 
\ref{lip_k_cover_C0_thm} allow one to 
cost-effectively approximate $\Lip(\gamma)$ functions
defined on compact subsets
(cf. Corollaries \ref{lip_k_finite_subset_corollary} 
and \ref{lip_k_finite_subset_C0_corollary} and  
Remarks \ref{rmk:finite_subset_determines_behaviour}, 
\ref{rmk:explicit_example_1}, 
\ref{rmk:pointwise_behaviour_finite_subset}, 
and \ref{rmk:explicit_example_two}).
In Section \ref{lip_k_functions} we  
establish explicit pointwise remainder term 
estimates for a  $\Lip(\gamma)$ function.
In Section \ref{nested_embed_sec} we record that, 
for $\gamma_1 > \gamma_2 > 0$, any $\Lip(\gamma_1)$
function is also a $\Lip(\gamma_2)$ function.
Additionally, we establish an explicit constants 
$C \geq 1$ for which we have that
$||\cdot||_{\Lip(\gamma_2)} \leq C || \cdot ||_{\Lip(\gamma_1)}$ 
(cf. the \textit{Lipschitz Nesting} 
Lemma \ref{lip_k_spaces_nested_lemma}).
In Section \ref{local_lip_bds_sec} we record, given a 
a $\Lip(\gamma)$ function $f$ defined on a subset $\Sigma$, 
a point $p \in \Sigma$, and $\eta \in (0,\gamma)$, 
quantified estimates for the $\Lip(\eta)$-norm of $f$ over
a neighbourhood of the point $p$ in terms of the value of 
$f$ at the point $p$ (cf. Lemmas 
\ref{lip_k_lip_norm_ball_bound_lemma} and 
\ref{lip_k_func_ext_lemma_gen_eta}). Sections 
\ref{lip_k_sand_thms_proofs}, 
\ref{single_point_lip_sand_thm_pf_sec}, and 
\ref{lip_sand_thm_pf_sec} contain the proofs of the main 
results stated in Section \ref{main_results_sec}.
In Section \ref{lip_k_sand_thms_proofs} we prove the  
\textit{Pointwise Lipschitz Sandwich Theorem} 
\ref{lip_k_cover_C0_thm}. 
In Section \ref{single_point_lip_sand_thm_pf_sec} we prove the 
\textit{Single-Point Lipschitz Sandwich Theorem} 
\ref{lip_k_ball_estimates_thm}.
Finally, in Section \ref{lip_sand_thm_pf_sec} we establish the full 
\textit{Lipschitz Sandwich Theorem} \ref{lip_k_cover_thm}.

Several of our intermediary lemmata record properties of
$\Lip(\gamma)$ functions that appear elsewhere in the literature.
Nevertheless there are two main benefits of including these results.
Firstly, the variants we record are in specific 
forms that are particularly useful for our purposes. 
Secondly, their inclusion makes our article fully self-contained.
\vskip 4pt
\noindent 
\textit{Acknowledgements}: This work was supported by 
the \DataSig Program under the EPSRC grant 
ES/S026347/1, the Alan Turing Institute under the
EPSRC grant EP/N510129/1, 
the Mathematical Foundations of
Intelligence: An ``Erlangen Programme" for AI under the
EPSRC Grant EP/Y028872/1,
the Data Centric Engineering 
Programme (under Lloyd's Register Foundation grant G0095),
the Defence and Security Programme (funded by the UK 
Government) and the Hong Kong Innovation and Technology 
Commission (InnoHK Project CIMDA). 
This work was funded by the Defence and Security
Programme (funded by the UK Government).
For the purpose of open access, the author has applied a
CC BY public copyright licence to any Author Accepted Manuscript
(AAM) version arising from this submission.

\section{Mathematical Framework and Notation}
\label{sec:notation}
In this section we introduce the 
mathematical framework considered throughout the remainder 
of the article, provide Stein's definition of a 
$\Lip(\gamma)$ function within this framework 
(cf. Definition \ref{lip_k_def}), 
and fix the notation and conventions that will be used
throughout the article.
Throughout the remainder of this article, 
when referring to metric balls 
we use the convention that those denoted by $\B$
are taken to be open and those denoted by $\ovB$ 
are taken to be closed.

Let $V$ and $W$ be real Banach spaces and assume that 
$\Sigma \subset V$ is a closed subset. The first goal of 
this section is to define the space $\Lip(\gamma,\Sigma,W)$
of $\Lip(\gamma)$ functions with domain $\Sigma$ and target $W$. 
This will require a choice of norms for the 
tensor powers of $V$. We restrict to considering norms that 
are \textit{admissible} in the sense originating in \cite{Sch50}.
The precise definition is the following.

\begin{definition}[Admissible Norms on Tensor Powers]
\label{admissible_tensor_norm}
Let $V$ be a Banach space. We say that its tensor powers
are endowed with admissible norms if for each 
$n \in \Z_{\geq 2}$ we have equipped the tensor power 
$V^{\otimes n}$ of $V$ with a norm 
$|| \cdot ||_{V^{\otimes n}}$ such that the following 
conditions hold.
\begin{itemize}
    \item For each $n \in \Z_{\geq 1}$ the symmetric 
    group $S_n$ acts by isometries on $V^{\otimes n}$,
    i.e. for any $\rho \in S_n$ and any
    $v \in V^{\otimes n}$ we have
    \beq
        \label{sym_group_iso}
            || \rho (v) ||_{V^{\otimes n}}
            = || v ||_{V^{\otimes n}}.
    \eeq
    The action of $S_n$ on $V^{\otimes n}$
    is given by permuting the order of the letters,
    i.e. if $a_1 \otimes \ldots \otimes a_n 
    \in V^{\otimes n}$ and $\rho \in S_n$ then 
    $\rho ( a_1 \otimes \ldots \otimes a_n )
    := a_{\rho(1)} \otimes \ldots \otimes a_{\rho(n)}$,
    and the action is extended to the entirety of
    $V^{\otimes n}$ by linearity.
    \item For any $n,m \in \Z_{\geq 1}$ and any
    $v \in V^{\otimes n}$ and $w \in V^{\otimes m}$
    we have
    \beq
        \label{ten_prod_unit_norm}
            || v \otimes w ||_{V^{\otimes (n+m)}}
            \leq 
            || v ||_{V^{\otimes n}} 
            || w ||_{V^{\otimes m}}.
    \eeq
    \item For any $n,m \in \Z_{\geq 1}$ and any
    $\phi \in \left(V^{\otimes n}\right)^{\ast}$ 
    and $\sigma \in \left(V^{\otimes m}\right)^{\ast}$
    we have
    \beq
        \label{ten_prod_unit_dual_norm}
            || \phi \otimes \sigma
            ||_{\left(V^{\otimes (n+m)}\right)^{\ast}}
            \leq 
            || \phi 
            ||_{\left(V^{\otimes n}\right)^{\ast}} 
            || \sigma 
            ||_{\left(V^{\otimes m}\right)^{\ast}}.
    \eeq
    Here, given any $k \in \Z_{\geq 1}$,
    the norm 
    $|| \cdot ||_{\left(V^{\otimes k}\right)^{\ast}}$
    denotes the dual-norm induced by
    $|| \cdot ||_{V^{\otimes k}}$.
\end{itemize}
\end{definition}
\vskip 4pt
\noindent
It turns out (see \cite{Rya02}) that having 
\textit{both} the inequalities \eqref{ten_prod_unit_norm} 
and \eqref{ten_prod_unit_dual_norm} ensures that
we in fact have equality in both estimates.
Hence if the tensor powers of $V$ are equipped
with admissible norms, we have equality 
in both \eqref{ten_prod_unit_norm} and 
\eqref{ten_prod_unit_dual_norm}.

There are, in some sense, two main examples of 
admissible tensor norms. 
The first is the \textit{projective tensor norm}.
This is defined, for $n \geq 2$, on $V^{\otimes n}$ by
setting, for $v \in V^{\otimes n}$,
\beq
    \label{proj_ten_norm}  
        ||v||_{\proj,n} 
        :=
        \inf \left\{ 
        \sum_{i=1}^{\infty} \prod_{k=1}^n || a_{i_k} ||_V
        ~:~
        v = \sum_{i=1}^{\infty}
        a_{i_1} \otimes \ldots \otimes a_{i_n}
        \text{ and }
        \sum_{i=1}^{\infty} \prod_{k=1}^n || a_{i_k} ||_V 
        < \infty
        \right\}.
\eeq 
The second is the \textit{injective tensor norm}.
This is defined, for $n \geq 2$, on $V^{\otimes n}$ by
setting, for $v \in V^{\otimes n}$,
\beq
    \label{inj_ten_norm}
        ||v||_{\inj,n} := 
        \sup \left\{ \left|
        \vph_1 \otimes \ldots \otimes \vph_n (v)
        \right| 
        ~:~
        \vph_1 , \ldots , \vph_n \in V^{\ast} 
        \text{ and }
        ||\vph_1||_{V^{\ast}} = \ldots = 
        ||\vph_n||_{V^{\ast}} = 1
        \right\}.
\eeq
The injective and projective tensor norms are the main
two examples in the following sense. 
If we equip the tensor powers of $V$ with \textit{any} 
choice of admissible tensor norms in the sense of 
Definition \ref{admissible_tensor_norm}, then for every 
$n \in \Z_{\geq 2}$, if $||\cdot||_{V^{\otimes n}}$ denotes
the admissible tensor norm chosen for $V^{\otimes n}$, 
we have that for every $v \in V^{\otimes n}$ that
(cf. Proposition 2.1 in \cite{Rya02})
\beq
    \label{inj<norm<proj}
        ||v||_{\inj,n} \leq ||v||_{V^{\otimes n}} 
        \leq || v ||_{\proj , n}.
\eeq
In the case that the 
norm $||\cdot||_V$ is induced by an inner product 
$\left< \cdot , \cdot \right>_V$, 
in the sense that 
$||\cdot||_V = \sqrt{\left< \cdot , \cdot \right>_V}$,
we can equip the tensor powers of $V$ with admissible 
norms in the sense of Definition \ref{admissible_tensor_norm}
by extending the inner product $\left< \cdot , \cdot \right>_V$
to the tensor powers of $V$.

That is, suppose $n \in \Z_{\geq 2}$.
Then we define an inner product 
$\left< \cdot , \cdot \right>_{V^{\otimes n}}$ on 
$V^{\otimes n}$ as follows.
First, if $u , w \in V^{\otimes n}$ are given 
by $u = u_1 \otimes \ldots \otimes u_n$
and $w = w_1 \otimes \ldots \otimes w_n$
for elements $u_1 , \ldots , u_n , w_1 , \ldots , w_n \in V$,
define 
\beq
    \label{V^n_inner_product}
        \left< u , w \right>_{V^{\otimes n}}
        :=
        \prod_{s=1}^n \left< u_s , w_s \right>_V.
\eeq
Subsequently, extend 
$\left< \cdot , \cdot \right>_{V^{\otimes n}}$
defined in \eqref{V^n_inner_product} to the entirety of 
$V^{\otimes n} \times V^{\otimes n}$ by linearity so that 
the resulting function 
$\left< \cdot , \cdot \right>_{V^{\otimes n}} : 
V^{\otimes n} \times V^{\otimes n} \to \R$
defines an inner product on $V^{\otimes n}$.
For each $n \in \Z_{\geq 2}$, equip
$V^{\otimes n}$ with the norm $||\cdot||_{V^{\otimes n}} := 
\sqrt{\left< \cdot , \cdot \right>_{V^{\otimes n}}}$
induced by the inner product 
$\left< \cdot , \cdot \right>_{V^{\otimes n}}$. 
Then the tensor powers of $V$ are all equipped with admissible 
norms in the sense of Definition \ref{admissible_tensor_norm}.

Returning to the general setting that $V$ and $W$ are merely
assumed to be real Banach spaces,
we now define a $\Lip(\gamma,\Sigma,W)$ function.

\begin{definition}[$\Lip(\gamma, \Sigma,W)$ functions]
\label{lip_k_def}
Let $V$ and $W$ be Banach spaces, $\Sigma \subset V$ 
a closed subset, and assume that the tensor powers of 
$V$ are all equipped with admissible norms 
(cf. Definition \ref{admissible_tensor_norm}). 
Let $\gamma > 0$ with $k \in \Z_{\geq 0}$ such that
$\gamma \in (k,k+1]$. Suppose that 
$\psi^{(0)} : \Sigma \to W$ is a function, 
and that for each 
$l \in \{1, \ldots , k\}$ we have a function 
$\psi^{(l)} : \Sigma \to \cl( V^{\otimes l} ; W)$
taking its values in the space of symmetric 
$l$-linear forms from $V$ to $W$. Then the collection 
$\psi = (\psi^{(0)} , \psi^{(1)} , \ldots , \psi^{(k)})$ 
is an element of $\Lip(\gamma,\Sigma,W)$ if there exists
a constant $M \geq 0$ for which the following
conditions hold:
\begin{itemize}
    \item For each $l \in \{0, \ldots , k\}$ and 
    every $x \in \Sigma$ we have that
    \beq    
        \label{lip_k_bdd}
            || \psi^{(l)}(x) ||_{\cl(V^{\otimes l} ; W)}
            \leq M
    \eeq
    \item For each $j \in \{0, \ldots , k\}$ 
    define $R^{\psi}_j : \Sigma \times \Sigma \to 
    \cl (V^{\otimes j} ; W)$ for  
    $z,p \in \Sigma$ and $v \in V^{\otimes j}$ by 
    \beq
        \label{lip_k_tay_expansion}
            R^{\psi}_j(z,p)[v] :=
            \psi^{(j)}(p)[v] -
            \sum_{s=0}^{k-j} \frac{1}{s!}
            \psi^{(j+s)}(z)
            \left[v \otimes (p-z)^{\otimes s}\right].
    \eeq
    Then whenever $l \in \{0, \ldots , k\}$ and 
    $x,y \in \Sigma$ we have 
    \beq
        \label{lip_k_remain_holder}
            \left|\left| 
            R^{\psi}_l(x,y) 
            \right|\right|_{\cl(V^{\otimes l} ; W)}
            \leq
            M ||y - x||_V^{\gamma - l}.
    \eeq
\end{itemize}
We sometimes say that $\psi \in \Lip(\gamma,\Sigma,W)$
without explicitly mentioning the functions 
$\psi^{(0)} , \ldots , \psi^{(k)}$. 
Furthermore, given $l \in \{0, \ldots , k\}$, 
we introduce the notation that 
$\psi_{[l]} := (\psi^{(0)} , \ldots , \psi^{(l)})$.
The $\Lip(\gamma,\Sigma,W)$ norm of $\psi$, denoted by
$||\psi||_{\Lip(\gamma,\Sigma,W)}$, is the smallest 
$M \geq 0$ satisfying the requirements
\eqref{lip_k_bdd} and \eqref{lip_k_remain_holder}.
\end{definition}
\vskip 4pt
\noindent
In Definition \ref{lip_k_def} we use that $V^{\otimes 0} := \R$ to 
observe that $\cl(V^{\otimes 0};W) = W$. 
Thus we implicitly assume that $\cl(V^{\otimes 0};W) = W$ is 
taken to be equipped with the norm $||\cdot||_W$ on $W$; that is, 
$|| \cdot ||_{\cl(V^{\otimes 0};W)} = || \cdot ||_W$ in both 
\eqref{lip_k_bdd} and \eqref{lip_k_remain_holder}. Thus
a consequence of Definition \ref{lip_k_def} 
is that $\psi^{(0)} \in C^0(\Sigma;W)$ with 
$\left|\left| \psi^{(0)} \right|\right|_{C^0(\Sigma;W)}
\leq || \psi ||_{\Lip(\gamma,\Sigma,W)}$. Here, given 
$f \in C^0(\Sigma;W)$, we take 
$||f||_{C^0(\Sigma;W)} := 
\sup \left\{ ||f(x)||_W : x \in \Sigma \right\}$.

Further, we implicitly assume in Definition \ref{lip_k_def}
that, for each $j \in \{1, \ldots , k\}$,
norms have been chosen for the 
spaces $\cl(V^{\otimes j} ; W)$ of symmetric $j$-linear 
forms from $V$ to $W$. 
Observe that $\cl (V^{\otimes j} ; W) \subset 
L(V^{\otimes j} ; W)$ where $L(V^{\otimes j} ; W)$ denotes
the space of linear maps $V^{\otimes j}$ to $W$.
There are, of course, numerous possible choices for such 
norms. 
Throughout this article we will always assume the following 
choice. Given a norm $||\cdot||_{V^{\otimes j}}$ on 
$V^{\otimes j}$ and a norm $||\cdot||_{W}$ on $W$, we equip
$L(V^{\otimes j} ; W)$
with the corresponding operator norm.
That is, for any $\bA \in L(V^{\otimes j} ; W)$ we 
have
\beq
    \label{choice_of_op_norm}
        || \bA ||_{L(V^{\otimes j} ; W)}
        := 
        \sup \left\{ 
        \left|\left| \bA[v] \right|\right|_W 
        : v \in V^{\otimes j} \text{ and }
        ||v||_{V^{\otimes j}} = 1 \right\}.
\eeq
When $\bA \in \cl(V^{\otimes j} ; W)$ we will denote the 
norm defined in \eqref{choice_of_op_norm} by 
$||\bA||_{\cl(V^{\otimes j} ; W)}$.
This choice of norm means that our definition of a 
$\Lip(\gamma,\Sigma,W)$ function differs from the definition
used in \cite{Bou15}.
Indeed we require the bounds in \eqref{lip_k_bdd} 
and \eqref{lip_k_remain_holder} to hold for the 
operator norms, whilst in \cite{Bou15} estimates of the 
analogous form are only required to be valid for 
\textit{rank-one} elements in $V^{\otimes j}$; that is, 
for elements $v = v_1 \otimes \ldots \otimes v_j$ where 
$v_1 , \ldots , v_j \in V$.
Consequently, our definition of a $\Lip(\gamma,\Sigma,W)$ 
function is more restrictive than the notion in \cite{Bou15}.

A good way to understand a $\Lip(\gamma,\Sigma,W)$
function is as a function that ``locally looks like a 
polynomial function". Given any point $x \in \Sigma$,
consider the polynomial $\Psi_x : V \to W$ 
defined for $y \in V$ by
\beq
    \label{intro_lip_k_psi_proposed_value}
        \Psi_x(y)  
        := \sum_{s=0}^{k} 
        \frac{1}{s!}
        \psi^{(s)}(x)\left[   
        (y-x)^{\otimes s} \right].
\eeq
The polynomial $\Psi_x$ defined in 
\eqref{intro_lip_k_psi_proposed_value} gives a proposal, 
based at the point $x \in \Sigma$, for how the function 
$\psi$ behaves away from $x$.
The remainder term estimates in \eqref{lip_k_remain_holder}
of Definition \ref{lip_k_def} ensure that for every 
$y \in \Sigma$ we have that
\beq
    \label{lip_k_proposal_diff_bd}
        \left|\left| \psi^{(0)}(y) - \Psi_x(y) 
        \right|\right|_W 
        \leq
        ||\psi||_{\Lip(\gamma,\Sigma,W)}
        ||y-x||_V^{\gamma}.
\eeq
It follows from \eqref{lip_k_proposal_diff_bd} that, 
for a given $\ep > 0$, if we take 
$\de := (\ep/||\psi||_{\Lip(\gamma,\Sigma,W)})^{1/\gamma}$, 
then the polynomial $\Psi_x$ is within $\ep$ of 
$\psi^{(0)}$, in the $||\cdot||_W$ norm sense, 
throughout the neighbourhood 
$\ovB_V(x,\de) \cap \Sigma$ of the point $x$.

The collection of functions $\psi^{(0)} , \ldots , \psi^{(k)}$ 
are related to
$\Psi_x$ in the following sense. For each 
$l \in \{0, \ldots , k\}$ the element 
$\psi^{(l)}(x) \in \cl(V^{\otimes l};W)$ is the $l^{th}$ 
derivative of $\Psi_x(\cdot)$ at $x$.
The proposal functions $\Psi_x$ for points $x \in \Sigma$ 
enable one to view a 
$\Lip(\gamma,\Sigma,W)$ function in a more traditional 
manner as the single function 
$\Sigma \times V \to W$ defined by the mapping 
$(x,y) \mapsto \Psi_x(y)$.
The remainder term estimates in \eqref{lip_k_remain_holder}
of Definition \ref{lip_k_def} ensure that this mapping exhibits
H\"{o}lder regularity in a classical sense. 

To illustrate assume that $\gamma > 1$ so that $k \geq 1$, 
and suppose we have basepoints $x,w \in \Sigma$ and 
$y,z \in V$ such that $||x-w||_V \leq 1$ and $||y-z||_V \leq 1$.
Let $L_{z-x,y-x} , L_{z-w,z-x} \subset V$ denote the straight lines 
connecting $z-x$ to $y-x$ and $z-w$ to $z-x$ respectively.
Define $r_1 := \max \left\{ ||x-z||_V , ||x-y||_V \right\}$ 
and $r_2 := \max \left\{ ||z-w||_V , ||z-x||_V \right\}$ 
so that, in particular, we have both 
the inclusions $L_{z-x,y-x} \subset \ovB_V(0,r_1)$ and 
$L_{z-w,z-x} \subset \ovB_V(0,r_2)$.
Then we have the H\"{o}lder-type estimate that
\beq
    \label{classical_holder_bd}
        \left|\left| \Psi_x(y) - \Psi_w(z) \right|\right|_W
        \leq
        || \psi ||_{\Lip(\gamma,\Sigma,W)} \left(
        e^{r_1} 
        ||y-z||_V 
        +
        \left( e^{r_2} + e^{1+||z-x||_V}  \right)
        ||x-w||_V^{\gamma - k} \right).
\eeq 
To see this we first observe that 
\beq
    \label{evalpoint_vary_bd}
        \left|\left| \Psi_x(y) - \Psi_x(z) \right|\right|_W
        \stackrel{\eqref{intro_lip_k_psi_proposed_value}}{\leq}
        \sum_{s=0}^k \frac{1}{s!}
        \left|\left| \psi^{(s)}(x) \left[ 
        (y-x)^{\otimes s} - (z-x)^{\otimes s}
        \right] \right|\right|_W 
        \leq
        e^{r_1}
        ||\psi||_{\Lip(\gamma,\Sigma,W)}
        ||y-z||_V
\eeq
where the last inequality follows from applying the 
mean value theorem to each of the mappings
$v \mapsto (v-x)^{\otimes s}$ for $s \in \{1, \ldots , k\}$
and recalling that $L_{z-x,y-x} \subset \ovB_V(0,r_1)$.
Next observe, via \eqref{lip_k_tay_expansion} and 
\eqref{intro_lip_k_psi_proposed_value}, that
\begin{multline}
    \label{basepoint_vary_bd_A}
        \Psi_x(z) - \Psi_w(z) 
        =
        \psi^{(0)}(x) - \psi^{(0)}(w) + 
        \sum_{s=1}^k \frac{1}{s!} 
        R_s^{\psi}(w,x) \left[ (z-x)^{\otimes s} \right] + \\
        \sum_{s=1}^k 
        \sum_{j=1}^{k-s} \frac{1}{s!j!} \psi^{(s+j)}(w)
        \left[ (z-x)^{\otimes s} \otimes (x-w)^{\otimes j} 
        \right] +
        \sum_{s=1}^k \frac{1}{s!}  
        \psi^{(s)}(w) \left[ 
        (z-x)^{\otimes s} - (z-w)^{\otimes s}
        \right].
\end{multline}
To estimate the first term on the RHS 
of \eqref{basepoint_vary_bd_A} we compute 
\begin{multline}
    \label{term_one_basepoint_vary}
        \left|\left| \psi^{(0)}(x) - \psi^{(0)}(w)
        \right|\right|_W
        \stackrel{\eqref{lip_k_tay_expansion}}{\leq}
        \sum_{s=1}^k \frac{1}{s!} 
        \left|\left| \psi^{(s)}(w) \left[ 
        (x-w)^{\otimes s} \right]
        \right|\right|_W 
        + \left|\left| R_0^{\psi}(w,x) 
        \right|\right|_W \\
        \stackrel{
        \eqref{lip_k_bdd} ~\&~ \eqref{lip_k_remain_holder}
        }{\leq}
        \left( \sum_{s=1}^k \frac{1}{s!}||x-w||_V^s +
        ||w-x||_V^{\gamma}
        \right) ||\psi||_{\Lip(\gamma,\Sigma,W)}
        \leq e ||\psi||_{\Lip(\gamma,\Sigma,W)}
        || w - x||_V
\end{multline}
where the last inequality uses that $||w-x||_V \leq 1$.

For the second term on the RHS of \eqref{basepoint_vary_bd_A}
we compute that
\begin{multline}
    \label{term_two_basepoint_vary}
        \sum_{s=1}^k \frac{1}{s!} 
        \left|\left| 
        R_s^{\psi}(w,x) \left[ (z-x)^{\otimes s} \right]
        \right|\right|_W 
        \stackrel{\eqref{lip_k_remain_holder}}{\leq}
        \sum_{s=1}^k \frac{1}{s!}
        ||\psi||_{\Lip(\gamma,\Sigma,W)}
        ||w-x||_V^{\gamma - s} ||z-x||_V^s \\
        \leq 
        \left( e^{||z-x||_V} - 1 \right) 
        ||\psi||_{\Lip(\gamma,\Sigma,W)}
        ||w-x||^{\gamma - k}_V
\end{multline}
where the last inequality uses that $||w-x||_V \leq 1$.

For the third term on the RHS of \eqref{basepoint_vary_bd_A}
we compute that
\begin{multline}
    \label{term_three_basepoint_vary}
        \sum_{s=1}^k 
        \sum_{j=1}^{k-s} \frac{1}{s!j!} 
        \left|\left| \psi^{(s+j)}(w)
        \left[ (z-x)^{\otimes s} \otimes (x-w)^{\otimes j} 
        \right] \right|\right|_W 
        \stackrel{\eqref{lip_k_bdd}}{\leq}
        \sum_{s=1}^k \sum_{j=1}^{k-s}
        \frac{1}{j!s!} 
        ||\psi||_{\Lip(\gamma,\Sigma,W)}
        ||z-x||_V^s ||x-w||^j \\
        \leq 
        \left( e^{||z-x||_V} - 1 \right)(e-1)
        ||\psi||_{\Lip(\gamma,\Sigma,W)}
        ||x-w||_V
\end{multline}
where the last inequality uses that $||w-x||_V \leq 1$.

For the fourth term on the RHS of \eqref{basepoint_vary_bd_A}
we compute that
\beq
    \label{term_four_basepoint_vary}
        \sum_{s=1}^k \frac{1}{s!} 
        \left|\left|
        \psi^{(s)}(w) \left[ 
        (z-x)^{\otimes s} - (z-w)^{\otimes s}
        \right] \right|\right|_W
        \leq 
        e^{r_2} ||\psi||_{\Lip(\gamma,\Sigma,W)}
        ||x-w||_V 
\eeq
where, for each $s \in \{1, \ldots , k\}$, we have 
applied the mean value theorem to the mapping 
$v \mapsto (z-v)^{\otimes s}$ and recalled the 
inclusion $L_{z-w,z-x} \subset \ovB_V(0,r_2)$.

The combination of \eqref{basepoint_vary_bd_A}, 
\eqref{term_one_basepoint_vary}, 
\eqref{term_two_basepoint_vary}, 
\eqref{term_three_basepoint_vary}, and 
\eqref{term_four_basepoint_vary} yields that
\beq
    \label{basepoint_vary_bd}
        \left|\left| \Psi_x(z) - \Psi_w(z) \right|\right|_W
        \leq
        \left( e^{r_2} + e^{1 + ||z-x||_V}
        \right) ||\psi||_{\Lip(\gamma,\Sigma,W)}
        ||x-w||_V^{\gamma - k} 
\eeq
where we have used that $||x-w||_V \leq 1$ means 
$||x-w||_V \leq ||x-w||_V^{\gamma - k}$ since 
$\gamma - k \in (0,1]$.
The combination of \eqref{evalpoint_vary_bd}
and \eqref{basepoint_vary_bd} then establishes the 
estimate claimed in \eqref{classical_holder_bd}.

Returning to considering the collection 
$\psi = \left( \psi^{(0)} , \ldots , \psi^{(k)}
\right) \in \Lip(\gamma,\Sigma,W)$, 
on the interior of $\Sigma$ the functions
$\psi^{(1)} , \ldots , \psi^{(k)}$ are determined 
by $\psi^{(0)}$. The remainder term estimates in
\eqref{lip_k_remain_holder} for $\psi^{(l)}$
for each $l \in \{0,\ldots,k\}$ 
ensure, for each $j \in \{1 , \ldots , k\}$, that
$\psi^{(j)}$ is the classical $j^{th}$ Fr\'{e}chet derivative
of $\psi^{(0)}$ on the interior of $\Sigma$. 
Thus, on the interior of $\Sigma$, 
$\psi^{(0)}$ is $k$ times continuously 
differentiable, and its $k^{th}$ derivative is
$(\gamma - k)$-H\"{o}lder continuous.

\section{Main Results}
\label{main_results_sec}
In this section we state our main results 
and discuss some of their consequences.
Suppose that $V$ and $W$ are real Banach spaces, that 
$\Sigma \subset V$ is a closed subset, and that all the 
tensor powers of $V$ are equipped with admissible norms
(cf. Definition \ref{admissible_tensor_norm}).
Our starting point is to observe that Stein's extension 
theorem (Theorem 4 in Chapter VI of \cite{Ste70}) remains 
valid for functions in $\Lip(\gamma,\Sigma,W)$ provided 
the Banach space $V$ is finite dimensional.
Indeed, following the method proposed by Stein in \cite{Ste70},
one uses the Whitney cube decomposition of 
$V \setminus \Sigma$ (originating in \cite{Whi34-I}) 
to define an appropriately weighted average of the collection 
$\left\{ \Psi_x (\cdot) ~:~ x \in \Sigma \right\}$
to give an extension $\Psi$ of $\psi$ to the entirety of $V$.
Provided $V$ is finite dimensional, this approach and the 
corresponding estimates carry across to our setting verbatim 
from Chapter VI in \cite{Ste70}.
Only the given values of $\psi^{(0)} , \ldots , \psi^{(k)}$
at points $x \in \Sigma$ are used to define this extension;
consequently there is no dependence on the dimension 
of the target space $W$.

Moreover, the operator 
$A : \Lip(\gamma,\Sigma,W) \to \Lip(\gamma,V,W)$
defined by mapping $\phi \in \Lip(\gamma,\Sigma,W)$ to its 
corresponding weighted average $\Phi \in \Lip(\gamma,V,W)$ 
of the collection 
$\left\{ \Phi_x (\cdot) ~:~ x \in \Sigma \right\}$
is a bounded linear operator whose norm depends only 
on $\gamma$ and the dimension of $V$. 
That is, there is a constant $C = C(\gamma,\dim(V)) \geq 1$
such that for any $\phi \in \Lip(\gamma,\Sigma,W)$ we 
have that $A[\phi] \in \Lip(\gamma,V,W)$ satisfies
$|| A[\phi] ||_{\Lip(\gamma,V,W)} \leq 
C || \phi ||_{\Lip(\gamma,\Sigma,W)}$.
This is again a consequence of a verbatim repetition of 
the arguments of Stein in Chapter VI of \cite{Ste70}.

Suppose that $B \subset \Sigma$ is a non-empty 
closed subset. A particular consequence of 
Stein's extension theorem
(Theorem 4 in Chapter VI of \cite{Ste70}) is
that any element in $\Lip(\gamma,B,W)$ can be extended 
to an element in $\Lip(\gamma,\Sigma,W)$.
Recall that it is unreasonable to expect uniqueness 
for such an extension.
We are interested in understanding when 
extensions of an element in $\Lip(\gamma,B,W)$ to
$\Lip(\gamma,\Sigma,W)$ are forced to remain, in some 
sense, close throughout $\Sigma$.
We consider the following problem. Given elements
$\psi = \left(\psi^{(0)}, \ldots ,\psi^{(k)}\right)$
and $\vph = \left(\vph^{(0)}, \ldots ,\vph^{(k)}\right)$ 
in $\Lip(\gamma,\Sigma,W)$, when does knowing that 
$\psi$ and $\vph$ are, in some sense, ``close" 
on $B$ ensure that $\psi$ and $\vph$ remain ``close", 
in some possibly different sense, throughout $\Sigma$.

The following \textit{Lipschitz Sandwich Theorem} gives a 
condition for the subset $B$ and precise meanings for 
the notions of closeness to be considered between $\psi$
and $\vph$ on $B$ and $\Sigma$ respectively under which 
this problem has an affirmative answer.

\begin{theorem}[\textbf{Lipschitz Sandwich Theorem}]
\label{lip_k_cover_thm}
Let $V$ and $W$ be Banach spaces, and assume  
that the tensor powers of 
$V$ are all equipped with admissible norms (cf. 
Definition \ref{admissible_tensor_norm}).
Assume that $\Sigma \subset V$ is non-empty and closed.
Let $\ep >0$, 
$(K_1 , K_2) \in \left( \R_{\geq 0} \times \R_{\geq 0} \right) 
\setminus \{(0,0)\}$, and $\gamma > \eta > 0$ with 
$k,q \in \Z_{\geq 0}$ such that $\gamma \in (k,k+1]$
and $\eta \in (q,q+1]$. Then there exist constants 
$\de_0 = \de_0 (\ep,K_1+K_2,\gamma ,\eta) > 0$ and 
$\ep_0 = \ep_0(\ep, K_1+K_2, \gamma, \eta) > 0$ for
which the following is true.

Suppose $B \subset \Sigma$ is a closed subset that is
a $\de_0$-cover of $\Sigma$ in the sense that
\beq
    \label{lip_k_cover_thm_cover_sigma}
        \Sigma \subset \bigcup_{x \in B}
        \ovB_V (x, \de_0)
        = B_{\de_0}
        := \left\{ v \in V ~:~
        \text{There exists } z \in B \text{ such that }
        ||v - z ||_V \leq \de_0 \right\}.
\eeq
Suppose 
$\psi = \left(\psi^{(0)} ,\ldots ,\psi^{(k)}\right),
\vph = \left(\vph^{(0)} ,\ldots ,\vph^{(k)}\right) 
\in \Lip(\gamma,\Sigma,W)$ satisfy the $\Lip(\gamma,\Sigma,W)$
norm estimates
$||\psi||_{\Lip(\gamma,\Sigma,W)} \leq K_1$ and
$||\vph||_{\Lip(\gamma,\Sigma,W)} \leq K_2$.
Further suppose that for every $l \in \{0, \ldots , k\}$
and every $x \in B$ the difference 
$\psi^{(l)}(x) - \vph^{(l)}(x) \in \cl(V^{\otimes l};W)$
satisfies the bound
\beq
    \label{lip_k_sand_thm_B_close_assump}
        \left|\left| \psi^{(l)}(x) - \vph^{(l)}(x) 
        \right|\right|_{\cl(V^{\otimes l};W)} 
        \leq \ep_0.
\eeq
Then we may conclude that
\beq
    \label{lip_k_cover_thm_conc}
        \left|\left| \psi_{[q]} - \vph_{[q]} 
        \right|\right|_{\Lip(\eta,\Sigma,W)}
        \leq \ep
\eeq
where 
$\psi_{[q]}:= \left(\psi^{(0)},\ldots ,\psi^{(q)}\right)$
and $\vph_{[q]}:= 
\left(\vph^{(0)} , \ldots , \vph^{(q)}\right)$.
\end{theorem}

\begin{remark}
\label{rmk:K_1&K_2_not_both_zero_A}
The condition 
$(K_1 , K_2) \in \left( \R_{\geq 0} \times \R_{\geq 0} \right) 
\setminus \{(0,0)\}$ is imposed to avoid having 
$K_1 = K_2 = 0$.
Since $K_1 = K_2 = 0$ means that $\psi \equiv 0 \equiv \vph$, 
we see that in this case the content of 
Theorem \ref{lip_k_cover_thm} vacuous.
\end{remark}

\begin{remark}
\label{rmk:trivial_case_B=Sigma}
Assume the notation as in Theorem \ref{lip_k_cover_thm}.
For any $\de > 0$, we have that 
$\Sigma$ is a subset of its own $\de$-fattening 
$\Sigma_{\de} := \left\{ v \in V : \exists p \in \Sigma 
\text{ with } ||v - p||_V \leq \de \right\}$.
Consequently, Theorem \ref{lip_k_cover_thm} is valid for the choice 
$B := \Sigma$.  
For this choice of $B$, Theorem \ref{lip_k_cover_thm} tells us that 
there exists a constant 
$\ep_0 = \ep_0 (\ep,K_1 + K_2,\gamma,\eta) > 0$
for which the following is true. 
If $\psi = \left( \psi^{(0)} , \ldots , \psi^{(k)} \right), 
\vph = \left(\vph^{(0)} ,\ldots ,\vph^{(k)}\right) 
\in \Lip(\gamma,\Sigma,W)$ with 
$||\psi||_{\Lip(\gamma,\Sigma,W)} \leq K_1$ and
$||\vph||_{\Lip(\gamma,\Sigma,W)} \leq K_2$ satisfy, 
for every $x \in \Sigma$ and every $l \in \{0, \ldots ,k\}$,
that 
$\left|\left| \psi^{(l)}(x) - \vph^{(l)}(x) 
\right|\right|_{\cl(V^{\otimes l};W)} \leq \ep_0$, 
then we may in fact conclude that 
$||\psi_{[q]} - \phi_{[q]}||_{\Lip(\eta,\Sigma,W)} \leq \ep$
where 
$\psi_{[q]}:= \left(\psi^{(0)},\ldots ,\psi^{(q)}\right)$
and $\vph_{[q]}:= 
\left(\vph^{(0)} , \ldots , \vph^{(q)}\right)$.
\end{remark}

\begin{remark}
\label{lip_sand_thm_alt_phrasing_rmk}
Using the same notation as in Theorem \ref{lip_k_cover_thm},
by taking $\vph \equiv 0$ we may conclude from Theorem 
\ref{lip_k_cover_thm} that if
$\psi = \left( \psi^{(0)} , \ldots , \psi^{(k)} \right) 
\in \Lip(\gamma,\Sigma,W)$ satisfies both that
$||\psi||_{\Lip(\gamma,\Sigma,W)} \leq K_1$ and, 
for every $l \in \{0, \ldots , k\}$ and every $x \in B$, 
that $\left|\left| \psi^{(l)}(x) 
\right|\right|_{\cl(V^{\otimes l};W)} \leq \ep_0$, 
then we have that 
$\left|\left| \psi_{[q]} \right|\right|_{\Lip(\eta,\Sigma,W)} 
\leq \ep$. 
\end{remark}

\begin{remark}
\label{lip_k_sand_thm_B_close_cond_rmk}
Using the same notation as in Theorem 
\ref{lip_k_cover_thm}, the estimates 
\eqref{lip_k_sand_thm_B_close_assump} throughout $B$ 
are a weaker condition than 
$||\psi - \vph||_{\Lip(\gamma,B,W)} \leq \ep_0$.
The bound 
$||\psi - \vph||_{\Lip(\gamma,B,W)} \leq \ep_0$ 
implies that the pointwise estimates in 
\eqref{lip_k_sand_thm_B_close_assump} are valid. But 
the converse is \textit{not} true since the pointwise 
estimates in \eqref{lip_k_sand_thm_B_close_assump}
alone are insufficient to establish the required 
estimates for the remainder terms associated to the
difference $\psi - \vph$ (cf. Definition 
\ref{lip_k_def}).
\end{remark}

\begin{remark}
\label{lip_k_sand_thm_eta_rest_rmk}
The restriction that $\eta \in (0,\gamma)$ in 
Theorem \ref{lip_k_cover_thm} is necessary; the 
theorem is \textit{false} for $\eta := \gamma$.
As an example, fix $K_0,\ep >0$ 
with $\ep < 2 K_0$, let $\de > 0$ and consider a fixed 
$N \in \Z_{\geq 1}$ for which $1/N < \de$. 
Define $\Sigma := \left\{ 0 , 1/N\right\} \subset \R$ and 
$B := \Sigma \setminus \{1/N\} = \{0\} \subset \R$. 
Then we have that $\Sigma \subset [-\de,\de]$ and so 
$B$ is a $\de$-cover of $\Sigma$ as required in 
\eqref{lip_k_cover_thm_cover_sigma}.
Define $\psi , \vph : \Sigma \to \R $ 
by $\psi(0) := 0$, $\psi(1/N) := K_0/N$ and 
$\vph(0) := 0$, $\vph(1/N) := -K_0/N$.
Then $\psi, \vph \in \Lip(1,\Sigma,\R)$
with $||\psi||_{\Lip(1,\Sigma,W)} =
||\vph||_{\Lip(1,\Sigma,W)} = K_0$ and 
$\psi - \vph \equiv 0$ throughout $B$,  
establishing the validity of the bounds 
\eqref{lip_k_sand_thm_B_close_assump} 
for any $\ep_0 \geq 0$. However
$| (\psi - \vph)(1/N) - (\psi-\vph)(0)|
= 2K_0/N = 2K_0 | 1/N - 0|$, which means that 
$||\psi - \vph||_{\Lip(1,\Sigma,\R)} = 2K_0 > \ep$. 
\end{remark}

\begin{remark}
\label{lip_k_B_cover_thm_ep0_req}
It may initially appear that the Theorem should
be valid for any fixed $\ep_0$ with $\ep_0 < \ep$
by suitably restricting $\de_0$, rather than
having to allow $\ep_0$ to depend on 
$\ep, K_1+K_2, \gamma$ and $\eta$. 
But this is \textit{not} the case. If we only assume 
$\ep_0 < \ep$, then the estimates
in \eqref{lip_k_sand_thm_B_close_assump}
can even be insufficient to establish that 
$||\psi - \vph||_{\Lip(\eta,B,W)} \leq \ep$.
For example, let $\gamma := 1$, $\eta := 1/2$, and 
fix $0 < \ep_0 < \ep < 1 < K_0$ such that
$2\ep_0 K_0 > \ep^2$.
Define $x_0 := 2\ep_0/K_0 > 0$ and
consider $\Sigma = B := \{0,x_0\}$.
Define $\psi , \vph : \Sigma \to \R$ by
$\psi(0) := -\ep_0$, $\psi(x_0) := \ep_0$
and $\vph(0) := 0 =: \vph(x_0)$.
Then $\psi ,\vph \in \Lip(1,\Sigma,\R)$, with 
$||\vph||_{\Lip(1,\Sigma,\R)} = 0$ and 
$||\psi ||_{\Lip(1,\Sigma,\R)} = K_0$. 
Moreover, $|\psi - \vph| =|\psi| \leq \ep_0$
throughout $\Sigma = B$ so that the estimates 
\eqref{lip_k_sand_thm_B_close_assump} are valid.
However we may also compute that
$| (\psi - \vph)(x_0) - (\psi-\vph)(0)| = 2 \ep_0
= 2 \ep_0 \sqrt{1/x_0} \sqrt{| x_0 - 0 |} 
= \sqrt{2 \ep_0 K_0} \sqrt{| x_0 - 0 |}$
so that $||\psi - \vph||_{\Lip(1/2,B,\R)} =
\sqrt{2 \ep_0 K_0} > \ep$.
\end{remark}
\vskip 4pt
\noindent 
The issue described in Remark 
\ref{lip_k_B_cover_thm_ep0_req} is only present
when the cardinality of the subset $B$ is greater 
than $1$, i.e. when $B$ contains at least two 
distinct points. 
When $B$ consists of a single point we can in fact
allow for an arbitrary $\ep_0 \in [0,\ep)$ in 
Theorem \ref{lip_k_cover_thm} 
rather than having to allow $\ep_0$ to depend on 
$\ep, K_1+K_2, \gamma$ and $\eta$.
The precise statement is recorded in the following 
theorem.

\begin{theorem}[\textbf{Single-Point Lipschitz 
Sandwich Theorem}]
\label{lip_k_ball_estimates_thm}
Let $V$ and $W$ be Banach spaces and assume that
the tensor powers of $V$ are all equipped with 
admissible tensor norms (cf. Definition 
\ref{admissible_tensor_norm}).
Assume that $\Sigma \subset V$ is closed and 
non-empty. Let $\ep > 0$, 
$(K_1 , K_2) \in \left( \R_{\geq 0} \times \R_{\geq 0} \right) 
\setminus \{(0,0)\}$, 
$\gamma > \eta > 0$ with $k,q \in \Z_{\geq 0}$ such
that $\gamma \in (k,k+1]$ and $\eta \in (q,q+1]$, and 
$0 \leq \ep_0 < \min \left\{ K_1+K_2 ,\ep\right\}$.
Then there exists a constant
$\de_0 = \de_0 (\ep,\ep_0,K_1+K_2,\gamma,\eta)> 0$
for which the following is true.

Suppose $p \in \Sigma$ and that
$\psi = \left(\psi^{(0)} ,\ldots ,\psi^{(k)}\right)$
and 
$\vph = \left(\vph^{(0)} ,\ldots ,\vph^{(k)}\right)$
are elements in $\Lip(\gamma,\Sigma,W)$ with 
$||\psi||_{\Lip(\gamma,\Sigma,W)} \leq K_1$ and
$||\vph||_{\Lip(\gamma,\Sigma,W)} \leq K_2$.
Further suppose that for every $l \in \{0, \ldots , k\}$
the difference 
$\psi^{(l)}(p) - \vph^{(l)}(p) \in \cl(V^{\otimes l};W)$
satisfies the bound
\beq
    \label{lip_k_sand_thm_B_close_assump_ball}
        \left|\left| \psi^{(l)}(p) - \vph^{(l)}(p) 
        \right|\right|_{\cl(V^{\otimes l};W)} 
        \leq \ep_0.
\eeq
Then we may conclude that
\beq
    \label{lip_k_sing_point_cover_thm_conc}
        \left|\left| \psi_{[q]} - \vph_{[q]} 
        \right|\right|_{\Lip(\eta,
        \ovB_V (p , \de_0) ~\cap~ \Sigma,W)}
        \leq \ep
\eeq
where 
$\psi_{[q]}:= 
\left(\psi^{(0)},\ldots ,\psi^{(q)}\right)$
and $\vph_{[q]}:= 
\left(\vph^{(0)} , \ldots , \vph^{(q)}\right)$.
\end{theorem}

\begin{remark}
\label{rmk:K_1&K_2_not_both_zero_B}
The condition 
$(K_1 , K_2) \in \left( \R_{\geq 0} \times \R_{\geq 0} \right) 
\setminus \{(0,0)\}$ is imposed to avoid having 
$K_1 = K_2 = 0$.
Since $K_1 = K_2 = 0$ means that $\psi \equiv 0 \equiv \vph$, 
we see that in this case the content of 
Theorem \ref{lip_k_ball_estimates_thm} vacuous.
\end{remark}

\begin{remark}
\label{single_point_lip_sand_thm_alt_phrasing_rmk}
Using the same notation as in Theorem 
\ref{lip_k_ball_estimates_thm},
by taking $\vph \equiv 0$ we may conclude from Theorem 
\ref{lip_k_ball_estimates_thm} that if
$\psi = \left( \psi^{(0)} , \ldots , \psi^{(k)} \right) 
\in \Lip(\gamma,\Sigma,W)$ satisfies both that
$||\psi||_{\Lip(\gamma,\Sigma,W)} \leq K_1$ and, 
for every $l \in \{0, \ldots , k\}$, 
that $\left|\left| \psi^{(l)}(p) 
\right|\right|_{\cl(V^{\otimes l};W)} \leq \ep_0$, 
then we have that 
$\left|\left| \psi_{[q]} 
\right|\right|_{\Lip(\eta,\ovB_V(p,\de_0) ~\cap~ \Sigma,W)} 
\leq \ep$. 
\end{remark}
\vskip 4pt
\noindent
Establishing Theorem \ref{lip_k_ball_estimates_thm}
will form the first step in our proof of Theorem 
\ref{lip_k_cover_thm}.

Returning our attention to Theorem \ref{lip_k_cover_thm}, 
the Lipschitz estimates obtained in the 
conclusion \eqref{lip_k_cover_thm_conc} yield
pointwise estimates for the difference 
$\psi^{(0)} - \vph^{(0)} : \Sigma \to W$.
In particular, we may conclude that
$\left|\left| \psi^{(0)} - \vph^{(0)} 
\right|\right|_{C^0(\Sigma;W)} \leq \ep$.
However such pointwise estimates can be established 
directly without needing to appeal to Theorem 
\ref{lip_k_cover_thm}. Moreover, this direct approach
allows us to obtain estimates for the difference
$\psi^{(l)} - \vph^{(l)} : \Sigma \to 
\cl(V^{\otimes l};W)$ for every 
$l \in \{0, \ldots , k\}$. An additional benefit is 
that we are able to provide a more explicit constant
$\de_0$ for which we require the subset $B \subset \Sigma$
to be a $\de_0$-cover of $\Sigma$. The precise result is 
recorded in the following theorem.

\begin{theorem}[\textbf{Pointwise Lipschitz Sandwich Theorem}]
\label{lip_k_cover_C0_thm}
Let $V$ and $W$ be Banach spaces, and assume  
that the tensor powers of 
$V$ are all equipped with admissible norms (cf. 
Definition \ref{admissible_tensor_norm}).
Assume that $\Sigma \subset V$ is closed.
Let $\ep > 0$, 
$(K_1 , K_2) \in \left( \R_{\geq 0} \times \R_{\geq 0} \right) 
\setminus \{(0,0)\}$ , $\gamma > 0$ 
with $k \in \Z_{\geq 0}$
such that $\gamma \in (k,k+1]$, and 
$0 \leq \ep_0 < \min \left\{ K_1 + K_2 , \ep \right\}$.
Then given any $l \in \{0, \ldots , k\}$, there exists
a constant $\de_0 = \de_0 (\ep,\ep_0,K_1 + K_2,\gamma,l) > 0$,
defined by
\beq
    \label{lip_k_cover_C0_de0_spec}
        \de_0 := \sup \left\{
        \th > 0 ~:~
        (K_1+K_2) \th^{\gamma - l} + \ep_0 e^{\th} 
        \leq 
        \min \left\{ K_1 + K_2 ~,~ \ep \right\}
        \right\}
        > 0,
\eeq
for which the following is true. 

Suppose $B \subset \Sigma$ is a $\de_0$-cover of 
$\Sigma$ in the sense that 
\beq
    \label{lip_k_cover_C0_cover_sigma}
        \Sigma \subset \bigcup_{x \in B}
        \ovB_V (x, \de_0)
        =
        B_{\de_0} :=
        \left\{ v \in V ~:~
        \text{There exists } z \in B \text{ such that }
        ||v - z||_V \leq \de_0 \right\}. 
\eeq
Suppose 
$\psi = \left(\psi^{(0)} ,\ldots ,\psi^{(k)}\right),
\vph = \left(\vph^{(0)} ,\ldots ,\vph^{(k)}\right) 
\in \Lip(\gamma,\Sigma,W)$ satisfy the 
$\Lip(\gamma,\Sigma,W)$ norm estimates that
$||\psi||_{\Lip(\gamma,\Sigma,W)} \leq K_1$ and 
$||\vph||_{\Lip(\gamma,\Sigma,W)} \leq K_2$.
Further suppose that for every $j \in \{0, \ldots , k\}$
and every $x \in B$ the difference 
$\psi^{(j)}(x) - \vph^{(j)}(x) \in \cl(V^{\otimes j};W)$
satisfies the bound 
\beq
    \label{lip_k_point_sand_thm_B_close_assump}
        \left|\left| \psi^{(j)}(x) - \vph^{(j)}(x) 
        \right|\right|_{\cl(V^{\otimes j};W)} 
        \leq \ep_0.
\eeq
Then we may conclude that for every 
$s \in \{0, \ldots , l\}$ and every $x \in \Sigma$ 
that
\beq
    \label{lip_k_cover_C0_conc}
        \left|\left| \psi^{(s)}(x) - \vph^{(s)}(x) 
        \right|\right|_{\cl(V^{\otimes s};W)}
        \leq \ep.
\eeq
\end{theorem}

\begin{remark}
\label{rmk:K_1&K_2_not_both_zero_C}
The condition 
$(K_1 , K_2) \in \left( \R_{\geq 0} \times \R_{\geq 0} \right) 
\setminus \{(0,0)\}$ is imposed to avoid having 
$K_1 = K_2 = 0$.
Since $K_1 = K_2 = 0$ means that $\psi \equiv 0 \equiv \vph$, 
we see that in this case the content of 
Theorem \ref{lip_k_ball_estimates_thm} vacuous.
\end{remark}

\begin{remark}
In contrasts to Theorem \ref{lip_k_cover_thm}
we are able to deal with arbitrary
$\ep_0 < \ep$ by suitably restricting $\de_0$.
The issue outlined in Remark 
\ref{lip_k_B_cover_thm_ep0_req} is no longer
a problem in this setting since the same notion of 
closeness is used in both the assumption
\eqref{lip_k_point_sand_thm_B_close_assump} 
and the conclusion \eqref{lip_k_cover_C0_conc}.
\end{remark}

\begin{remark}
\label{rmk:trivial_case_B=Sigma_pt_wise}
Assume the notation as in Theorem \ref{lip_k_cover_C0_thm}.
For any $\de > 0$, we have that 
$\Sigma$ is a subset of its own $\de$-fattening 
$\Sigma_{\de} := \left\{ v \in V : \exists p \in \Sigma 
\text{ with } ||v - p||_V \leq \de \right\}$.
Consequently, Theorem \ref{lip_k_cover_C0_thm} 
is valid for the choice $B := \Sigma$.  
For this choice of $B$, Theorem \ref{lip_k_cover_C0_thm}
recovers the following trivial statement. 
Let $\ep_0 < \min \{ K_1 + K_2 , \ep \}$, 
$l \in \{0, \ldots ,k\}$, and 
$\psi = \left( \psi^{(0)} , \ldots , \psi^{(k)} \right), 
\vph = \left(\vph^{(0)} ,\ldots ,\vph^{(k)}\right) 
\in \Lip(\gamma,\Sigma,W)$. 
Suppose that
$||\psi||_{\Lip(\gamma,\Sigma,W)} \leq K_1$ and
$||\vph||_{\Lip(\gamma,\Sigma,W)} \leq K_2$. 
Further suppose, for every point 
$x \in \Sigma$ and every integer $j \in \{0, \ldots , k\}$, 
that 
$\left|\left| \psi^{(j)}(x) - \vph^{(j)}(x) 
\right|\right|_{\cl(V^{\otimes j};W)} \leq \ep_0$. 
Then we have, for every point $x \in \Sigma$ and every 
$s \in \{0 , \ldots , l\}$, that
$\left|\left| \psi^{(s)}(x) - \vph^{(s)}(x) 
\right|\right|_{\cl(V^{\otimes s};W)} \leq \ep$.
\end{remark}

\begin{remark}
\label{pointwise_lip_sand_thm_alt_phrasing_rmk}
Using the same notation as in Theorem \ref{lip_k_cover_C0_thm},
by taking $\vph \equiv 0$ we may conclude from Theorem 
\ref{lip_k_cover_C0_thm} that if
$\psi = \left( \psi^{(0)} , \ldots , \psi^{(k)} \right) 
\in \Lip(\gamma,\Sigma,W)$ satisfies both that
$||\psi||_{\Lip(\gamma,\Sigma,W)} \leq K_1$ and, 
for every $j \in \{0, \ldots , k\}$ and every $x \in B$, 
that $\left|\left| \psi^{(j)}(x) 
\right|\right|_{\cl(V^{\otimes l};W)} \leq \ep_0$, 
then we have, for every $s \in \{0, \ldots , l\}$ and 
every point $x \in \Sigma$, that 
$\left|\left| \psi^{(s)}(x) \right|\right|_{\cl(V^{\otimes s};W)} 
\leq \ep$. 
\end{remark}

\begin{remark}
\label{lip_k_point_sand_thm_deltas_rmk}
It follows from \eqref{lip_k_cover_C0_de0_spec} 
that, for every $l \in \{0, \ldots , k\}$, the
constant $\de_0 = \de_0(\ep,\ep_0,K_1+K_2,\gamma,l) > 0$ 
is bounded above by $1$. Moreover, if the 
constants $\ep$, $\ep_0$, $K_1+K_2$, and $\gamma$ 
remain fixed, we may conclude that the constant
$\de_0$ specified in \eqref{lip_k_cover_C0_de0_spec} 
is decreasing with respect to the argument 
$l \in \{0, \ldots , k\}$ in the sense that the mapping
$l \mapsto \de_0(\ep,\ep_0,K_0,\gamma,l)$ is 
a decreasing function on $\{0, \ldots , k\}$.
\end{remark}

\section{Cost-Effective Approximation Application}
\label{sec:cost-effect_approx}
In this section we use the results presented in 
Section \ref{main_results_sec} to establish that, when 
the closed subset $\Sigma \subset V$ is compact, 
an element $\psi \in \Lip(\gamma,\Sigma,W)$ can be, 
in some to be detailed sense, well-approximated
using only its values at a finite number
of points in $\Sigma$.
We start with the following corollary of 
the \textit{Lipschitz Sandwich Theorem} \ref{lip_k_cover_thm}
establishing that, when $\Sigma \subset V$ is compact, 
$\psi$ can be well-approximated in the $\Lip(\eta,\Sigma,W)$-norm
sense using only the values of $\psi$ at a finite number of 
points in $\Sigma$.
The precise result is the following corollary.

\begin{corollary}[Consequence of the 
\textit{Lipschitz Sandwich Theorem} \ref{lip_k_cover_thm}]
\label{lip_k_finite_subset_corollary}
Let $V$ and $W$ be Banach spaces, and assume  
that the tensor powers of 
$V$ are all equipped with admissible norms (cf. 
Definition \ref{admissible_tensor_norm}).
Assume that $\Sigma \subset V$ is compact.
Let $\ep > 0$, 
$(K_1 , K_2) \in \left( \R_{\geq 0} \times \R_{\geq 0} \right) 
\setminus \{(0,0)\}$, and $\gamma > \eta > 0$ with 
$k,q \in \Z_{\geq 0}$ such that $\gamma \in (k,k+1]$
and $\eta \in (q,q+1]$.
Let $\de_0 = \de_0 (\ep ,K_1+K_2, \gamma, \eta) > 0$
and $\ep_0 = \ep_0(\ep,K_1+K_2,\gamma,\eta) > 0$
denote the constants arising from Theorem 
\ref{lip_k_cover_thm}, and let
$N = N (\Sigma, \ep, K_1+K_2, \gamma, \eta) \in \Z_{\geq 0}$
denote the $\de_0$-covering number of $\Sigma$.
That is, 
\beq
    \label{lip_k_finite_subset_corollary_cover_num_state}
        N :=
        \min \left\{
        d \in \Z : 
            \text{ There exists } 
            x_1 , \ldots , x_d \in \Sigma 
            \text{ such that } 
            \Sigma \subset \bigcup_{j=1}^d
            \ovB_{V} ( x_j , \de_0)
        \right\}.
\eeq
Then there is a finite subset 
$\Sigma_N = \{ z_1 , \ldots , z_N \} \subset \Sigma$
for which the following is true.

Suppose 
$\psi = \left(\psi^{(0)} ,\ldots ,\psi^{(k)}\right), 
\vph = \left(\vph^{(0)} ,\ldots ,\vph^{(k)}\right)
\in \Lip(\gamma,\Sigma,W)$ satisfy the $\Lip(\gamma,\Sigma,W)$
norm estimates that
$||\psi||_{\Lip(\gamma,\Sigma,W)} \leq K_1$ and 
$||\vph||_{\Lip(\gamma,\Sigma,W)} \leq K_2$.
Further suppose that for every $l \in \{0, \ldots , k\}$
and every $j \in \{1, \ldots , N\}$ the difference 
$\psi^{(l)}(z_j) - \vph^{(l)}(z_j) \in \cl(V^{\otimes l};W)$
satisfies the bound 
\beq
    \label{lip_k_finite_subset_corollary_close_assump}
        \left|\left| \psi^{(l)}(z_j) - \vph^{(l)}(z_j) 
        \right|\right|_{\cl(V^{\otimes l};W)} 
        \leq \ep_0.
\eeq
Then we may conclude that 
\beq
    \label{lip_k_finite_subset_corollary_conc}
        \left|\left| \psi_{[q]} - \vph_{[q]} 
        \right|\right|_{\Lip(\eta,\Sigma,W)}
        \leq \ep
\eeq
where $\psi_{[q]}:= 
\left(\psi^{(0)},\ldots ,\psi^{(q)}\right)$
and $\vph_{[q]}:= 
\left(\vph^{(0)} , \ldots , \vph^{(q)}\right)$.
\end{corollary}

\begin{remark}
\label{rmk:Corollary_one_func_zero}
In a similar spirit to Remark \ref{lip_sand_thm_alt_phrasing_rmk} 
and
using the same notation as in Corollary 
\ref{lip_k_finite_subset_corollary},
by taking $\vph \equiv 0$ we conclude from Corollary 
\ref{lip_k_finite_subset_corollary} that if
$\psi = \left( \psi^{(0)} , \ldots , \psi^{(k)} \right) 
\in \Lip(\gamma,\Sigma,W)$ satisfies both
$||\psi||_{\Lip(\gamma,\Sigma,W)} \leq K_1$ and, 
for every $l \in \{0, \ldots , k\}$ and every $x \in \Sigma_N$, 
that $\left|\left| \psi^{(l)}(x) 
\right|\right|_{\cl(V^{\otimes l};W)} \leq \ep_0$, 
then we have that 
$\left|\left| \psi_{[q]} \right|\right|_{\Lip(\eta,\Sigma,W)} \leq \ep$. 
\end{remark}

\begin{remark}
\label{rmk:finite_subset_determines_behaviour}
When $N \in \Z_{\geq 1}$ defined in 
\eqref{lip_k_finite_subset_corollary_cover_num_state} 
is less than 
the cardinality of $\Sigma$, 
Corollary \ref{lip_k_finite_subset_corollary} guarantees 
that we are able to identify a strictly 
smaller collection of points at which the 
behaviour of a $\Lip(\gamma,\Sigma,W)$ function 
determines the functions $\Lip(\eta)$-behaviour up to an 
arbitrarily small error over the entire set $\Sigma$.
That is, using the notation of Corollary 
\ref{lip_k_finite_subset_corollary},  
if $F \in \Lip(\gamma,\Sigma_N,W)$ then any two 
extensions $\psi$ and $\vph$ of $F$ to 
elements in $\Lip(\gamma,\Sigma,W)$, with 
$\Lip(\gamma,\Sigma,W)$-norms bounded above by $K_1$
and $K_2$ respectively,
can differ, in the $\Lip(\eta)$-sense, by at most 
$\ep$ throughout $\Sigma$.

A particular consequence of this is that a function in 
$\Lip(\gamma,\Sigma,W)$ can be cost-effectively approximated.
That is, let
$\psi = \left( \psi^{(0)} , \ldots , \psi^{(k)} \right) \in 
\Lip(\gamma,\Sigma,W)$ with $||\psi||_{\Lip(\gamma,\Sigma,W)} 
\leq K_1$
and suppose that we want to approximate $\psi$ 
in a $\Lip(\eta,\Sigma,W)$-norm sense.
Then Corollary \ref{lip_k_finite_subset_corollary} guarantees 
us that \textit{any} 
$\vph = \left( \vph^{(0)} , \ldots , \vph^{(k)} \right) 
\in \Lip(\gamma,\Sigma,W)$
with $||\vph||_{\Lip(\gamma,\Sigma,W)} \leq K_2$ will 
satisfy that 
$||\psi_{[q]} - \vph_{[q]} ||_{\Lip(\eta,\Sigma,W)} \leq \ep$ 
provided we have, for every point
$x \in \Sigma_N$ and every $l \in \{0, \ldots , k\}$, that 
$\left|\left| \psi^{(l)}(x) - \vph^{(l)}(x) 
\right|\right|_{\cl(V^{\otimes l};W)} \leq \ep_0$.
Thus the task of approximating $\psi$ throughout $\Sigma$ 
in the $\Lip(\eta,\Sigma,W)$-norm 
sense can be reduced to needing only to approximate $\psi$ 
in a pointwise sense at 
the finite number of points in the subset $\Sigma_N$.
\end{remark}

\begin{remark} 
\label{rmk:explicit_example_1}
We illustrate the content of Remark 
\ref{rmk:finite_subset_determines_behaviour} 
via an explicit example. 
For this purpose, let $\gamma > 0$ with $k \in \Z_{\geq 0}$ 
such that $\gamma \in (k,k+1]$.
Fix a choice of $\eta \in (0,\gamma)$ and let 
$q \in \Z_{\geq 0}$ such that 
$\eta \in (q,q+1]$. 
Consider fixed $K_0 , \ep > 0$ and $d \in \Z_{\geq 1}$.
Take $V := \R^d$ equipped with its usual Euclidean norm 
$||\cdot||_2$, take $\Sigma := [0,1]^d \subset \R^d$ to be the 
unit cube in $\R^d$, and take $W := \R$.
Observe that the norm $||\cdot||_2$ is induced by the usual 
Euclidean dot product $\left< \cdot , \cdot \right>_{\R^d}$
on $\R^d$.
Equip the tensor powers of $\R^d$ with admissible norms 
in the sense of Definition \ref{admissible_tensor_norm} by
extending the inner product $\left< \cdot , \cdot \right>_{\R^d}$
to the tensor powers, and subsequently taking the norm induced
by the resulting inner product on the tensor powers 
(cf. Section \ref{sec:notation}). Introduce the notation,
for $x \in \R^d$ and $r > 0$, that 
$\B^d (x,r) := \left\{ y \in \R^d : ||x-y||_2 < r \right\}$.

Retrieve the constants $\de_0 = \de_0(\ep,K_0,\gamma,\eta) > 0$ 
and $\ep_0 = \ep_0(\ep, K_0, \gamma,\eta) > 0$
arising in Corollary \ref{lip_k_finite_subset_corollary} for 
these choices of $\ep$, $\gamma$, and $\eta$, 
and for both the constants $K_1$ and $K_2$ there are $K_0$ here.
Let $N = N ([0,1]^d, \ep, K_0, \gamma, \eta) \in \Z_{\geq 0}$
denote the $\de_0$-covering number of $[0,1]^d$.
That is, 
\beq
\label{lip_k_finite_subset_corollary_cover_num_state_rmk_version}
        N :=  
        \min \left\{ m \in \Z : 
            \text{ There exists } 
            x_1 , \ldots , x_m \in [0,1]^d
            \text{ such that } 
            [0,1]^d \subset \bigcup_{j=1}^m
            \ovB^d ( x_j , \de_0)
        \right\}.
\eeq
We first claim that $N$ defined in 
\eqref{lip_k_finite_subset_corollary_cover_num_state_rmk_version}
satisfies that
\beq
    \label{cover_num_upper_bound_rmk}
        N \leq \frac{2^d}{\omega_d} 
        \left( 1 + \frac{1}{\de_0} \right)^d
\eeq
where $\omega_d$ denotes the Euclidean volume of the 
unit ball $\B^d(0,1) \subset \R^d$.

To see this, observe that the $\de_0$-covering number 
of $[0,1]^d$ is bounded from 
above by the $\de_0$-packing number of $[0,1]^d$ defined by
\begin{multline}
    \label{packing_number_rmk}
        N_{\pack}(\de_0, [0,1]^d , \R^d) := \\
        \max \left\{ m \in \Z : \text{ There exists } 
        x_1 , \ldots , x_m \in [0,1]^d 
        \text{ such that } ||x_i - x_j ||_2 > \de_0 
        \text{ whenever } i \neq j \right\}.
\end{multline}
Suppose $x_1 , \ldots , x_{N_{\pack}(\de_0,[0,1]^d,\R^d)} 
\in [0,1]^d$ satisfy the condition
specified in \eqref{packing_number_rmk}, i.e. that whenever 
$i,j \in \{1, \ldots , N_{\pack}(\de_0,[0,1]^d,\R^d) \}$ 
with $i \neq j$ we have $||x_i - x_j||_2 > \de_0$.
A consequence of this is that the collection of balls 
$\left\{ \B^d (x_i , \de_0/2) : i \in \left\{1 , \ldots , 
N_{\pack}(\de_0,[0,1]^d,\R^d) \right\} \right\}$ are 
pairwise disjoint. 
Moreover, the disjoint union of this collection of balls
is a subset of the cube $[-\de_0/2 , 1 + \de_0/2 ]^d$.
Hence a volume comparison argument yields that 
$N_{\pack}(\de_0, [0,1]^d , \R^d)$ defined 
in \eqref{packing_number_rmk} satisfies that
\beq
    \label{packing_number_upper_bound}
        N_{\pack}(\de_0, [0,1]^d , \R^d) 
        \leq 
        \frac{2^d}{\omega_d} 
        \left( 1 + \frac{1}{\de_0} \right)^d.
\eeq
The estimate claimed in \eqref{cover_num_upper_bound_rmk} 
is now a consequence of 
$N \leq N_{\pack}(\de_0, [0,1]^d , \R^d)$ and 
\eqref{packing_number_upper_bound}.

Define $m \in \Z_{\geq 1}$ by 
\beq
    \label{int_m_def_rmk}
		m := \min \left\{ n \in \Z : 
        n \geq \frac{2^d}{\omega_d} 
        \left( 1 + \frac{1}{\de_0} \right)^d 
        \right\}.
\eeq
Choose distinct points $x_1, \ldots , x_m \in [0,1]^d$. 
Then, via \eqref{cover_num_upper_bound_rmk} and 
\eqref{int_m_def_rmk}, 
we see that $[0,1]^d \subset \cup_{j=1}^m \ovB^d(x_j,\de_0)$.
Let $\psi = \left( \psi^{(0)} , \ldots , \psi^{(k)} \right) 
\in \Lip(\gamma,[0,1]^d,\R)$ with 
$||\psi||_{\Lip(\gamma,[0,1]^d,\R)} \leq K_0$.
Then Corollary \ref{lip_k_finite_subset_corollary} tells us 
that if $\vph = \left( \vph^{(0)} , \ldots , \vph^{(k)} 
\right) \in \Lip(\gamma,[0,1]^d,\R)$ satisfies both that 
$||\vph||_{\Lip(\gamma,[0,1]^d,\R)} \leq K_0$ and, for every 
$j \in \{1, \ldots , m\}$ and every $l \in \{0, \ldots , k\}$, 
that 
$\left|\left| \psi^{(l)}(x_j) - \vph^{(l)}(x_j) 
\right|\right|_{\cl((\R^d)^{\otimes l};\R)} \leq \ep_0$, 
then 
$||\psi_{[q]} - \vph_{[q]} ||_{\Lip(\eta,[0,1]^d,\R)} \leq \ep$.

Therefore, in order to approximate $\psi$ up to an error 
of $\ep$ in the $\Lip(\eta,[0,1]^d,\R)$-norm sense, we need 
only find 
$\vph = \left( \vph^{(0)} , \ldots , \vph^{(k)} 
\right) \in \Lip(\gamma,[0,1]^d,\R)$ satisfying both that 
$||\vph||_{\Lip(\gamma,[0,1]^d,\R)} \leq K_0$ and, for every 
$j \in \{1, \ldots , m\}$ and every $l \in \{0, \ldots , k\}$, 
that
\beq
    \label{eq:psi_approx_conds}
        \left|\left| \psi^{(l)}(x_j) - \vph^{(l)}(x_j) 
        \right|\right|_{\cl((\R^d)^{\otimes l};\R)} 
        \leq \ep_0.
\eeq
That is, up to an error of magnitude $\ep > 0$, the 
$\Lip(\eta)$-behaviour of $\psi$ throughout the entire cube
$[0,1]^d$ is captured by the pointwise values of $\psi$
at the finite number of points 
$x_1 , \ldots , x_m \in [0,1]^d$, and we have the explicit 
upper bound resulting from \eqref{int_m_def_rmk} 
for the number of points $m$ that are required.
\end{remark}
\vskip 4pt
\noindent
We now provide a short proof of Corollary 
\ref{lip_k_finite_subset_corollary} using the 
\textit{Lipschitz Sandwich Theorem} \ref{lip_k_cover_thm}.

\begin{proof}[Proof of Corollary 
\ref{lip_k_finite_subset_corollary}]
Let $V$ and $W$ be Banach spaces, and assume  
that the tensor powers of 
$V$ are all equipped with admissible norms (cf. 
Definition \ref{admissible_tensor_norm}).
Assume that $\Sigma \subset V$ is compact.
Let $\ep > 0$, $\gamma > \eta > 0$ with 
$k,q \in \Z_{\geq 0}$ such that $\gamma \in (k,k+1]$
and $\eta \in (q,q+1]$, and 
$(K_1,K_2) \in \left( \R_{\geq 0} \times \R_{\geq 0} \right) 
\setminus \{(0,0)\}$.
Retrieve the constants $\de_0 = \de_0 
(\ep, K_1+K_2, \gamma, \eta) > 0$ 
and $\ep_0 = \ep_0(\ep,K_1+K_2,\gamma,\eta) > 0$
arising from Theorem
\ref{lip_k_cover_thm} for these choices
of $\ep$, $K_1$, $K_2$, $\gamma$ and $\eta$.
Note that we are not actually applying Theorem
\ref{lip_k_cover_thm}, but simply 
retrieving constants in preparation for its future
application. Define 
$N = N(\Sigma,\ep,K_1+K_2,\gamma,\eta) \in \Z_{\geq 0}$
to be the $\de_0$-covering number for $\Sigma$.
That is,
\beq
    \label{lip_k_finite_subset_corollary_cover_num}
        N := 
        \min \left\{
        a \in \Z : 
            \text{ There exists } 
            x_1 , \ldots , x_a \in \Sigma 
            \text{ such that } 
            \Sigma \subset \bigcup_{j=1}^a
            \ovB_{V} ( x_j , \de_0)
        \right\}.
\eeq
The compactness of $\Sigma$ ensures that $N$ defined in 
\eqref{lip_k_finite_subset_corollary_cover_num} is finite.
Let $z_1 , \ldots , z_N \in \Sigma$ be any 
collection of $N$ points in $\Sigma$ for which
\beq
    \label{lip_k_finite_subset_corollary_cover_prop}
        \Sigma \subset \bigcup_{j=1}^N 
        \ovB_{V} ( z_j , \de_0 ).
\eeq
Set $\Sigma_N := \{ z_1 , \ldots , z_N \}$.

Let
$\psi = \left(\psi^{(0)} ,\ldots ,\psi^{(k)}\right), 
\vph = \left(\vph^{(0)} ,\ldots ,\vph^{(k)}\right) 
\in \Lip(\gamma,\Sigma,W)$ satisfy the $\Lip(\gamma,\Sigma,W)$
norm estimates that
$||\psi||_{\Lip(\gamma,\Sigma,W)} \leq K_1$ and
$||\vph||_{\Lip(\gamma,\Sigma,W)} \leq K_2$.
Suppose that for every $l \in \{0, \ldots , k\}$
and every $j \in \{1, \ldots , N\}$ the difference 
$\psi^{(l)}(z_j) - \vph^{(l)}(z_j) \in \cl(V^{\otimes l};W)$
satisfies the bound 
\beq
    \label{lip_k_finite_subset_corollary_close_assump_pf}
        \left|\left| \psi^{(l)}(z_j) - \vph^{(l)}(z_j) 
        \right|\right|_{\cl(V^{\otimes l};W)} 
        \leq \ep_0.
\eeq
Then 
\eqref{lip_k_finite_subset_corollary_cover_prop}
and
\eqref{lip_k_finite_subset_corollary_close_assump_pf}
enable us to appeal to Theorem \ref{lip_k_cover_thm}, 
with $B := \Sigma_N$, to conclude 
$|| \psi_{[q]} - \vph_{[q]} ||_{\Lip(\eta,\Sigma,W)} 
\leq \ep$ where $\psi_{[q]}:= 
\left(\psi^{(0)},\ldots ,\psi^{(q)}\right)$
and $\vph_{[q]}:= 
\left(\vph^{(0)} , \ldots , \vph^{(q)}\right)$.
This is precisely the estimate claimed in 
\eqref{lip_k_finite_subset_corollary_conc}. This 
completes the proof of Corollary 
\ref{lip_k_finite_subset_corollary}.
\end{proof}
\vskip 4pt
\noindent
If we weaken the sense in which we aim to approximate 
$\psi$ to the pointwise notion considered in the 
\textit{Pointwise Lipschitz Sandwich Theorem} 
\ref{lip_k_cover_C0_thm}, then we are able to establish 
the following consequence of the
\textit{Pointwise Lipschitz Sandwich Theorem}
\ref{lip_k_cover_C0_thm} when the subset 
$\Sigma \subset V$ is compact.

\begin{corollary}[Consequence of the 
\textit{Pointwise Lipschitz Sandwich Theorem}
\ref{lip_k_cover_C0_thm}]
\label{lip_k_finite_subset_C0_corollary}
Let $V$ and $W$ be Banach spaces, and assume  
that the tensor powers of 
$V$ are all equipped with admissible norms (cf. 
Definition \ref{admissible_tensor_norm}).
Assume that $\Sigma \subset V$ is compact.
Let $\ep, \gamma > 0$, with $k \in \Z_{\geq 0}$
such that $\gamma \in (k,k+1]$, 
$(K_1,K_2) \in \left( \R_{\geq 0} \times \R_{\geq 0} \right) 
\setminus \{(0,0)\}$, and 
$0 \leq \ep_0 < \min \left\{ K_1 + K_2 , \ep\right\}$.
Given $l \in \{0, \ldots , k\}$ 
let $\de_0 = \de_0 (\ep , \ep_0, K_1+K_2, \gamma,l) > 0$
denote the constant arising from Theorem 
\ref{lip_k_cover_C0_thm}
(cf. \eqref{lip_k_cover_C0_de0_spec}). Let
$N = N (\Sigma,\ep,\ep_0,K_1+K_2,\gamma,l) \in \Z_{\geq 0}$
denote the $\de_0$-covering number of $\Sigma$.
That is, 
\beq
    \label{lip_k_finite_subset_C0_corollary_cover_num_state}
        N :=
        \min \left\{
        d \in \Z : 
            \text{ There exists } 
            x_1 , \ldots , x_d \in \Sigma 
            \text{ such that } 
            \Sigma \subset \bigcup_{j=1}^d
            \ovB_{V} ( x_j , \de_0)
        \right\}.
\eeq
Then there is a finite subset
$\Sigma_N = \{ z_1 , \ldots , z_N \} \subset \Sigma$
for which the following is true.

Let 
$\psi = \left(\psi^{(0)} ,\ldots ,\psi^{(k)}\right),
\vph = \left(\vph^{(0)} ,\ldots ,\vph^{(k)}\right) 
\in \Lip(\gamma,\Sigma,W)$ satisfy the $\Lip(\gamma,\Sigma,W)$
norm estimates
$||\psi||_{\Lip(\gamma,\Sigma,W)} \leq K_1$ and
$||\vph||_{\Lip(\gamma,\Sigma,W)} \leq K_2$.
Suppose that for every $i \in \{0, \ldots , k\}$
and every $j \in \{1, \ldots , N\}$ the difference 
$\psi^{(i)}(z_j)-\vph^{(i)}(z_j) \in \cl(V^{\otimes i};W)$
satisfies the bound
\beq
    \label{lip_k_finite_subset_C0_corollary_assump}
        \left|\left| \psi^{(i)}(z_j) - \vph^{(i)}(z_j) 
        \right|\right|_{\cl(V^{\otimes s};W)} 
        \leq \ep_0.
\eeq
Then we may conclude that for every 
$s \in \{0, \ldots , l\}$ and every $x \in \Sigma$ that
\beq
    \label{lip_k_finite_subset_C0_corollary_conc}
        \left|\left| \psi^{(s)}(x) - \vph^{(s)}(x) 
        \right|\right|_{\cl(V^{\otimes s};W)}
        \leq \ep.
\eeq
\end{corollary}

\begin{remark}
\label{rmk:pointwise_corollary_one_func_zero}
In a similar spirit to Remark 
\ref{pointwise_lip_sand_thm_alt_phrasing_rmk}, 
and using the same notation as in Corollary 
\ref{lip_k_finite_subset_C0_corollary},
by taking $\vph \equiv 0$ we may conclude from Corollary
\ref{lip_k_finite_subset_C0_corollary} that if
$\psi = \left( \psi^{(0)} , \ldots , \psi^{(k)} \right) 
\in \Lip(\gamma,\Sigma,W)$ satisfies both that
$||\psi||_{\Lip(\gamma,\Sigma,W)} \leq K_1$ and, 
for every $j \in \{0, \ldots , k\}$ and every $x \in \Sigma_N$, 
that $\left|\left| \psi^{(j)}(x) 
\right|\right|_{\cl(V^{\otimes l};W)} \leq \ep_0$, 
then we have, for every $s \in \{0, \ldots , l\}$ and 
every point $x \in \Sigma$, that 
$\left|\left| \psi^{(s)}(x) \right|\right|_{\cl(V^{\otimes s};W)} 
\leq \ep$.
\end{remark}

\begin{remark}
\label{rmk:pointwise_behaviour_finite_subset}
When $N \in \Z_{\geq 1}$ defined in 
\eqref{lip_k_finite_subset_C0_corollary_cover_num_state} 
is less than the cardinality of $\Sigma$, 
Corollary \ref{lip_k_finite_subset_C0_corollary} 
guarantees that we are able to identify a strictly 
smaller collection of points $\Sigma_N$ such that the 
behaviour of a $\Lip(\gamma,\Sigma,W)$ function 
$F = \left( F^{(0)} , \ldots , F^{(k)}\right)$
on $\Sigma_N$ determines the pointwise behaviour of 
$F_{[l]} = \left( F^{(0)} , \ldots , F^{(l)} \right)$
over the entire set $\Sigma$ up to an arbitrarily small 
error. That is, using the notation of Corollary 
\ref{lip_k_finite_subset_C0_corollary},  
if $F \in \Lip(\gamma,\Sigma_N,W)$ and 
$\psi = \left(\psi^{(0)} ,\ldots ,\psi^{(k)}\right)$ and
$\vph = \left(\vph^{(0)} ,\ldots ,\vph^{(k)}\right)$ are 
both extensions of $F$ to $\Lip(\gamma,\Sigma,W)$, with 
$\Lip(\gamma,\Sigma,W)$-norms bounded above by $K_1$ and 
$K_2$ respectively,
then for every $s \in \{0, \ldots , l\}$ the functions
$\psi^{(l)}$ and $\vph^{(l)}$ may only differ, in the
pointwise sense, by at most $\ep$ throughout $\Sigma$.

Similarly to Remark \ref{rmk:finite_subset_determines_behaviour},
a particular consequence is that a function in 
$\Lip(\gamma,\Sigma,W)$ can be cost-effectively approximated 
in a pointwise sense.
That is, let
$\psi = \left( \psi^{(0)} , \ldots , \psi^{(k)} \right) \in 
\Lip(\gamma,\Sigma,W)$ with $||\psi||_{\Lip(\gamma,\Sigma,W)} 
\leq K_1$
and suppose, for some $l \in \{0, \ldots ,k\}$, 
that we want to approximate the functions 
$\psi^{(0)} , \ldots , \psi^{(l)}$ throughout $\Sigma$ in 
a pointwise sense.
Then Corollary \ref{lip_k_finite_subset_C0_corollary} guarantees 
that \textit{any} 
$\vph = \left( \vph^{(0)} , \ldots , \vph^{(k)} \right) 
\in \Lip(\gamma,\Sigma,W)$
with $||\vph||_{\Lip(\gamma,\Sigma,W)} \leq K_2$ will 
satisfy, for every $x \in \Sigma$ and every 
$s \in \{0 , \ldots , l \}$, that 
$\left|\left|\psi^{(s)}(x) - \vph^{(s)}(x) 
\right|\right|_{\cl(V^{\otimes s};W)} \leq \ep$ provided 
we have, for every point
$x \in \Sigma_N$ and every $j \in \{0, \ldots , k\}$, that 
$\left|\left| \psi^{(j)}(x) - \vph^{(j)}(x) 
\right|\right|_{\cl(V^{\otimes j};W)} \leq \ep_0$.
Thus the task of approximating the functions 
$\psi^{(0)} , \ldots , \psi^{(l)}$ throughout $\Sigma$ 
in a pointwise sense can be reduced to needing only 
to approximate $\psi$ in a pointwise sense at 
the finite number of points in the subset $\Sigma_N$.
\end{remark}

\begin{remark}
\label{rmk:explicit_example_two}
We illustrate the content of Remark 
\ref{rmk:pointwise_behaviour_finite_subset} 
via an explicit example. 
The explicit example is in the same setting considered in 
Remark \ref{rmk:explicit_example_1}.
Let $d \in \Z_{\geq 1}$, take
$V := \R^d$ equipped with its usual Euclidean norm 
$||\cdot||_2$, take $\Sigma := [0,1]^d \subset \R^d$ to be the 
unit cube in $\R^d$, and take $W := \R$.
Observe that the norm $||\cdot||_2$ is induced by the usual 
Euclidean dot product $\left< \cdot , \cdot \right>_{\R^d}$
on $\R^d$.
Equip the tensor powers of $\R^d$ with admissible norms 
in the sense of Definition \ref{admissible_tensor_norm} by
extending the inner product $\left< \cdot , \cdot \right>_{\R^d}$
to the tensor powers, and subsequently taking the norm induced
by the resulting inner product on the tensor powers 
(cf. Section \ref{sec:notation}). As introduced in Remark 
\ref{rmk:explicit_example_1}, we use the notation, 
for $x \in \R^d$ and $r > 0$, that 
$\B^d (x,r) := \left\{ y \in \R^d : ||x-y||_2 < r \right\}$.

Let $\gamma > 0$ with $k \in \Z_{\geq 0}$ 
such that $\gamma \in (k,k+1]$.
Consider fixed $K_0 , \ep > 0$, $l \in \{0 , \ldots , k\}$,
and $0 \leq \ep_0 < \min \{ 2K_0 , \ep \}$.
Retrieve the constant $\de_0 = \de_0(\ep,\ep_0, K_0, \gamma,l) > 0$
arising in Corollary \ref{lip_k_finite_subset_C0_corollary} 
for these choices of $\ep$, $\ep_0$, $\gamma$, and $l$, 
and for both the constants $K_1$ and $K_2$ there as $K_0$ here.
Let $N = N([0,1]^d, \ep, \ep_0, K_0 , \gamma, l) \in \Z_{\geq 0}$
denote the $\de_0$-covering number of $[0,1]^d$. That is, 
\beq
    \label{rmk_explicit_example_2_cover_num_state}
        N :=  
        \min \left\{ m \in \Z : 
            \text{ There exists } 
            x_1 , \ldots , x_m \in [0,1]^d
            \text{ such that } 
            [0,1]^d \subset \bigcup_{j=1}^m
            \ovB^d ( x_j , \de_0)
        \right\}.
\eeq
Following the method used in Remark \ref{rmk:explicit_example_1} 
to obtain \eqref{cover_num_upper_bound_rmk} verbatim enables us 
to conclude that
\beq
    \label{cover_num_upper_bound_rmk_2}
        N \leq \frac{2^d}{\omega_d} 
        \left( 1 + \frac{1}{\de_0} \right)^d
\eeq
where $\omega_d$ denotes the Euclidean volume of the 
unit ball $\B^d(0,1) \subset \R^d$.
Define $m \in \Z_{\geq 1}$ by 
\beq
    \label{int_m_def_rmk_2}
        m := \min \left\{ n \in \Z : 
        n \geq \frac{2^d}{\omega_d} 
        \left( 1 + \frac{1}{\de_0} \right)^d 
        \right\}.
\eeq
Choose distinct points $x_1, \ldots , x_m \in [0,1]^d$. 
Then, via \eqref{cover_num_upper_bound_rmk_2} and 
\eqref{int_m_def_rmk_2}, 
we see that $[0,1]^d \subset \cup_{j=1}^m \ovB^d(x_j,\de_0)$.
Let $\psi = \left( \psi^{(0)} , \ldots , \psi^{(k)} \right) 
\in \Lip(\gamma,[0,1]^d,\R)$ with 
$||\psi||_{\Lip(\gamma,[0,1]^d,\R)} \leq K_0$.
Then Corollary \ref{lip_k_finite_subset_C0_corollary} tells us 
that if $\vph = \left( \vph^{(0)} , \ldots , \vph^{(k)} 
\right) \in \Lip(\gamma,[0,1]^d,\R)$ satisfies both that 
$||\vph||_{\Lip(\gamma,[0,1]^d,\R)} \leq K_0$ and, for every 
$i \in \{1, \ldots , m\}$ and every $j \in \{0, \ldots , k\}$, 
that 
$\left|\left| \psi^{(j)}(x_i) - \vph^{(j)}(x_i) 
\right|\right|_{\cl((\R^d)^{\otimes j};\R)} \leq \ep_0$, 
then, for every $x \in [0,1]^d$ and every 
$s \in \{0, \ldots , l\}$, we have
$\left|\left|\psi^{(s)}(x) - \vph^{(s)}(x) 
\right|\right|_{\cl((\R^d)^{\otimes s};\R)} \leq \ep$.

Therefore, in order to approximate 
$\psi^{(0)} , \ldots , \psi^{(l)}$ up to an error 
of $\ep$ in a pointwise sense, we need only find 
$\vph = \left( \vph^{(0)} , \ldots , \vph^{(k)} 
\right) \in \Lip(\gamma,[0,1]^d,\R)$ satisfying both that 
$||\vph||_{\Lip(\gamma,[0,1]^d,\R)} \leq K_0$ and, for every 
$i \in \{1, \ldots , m\}$ and every $j \in \{0, \ldots , k\}$, 
that
\beq
    \label{eq:psi_approx_conds_2}
        \left|\left| \psi^{(j)}(x_i) - \vph^{(j)}(x_i) 
        \right|\right|_{\cl((\R^d)^{\otimes j};\R)} 
        \leq \ep_0.
\eeq
That is, up to an error of magnitude $\ep > 0$, the 
pointwise behaviour of the functions 
$\psi^{(0)} , \ldots , \psi^{(k)}$ throughout the entire
cube $[0,1]^d$ is captured by their pointwise values 
at the finite number of points 
$x_1 , \ldots , x_m \in [0,1]^d$, and we have the explicit 
upper bound resulting from \eqref{int_m_def_rmk_2} 
for the number of points $m$ that are required.
\end{remark}
\vskip 4pt
\noindent
We end this section with a short proof of Corollary 
\ref{lip_k_finite_subset_C0_corollary}.

\begin{proof}[Proof of Corollary 
\ref{lip_k_finite_subset_C0_corollary}]
Let $V$ and $W$ be Banach spaces, and assume  
that the tensor powers of 
$V$ are all equipped with admissible norms (cf. 
Definition \ref{admissible_tensor_norm}).
Assume that $\Sigma \subset V$ is compact.
Let $\ep, \gamma > 0$, with $k \in \Z_{\geq 0}$
such that $\gamma \in (k,k+1]$, 
$(K_1,K_2) \in \left( \R_{\geq 0}  \times \R_{\geq 0} \right) 
\setminus \{(0,0)\}$, and 
$0 \leq \ep_0 < \min \left\{ K_1+K_2 , \ep\right\}$.
Given $l \in \{0, \ldots , k\}$ let
$\de_0 = \de_0 (\ep , \ep_0, K_1+K_2, \gamma,l) > 0$
denote the constant arising from Theorem
\ref{lip_k_cover_C0_thm} for these choices
of $\ep$, $\ep_0$, $K_1$, $K_2$, $\gamma$, and $l$.
Note that we are not actually applying Theorem
\ref{lip_k_cover_C0_thm}, but simply 
retrieving a constant in preparation for its future
application. Define 
$N = N (\Sigma,\ep,\ep_0,K_1+K_2,\gamma,l) \in \Z_{\geq 0}$
to be the $\de_0$-covering number of $\Sigma$.
That is, 
\beq
    \label{lip_k_finite_subset_C0_corollary_cover_num}
        N := N_{\cov}(\Sigma,V,\de_0) =
        \min \left\{
        a \in \Z : 
            \text{ There exists } 
            x_1 , \ldots , x_a \in \Sigma 
            \text{ such that } 
            \Sigma \subset \bigcup_{j=1}^a
            \ovB_{V} ( x_j , \de_0)
        \right\}.
\eeq
The compactness of $\Sigma$ ensures that the integer $N$ 
defined in \eqref{lip_k_finite_subset_C0_corollary_cover_num}
is finite.
Let $z_1 , \ldots , z_N \in \Sigma$ be any 
collection of $N$ points in $\Sigma$ for which
\beq
    \label{lip_k_finite_subset_C0_corollary_cover_prop}
        \Sigma \subset \bigcup_{j=1}^N 
        \ovB_{V} ( z_j , \de_0 ).
\eeq
Set $\Sigma_N := \{ z_1 , \ldots , z_N \}$.

Now suppose that both
$\psi = \left(\psi^{(0)} ,\ldots ,\psi^{(k)}\right),
\vph = \left(\vph^{(0)} ,\ldots ,\vph^{(k)}\right) 
\in \Lip(\gamma,\Sigma,W)$ have their norms bounded
by $K_1$ and $K_2$ respectively, i.e.
$||\psi||_{\Lip(\gamma,\Sigma,W)} \leq K_1$ and 
$||\vph||_{\Lip(\gamma,\Sigma,W)} \leq K_2$.
Further suppose that for every $i \in \{0, \ldots , k\}$
and every $j \in \{1, \ldots , N\}$ the difference 
$\psi^{(i)}(z_j)-\vph^{(i)}(z_j) \in \cl(V^{\otimes i};W)$
satisfies the bound
\beq
    \label{lip_k_finite_subset_C0_corollary_assump_pf}
        \left|\left| \psi^{(i)}(z_j) - \vph^{(i)}(z_j) 
        \right|\right|_{\cl(V^{\otimes s};W)} 
        \leq \ep_0.
\eeq
Together, 
\eqref{lip_k_finite_subset_C0_corollary_cover_prop}
and 
\eqref{lip_k_finite_subset_C0_corollary_assump_pf}
provide the hypotheses required to allow us to
appeal to Theorem \ref{lip_k_cover_C0_thm} with the 
subset $B$ of that result as the subset $\Sigma_N$ here.
A consequence of doing so is that, for every 
$s \in \{0, \ldots , l\}$ and $x \in \Sigma$, 
we have that
$\left|\left| \psi^{(s)}(x) - \vph^{(s)}(x) 
\right|\right|_{\cl(V^{\otimes s};W)} \leq \ep$
as claimed in 
\eqref{lip_k_finite_subset_C0_corollary_conc}.
This completes the proof of Corollary
\ref{lip_k_finite_subset_C0_corollary}.
\end{proof}

\section{Remainder Term Estimates}
\label{lip_k_functions}
In this section we establish the following remainder term
estimates for a $\Lip(\gamma)$ function which will be 
particularly useful in subsequent sections.

\begin{lemma}[Remainder Term Estimates]
\label{lip_k_remainder_terms_pointwise_ests}
Let $V$ and $W$ be Banach spaces, and assume that the 
tensor powers of $V$ are all equipped with admissible 
norms (cf. Definition \ref{admissible_tensor_norm}).
Assume that $\Gamma \subset V$ is closed.
Let $\rho > 0$ with $n \in \Z_{\geq 0}$ such that
$\rho \in (n,n+1]$, let $\th \in (0, \rho)$, and 
suppose $\psi = ( \psi^{(0)} , 
\ldots , \psi^{(n)} ) \in \Lip(\rho,\Gamma,W)$.
For $l \in \{0, \ldots , n\}$ let 
$R^{\psi}_l : \Gamma \times \Gamma \to 
\cl ( V^{\otimes l} ; W)$ denote the remainder term 
associated to $\psi^{(l)}$ (cf. 
\eqref{lip_k_tay_expansion} in 
Definition \ref{lip_k_def}).
If $\th \in (n,\rho)$ then for every 
$l \in \{0, \ldots , n\}$ we have that for every 
$x,y \in \Gamma$ with $x \neq y$ that 
\beq
    \label{lip_k_remainder_bound_case_one}
        \frac{\left|\left| R^{\psi}_l(x,y) 
        \right|\right|_{\cl(V^{\otimes l};W)}}
        {||y-x||_V^{\th-l}}
        \leq
        \min \left\{ \Diam(\Gamma)^{\rho - \th} ~,~
        G(\rho,\th,l,\Gamma) \right\}
        || \psi ||_{\Lip(\rho,\Gamma,W)}
\eeq
where $G(\rho,\th,l,\Gamma)$ is defined by
\beq
    \label{lip_k_remainder_bound_case_one_W_def_state}
        G(\rho,\th,l,\Gamma) := 
        \inf_{r \in (0, \Diam(\Gamma))}
        \left\{
        \max \left\{ r^{\rho - \th} ~,~
        \frac{1 }{r^{\th-l}}
        \left( 1 + 
        \sum_{s=0}^{n-l} \frac{r^s}{s!} \right)
        \right\} \right\}.
\eeq
If $\th \in (0,n]$ (which is only possible if
$n \geq 1$) then let $q \in \{0, \ldots , n-1\}$ 
be such that $\th \in (q,q+1]$. For each 
$l \in \{0,\ldots,q\}$ let $\tilde{R}^{\psi}_l
: \Gamma \times \Gamma \to \cl(V^{\otimes l};W)$
denote the alteration of the remainder term 
$R^{\psi}_l$ defined for $x,y \in \Gamma$ and 
$v \in V^{\otimes l}$ by
\beq
    \label{lip_k_remainder_bound_alt_remainders}
        \tilde{R}^{\psi}_l(x,y)[v] :=
        R^{\psi}_l(x,y)[v] + 
        \sum_{s=q-l+1}^{n-l} \frac{1}{s!}
        \psi^{l+s}(x)\left[ v \otimes 
        (y-x)^{\otimes s}\right].
\eeq
Then for every $l \in \{0, \ldots , q\}$
and every $x,y \in \Gamma$ with $x \neq y$ we have that
\beq
    \label{lip_k_remainder_bound_case_two}
        \frac{\left|\left| \tilde{R}^{\psi}_l(x,y) 
        \right|\right|_{\cl(V^{\otimes l};W)}}
        {|| y - x ||_V^{\th - l}}
        \leq
        \min \left\{ \Diam(\Gamma)^{\rho - \th} 
        + \sum_{i=q+1}^n \frac{\Diam(\Gamma)^{i-\th} }
        {(i-l)!} ~,~ H(\rho,\th,l,\Gamma) \right\}
        || \psi ||_{\Lip(\rho,\Gamma,W)}
\eeq
where $H(\rho,\th,l,\Gamma)$ is defined by
\beq
    \label{lip_k_remain_bds_case_two_H_def_state}
        H(\rho,\th,l,\Gamma) :=
        \inf_{r \in (0, \Diam(\Gamma))}
        \left\{
        \max \left\{ r^{\rho - \th} + 
        \sum_{i=q+1}^n \frac{r^{i-\th}}{(i-l)!} ~,~
        \frac{1}{r^{\th-l}} \left( 1 + 
        \sum_{s=0}^{q-l} \frac{r^s}{s!} \right)
        \right\} \right\}.
\eeq
\end{lemma}

\begin{proof}[Proof of Lemma 
\ref{lip_k_remainder_terms_pointwise_ests}]
Let $V$ and $W$ be Banach spaces, and assume that the 
tensor powers of $V$ are all equipped with admissible 
norms (cf. Definition \ref{admissible_tensor_norm}).
Assume that $\Gamma \subset V$ is closed and that 
$\rho > 0$ with $n \in \Z_{\geq 0}$ such that
$\rho \in (n,n+1]$. Suppose that, for 
$l \in \{0 , \ldots , n \}$, we have functions 
$\psi^{(l)} : \Gamma \to \cl( V^{\otimes l} ; W)$
such that $\psi = (\psi^{(0)} ,\ldots ,\psi^{(n)})$
defines an element of $\Lip(\rho, \Gamma, W)$.
For each $l \in \{0, \ldots , n\}$ define
$R_l^{\psi} : \Gamma \times \Gamma \to 
\cl( V^{\otimes l};W)$ for 
$x,y \in \Gamma$ and $v \in V^{\otimes l}$ by
\beq    
    \label{lip_k_remainder_bounds_remainders}
        R^{\psi}_l(x,y)[v] :=
        \psi^{(l)}(y)[v] - \sum_{s=0}^{n-l}
        \frac{1}{s!}\psi^{(l+s)}(x)\left[ v \otimes 
        (y-x)^{\otimes s}\right].
\eeq
We claim that the estimates 
\eqref{lip_k_remainder_bound_case_one} and 
\eqref{lip_k_remainder_bound_case_two} are immediate 
when $||\psi||_{\Lip(\rho,\Gamma,W)} = 0$. 
To see this, note that if 
$||\psi||_{\Lip(\rho,\Gamma,W)} = 0$ then
for each $l \in \{ 0 , \ldots , n\}$ and any 
$x \in \Gamma$ we have that $\psi^{(l)}(x) \equiv 0$
in $\cl(V^{\otimes l};W)$. Conseuqently, for each
$l \in \{0, \ldots , n\}$ and any $x,y \in \Gamma$, 
we have via 
\eqref{lip_k_remainder_bounds_remainders} that  
$R_l^{\psi}(x,y) \equiv 0$ in $\cl(V^{\otimes l};W)$. 
When $\th \in (n,\rho)$, this tells us that the estimate 
\eqref{lip_k_remainder_bound_case_one} 
is true since both sides are zero.

When $ \th \in (0, n]$ (which is only possible if
$n \geq 1$) then, if $q \in \{0, \ldots , n-1\}$ is 
such that $\th \in (q,q+1]$, for each 
$l \in \{0, \ldots , q\}$ and any $x,y \in \Gamma$ we
have via \eqref{lip_k_remainder_bound_alt_remainders}
that the alteration $\tilde{R}_l^{\psi}$ of $R^{\psi}_l$
satisfies that $\tilde{R}_l^{\psi}(x,y) \equiv 0$ in 
$\cl(V^{\otimes l};W)$. Hence the estimate 
\eqref{lip_k_remainder_bound_case_two} is true 
since both sides are again zero.

If $||\psi||_{\Lip(\rho,\Gamma,W)} \neq 0$ then by 
replacing $\psi$ by 
$\psi / ||\psi||_{\Lip(\rho,\Gamma,W)}$ it suffices
to prove the estimates 
\eqref{lip_k_remainder_bound_case_one} and 
\eqref{lip_k_remainder_bound_case_two} under the 
additional assumption that 
$||\psi||_{\Lip(\rho,\Gamma,W)} = 1$.
As a consequence, 
whenever $l \in \{0, \ldots , n\}$ and $x,y \in \Gamma$,
we have the bounds (cf. \eqref{lip_k_bdd} and 
cf. \eqref{lip_k_remain_holder})
\beq
    \label{lip_k_remainder_bounds_bdd}
        (\bI) \quad \left|\left| \psi^{(l)}(x)
        \right|\right|_{\cl(V^{\otimes l}; W)}
        \leq 1 
        \qquad \text{and} \qquad
        (\bII) \quad
        \left|\left| R^{\psi}_l(x,y) 
        \right|\right|_{\cl(V^{\otimes l};W)}
        \leq 
        || y - x ||_V^{\rho - l}.
\eeq
First suppose $\th \in (n,\rho)$ and let 
$l \in \{0, \ldots , n\}$. 
For any $x,y \in \Gamma$ we have that 
\beq
    \label{lip_k_remain_bds_case_one_A}
        \left|\left| R^{\psi}_l(x,y) 
        \right|\right|_{\cl(V^{\otimes l};W)}
        \stackrel{ (\bII) \text{ of }
        \eqref{lip_k_remainder_bounds_bdd}
        }{\leq} 
        || y - x ||_V^{\rho - l}
        =
        || y - x ||_V^{\rho - \th}
        ||y-x||_V^{\th-l}.
\eeq
A first consequence of 
\eqref{lip_k_remain_bds_case_one_A} is 
\beq
    \label{lip_k_remain_bds_case_one_A1}
        \left|\left| R^{\psi}_l(x,y) 
        \right|\right|_{\cl(V^{\otimes l};W)}
        \leq
        \Diam(\Gamma)^{\rho - \th}
        ||y-x||_V^{\th-l}.
\eeq
A second consequence of 
\eqref{lip_k_remain_bds_case_one_A} is that,
for any fixed $r \in (0, \Diam(\Gamma))$,
if $||y-x||_V \leq r$ then 
\beq
    \label{lip_k_remain_bds_case_one_A2}
        \left|\left| R^{\psi}_l(x,y) 
        \right|\right|_{\cl(V^{\otimes l};W)}
        \leq
        r^{\rho - \th}
        ||y-x||_V^{\th-l}.
\eeq
If $||y-x||_V > r$ then we use
\eqref{lip_k_remainder_bounds_remainders} and that
the tensor powers of $V$ are equipped with admissible
norms (cf. Definition \ref{admissible_tensor_norm})
to compute for any $v \in V^{\otimes l}$ that
\begin{align*}
    \left|\left| R^{\psi}_l(x,y)[v] 
    \right|\right|_W 
    &\stackrel{
    \eqref{lip_k_remainder_bounds_remainders}
    }{\leq}
    \left|\left| \psi^{(l)}(y)[v] \right|\right|_W
    +
    \sum_{j=0}^{n-l} \frac{1}{j!}
    \left|\left| \psi^{(l+j)}(x) \left[ 
    v \otimes (y-x)^{\otimes j} \right|
    \right|\right|_W \\
    &\stackrel{ (\bI) \text{ of }
    \eqref{lip_k_remainder_bounds_bdd}
    }{\leq}
    ||v||_{V^{\otimes l}}
    +
    \sum_{j=0}^{n-l}\frac{1}{j!} ||y-x||^j_V
    ||v||_{V^{\otimes l}} 
    \leq 
    r^{-(\th - l)} \left( 1 + 
    \sum_{j=0}^{n-l} \frac{r^j}{j!} 
    \right)
    || y - x ||_V^{\th - l}
    ||v||_{V^{\otimes l}}.
\end{align*}
In the last line we have used that, 
for any $j \in \{0, \ldots , n-l\}$, 
that $||y-x||_V^{j-(\th-l)} < r^{j - (\th-l)}$.
This is itself a consequence of the facts that 
for any $j \in \{0, \ldots , n-l\}$ that 
$j - (\th - l) \leq n - \th < 0$, and that 
$r < || y - x ||_V$.
By taking the supremum over $v \in V^{\otimes l}$
with unit $V^{\otimes l}$ norm, we may conclude that
when $||y-x||_V > r$ we have
\beq
    \label{lip_k_remain_bds_case_one_A3}
        \left|\left| R^{\psi}_l(x,y) 
        \right|\right|_{\cl(V^{\otimes l};W)}
        \leq
        r^{-(\th-l)} \left( 1 + 
        \sum_{s=0}^{n-l} \frac{r^s}{s!} \right)
        ||y-x||_V^{\th-l}.
\eeq
By combining \eqref{lip_k_remain_bds_case_one_A2}
and \eqref{lip_k_remain_bds_case_one_A3} we deduce 
that for every $x,y \in \Gamma$ we have 
\beq
    \label{lip_k_remain_bds_case_one_done_ish}
        \left|\left| R^{\psi}_l(x,y) 
        \right|\right|_{\cl(V^{\otimes l};W)}
        \leq
        \max \left\{ r^{\rho - \th} ~,~
        \frac{1}{r^{\th-l}}
        \left( 1 + 
        \sum_{s=0}^{n-l} \frac{r^s}{s!} \right)
        \right\}
        ||y-x||_V^{\th-l}.
\eeq
Recall that the choice of $r \in (0, \Diam(\Gamma))$
was arbitrary. Consequently we may take the 
infimum over the choice of $r \in (0, \Diam(\Gamma))$
in \eqref{lip_k_remain_bds_case_one_done_ish} to 
obtain that whenever $x \neq y$ we have 
\beq
    \label{lip_k_remain_bds_case_one_doneA}
        \frac{\left|\left| R^{\psi}_l(x,y) 
        \right|\right|_{\cl(V^{\otimes l};W)}}
        {||y-x||_V^{\th-l}}
        \leq
        \inf_{r \in (0, \Diam(\Gamma))}
        \left\{
        \max \left\{ r^{\rho - \th} ~,~
        \frac{1 }{r^{\th-l}}
        \left( 1 + 
        \sum_{s=0}^{n-l} \frac{r^s}{s!} \right)
        \right\} \right\}.
\eeq
If we define
\beq
    \label{lip_k_remainder_bound_case_one_W_def}
        G(\rho,\th,l,\Gamma) := 
        \inf_{r \in (0, \Diam(\Gamma))}
        \left\{
        \max \left\{ r^{\rho - \th} ~,~
        \frac{1 }{r^{\th-l}}
        \left( 1 + 
        \sum_{s=0}^{n-l} \frac{r^s}{s!} \right)
        \right\} \right\},
\eeq
then \eqref{lip_k_remain_bds_case_one_A1} 
and \eqref{lip_k_remain_bds_case_one_doneA}
yield that 
\beq
    \label{lip_k_remainder_bound_case_one_doneB}
        \frac{\left|\left| R^{\psi}_l(x,y) 
        \right|\right|_{\cl(V^{\otimes l};W)}}
        {||y-x||_V^{\th-l}}
        \leq
        \min \left\{ \Diam(\Gamma)^{\rho - \th} ,
        G(\rho,\th,l,\Gamma) \right\}.
\eeq
The arbitrariness of $l \in \{0, \ldots , n\}$
and the points $x,y \in \Gamma$ with $x \neq y$
mean that \eqref{lip_k_remainder_bound_case_one_doneB}
establishes the estimates claimed in
\eqref{lip_k_remainder_bound_case_one} for the case
that $||\psi||_{\Lip(\rho,\Gamma,W)}=1$.

Now assume that $0 < \th \leq n < \rho \leq n +1$ 
which requires $n \geq 1$. Let 
$q \in \{0 , \ldots , n-1\}$ be such that 
$\th \in (q,q+1]$. For each 
$l \in \{0,\ldots,q\}$ let $\tilde{R}^{\psi}_l
: \Gamma \times \Gamma \to \cl(V^{\otimes l};W)$
denote the alteration of the remainder term 
$R^{\psi}_l$ defined for $x,y \in \Gamma$ and 
$v \in V^{\otimes l}$ by
\beq
    \label{lip_k_remainder_bound_alt_remainders_pf}
        \tilde{R}^{\psi}_l(x,y)[v] :=
        R^{\psi}_l(x,y)[v] + 
        \sum_{s=q-l+1}^{n-l} \frac{1}{s!}
        \psi^{(l+s)}(x)\left[ v \otimes 
        (y-x)^{\otimes s}\right].
\eeq
Let $l \in \{0, \ldots , q\}$, $x,y \in \Gamma$ and 
$v \in V^{\otimes l}$. Recalling that the 
tensor powers of $V$ are all equipped with admissible
norms (cf. Definition \ref{admissible_tensor_norm}), 
we may compute that 
\begin{align*}
    \left|\left| \tilde{R}^{\psi}_l(x,y)[v]
    \right|\right|_W 
    &\stackrel{
    \eqref{lip_k_remainder_bound_alt_remainders_pf}
    }{\leq}
    \left|\left|R^{\psi}_l(x,y)[v]
    \right|\right|_W + 
    \sum_{s=q-l+1}^{n-l} \frac{1}{s!}
    \left|\left| \psi^{(l+s)}(x)\left[ v \otimes 
    (y-x)^{\otimes s}\right]
    \right|\right|_W \\
    &\stackrel{
    \eqref{lip_k_remainder_bounds_bdd} 
    }{\leq}
    \left( 
    ||y-x||_V^{\rho - l}
    +
    \sum_{s=q-l+1}^{n-l} \frac{1}{s!} 
    ||y-x||_V^s \right)
    ||v||_{V^{\otimes l}} \\
    &=
    \left( 
    ||y-x||_V^{\rho - \th}
    +
    \sum_{i=q+1}^{n} \frac{1}{(i-l)!} 
    ||y-x||_V^{i-\th} \right)
    ||y-x||_V^{\th - l}
    ||v||_{V^{\otimes l}}.
\end{align*}
By taking the supremum over $v \in V^{\otimes l}$
with unit $V^{\otimes l}$ norm we may conclude that
\beq
    \label{lip_k_remainder_bounds_case_two_=<r}
        \left|\left| \tilde{R}^{\psi}_l(x,y)
        \right|\right|_{\cl(V^{\otimes l};W)}
        \leq
        \left( ||y-x||_V^{\rho - \th} +
        \sum_{i=q+1}^n \frac{||y-x||_V^{i-\th}}{(i-l)!}
        \right) 
        ||y-x||_V^{\th - l}.
\eeq
A first consequence of 
\eqref{lip_k_remainder_bounds_case_two_=<r} is that
\beq
    \label{lip_k_remainder_bounds_case_two_conc1}
        \left|\left| \tilde{R}^{\psi}_l(x,y)
        \right|\right|_{\cl(V^{\otimes l};W)}
        \leq
        \left( \Diam(\Gamma)^{\rho - \th} +
        \sum_{i=q+1}^n \frac{
        \Diam(\Gamma)^{i-\th}}{(i-l)!}
        \right) 
        ||y-x||_V^{\th - l}.
\eeq
Now consider a fixed choice of constant
$r \in (0, \Diam(\Gamma))$.
If $||y-x||_V \leq r$ then a consequence of
\eqref{lip_k_remainder_bounds_case_two_=<r} is that
\beq
    \label{lip_k_remainder_bounds_case_two_conc2}
        \left|\left| \tilde{R}^{\psi}_l(x,y)
        \right|\right|_{\cl(V^{\otimes l};W)}
        \leq
        \left( r^{\rho - \th} +
        \sum_{i=q+1}^n \frac{r^{i-\th}}{(i-l)!}
        \right) ||y-x||_V^{\th - l}.
\eeq
If $||y-x||_V > r$ then we may first observe via
\eqref{lip_k_remainder_bounds_remainders}
and \eqref{lip_k_remainder_bound_alt_remainders_pf}
that for any $v \in V^{\otimes l}$ we have
\beq
    \label{lip_k_remainder_bounds_alt_remain_expan}
        \tilde{R}^{\psi}_l(x,y)[v] =
        \psi^{(l)}(y)[v] - 
        \sum_{s=0}^{q-l} \frac{1}{s!}
        \psi^{(l+s)}(x) \left[ v \otimes
        (y-x)^{\otimes s} \right]. 
\eeq
We may use 
\eqref{lip_k_remainder_bounds_alt_remain_expan} and that
the tensor powers of $V$ are equipped with admissible
norms (cf. Definition \ref{admissible_tensor_norm})
to compute for any $v \in V^{\otimes l}$ that
\begin{align*}
    \left|\left| \tilde{R}^{\psi}_l(x,y)[v] 
    \right|\right|_W 
    &\stackrel{
    \eqref{lip_k_remainder_bounds_alt_remain_expan}
    }{\leq}
    \left|\left| \psi^{(l)}(y)[v] \right|\right|_W
    +
    \sum_{j=0}^{q-l} \frac{1}{j!}
    \left|\left| \psi^{(l+j)}(x) \left[ 
    v \otimes (y-x)^{\otimes j} \right|
    \right|\right|_W \\
    &\stackrel{
    \eqref{lip_k_remainder_bounds_bdd}
    }{\leq} ||v||_{V^{\otimes l}}
    +
    \sum_{j=0}^{q-l}\frac{1}{j!} ||y-x||^j_V
    ||v||_{V^{\otimes l}} 
    \leq 
    r^{-(\th - l)} \left( 1 + 
    \sum_{j=0}^{q-l} \frac{r^j}{j!} 
    \right)
    || y - x ||_V^{\th - l}
    ||v||_{V^{\otimes l}}.
\end{align*}
For the last inequality we have used that, 
for any $j \in \{0, \ldots , q-l\}$, 
that $||y-x||_V^{j-(\th-l)} < r^{j - (\th-l)}$.
This is itself a consequence of the facts that 
for any $j \in \{0, \ldots , q-l\}$ that 
$j - (\th - l) \leq q - \th < 0$, and that 
$r < || y - x ||_V$.
By taking the supremum over $v \in V^{\otimes l}$
with unit $V^{\otimes l}$ norm, we may conclude that
when $||y-x||_V > r$ we have
\beq
    \label{lip_k_remain_bds_case_two_>=r}
        \left|\left| \tilde{R}^{\psi}_l(x,y) 
        \right|\right|_{\cl(V^{\otimes l};W)}
        \leq
        r^{-(\th-l)} \left( 1 + 
        \sum_{s=0}^{q-l} \frac{r^s}{s!} \right)
        ||y-x||_V^{\th-l}.
\eeq
Together \eqref{lip_k_remainder_bounds_case_two_conc2}
and \eqref{lip_k_remain_bds_case_two_>=r} give 
that for every $x,y \in \Gamma$ with $x \neq y$ we have 
\beq
    \label{lip_k_remain_bds_case_two_done_ish}
        \frac{\left|\left| \tilde{R}^{\psi}_l(x,y) 
        \right|\right|_{\cl(V^{\otimes l};W)}}
        {||y-x||_V^{\th-l}}
        \leq
        \max \left\{ r^{\rho - \th} + 
        \sum_{i=q+1}^n \frac{r^{i-\th}}{(i-l)!} ~,~
        \frac{1}{r^{\th-l}} \left( 1 + 
        \sum_{s=0}^{q-l} \frac{r^s}{s!} \right)
        \right\}.
\eeq
Recall that the choice of $r \in (0, \Diam(\Gamma))$
was arbitrary. Consequently we may take the 
infimum over the choice of $r \in (0, \Diam(\Gamma))$
in \eqref{lip_k_remain_bds_case_two_done_ish} to 
obtain that whenever $x \neq y$ we have 
\beq
    \label{lip_k_remain_bds_case_two_doneA}
        \frac{\left|\left| \tilde{R}^{\psi}_l(x,y) 
        \right|\right|_{\cl(V^{\otimes l};W)}}
        {||y-x||_V^{\th-l}}
        \leq
        H(\rho,\th,l,\Gamma)
\eeq
for $H(\rho,\th,l,\Gamma)$ defined by
\beq
    \label{lip_k_remain_bds_case_two_H_def}
        H(\rho,\th,l,\Gamma) :=
        \inf_{r \in (0, \Diam(\Gamma))}
        \left\{
        \max \left\{ r^{\rho - \th} + 
        \sum_{i=q+1}^n \frac{r^{i-\th}}{(i-l)!} ~,~
        \frac{1}{r^{\th-l}} \left( 1 + 
        \sum_{s=0}^{q-l} \frac{r^s}{s!} \right)
        \right\} \right\}.
\eeq
Together \eqref{lip_k_remainder_bounds_case_two_conc1} 
and \eqref{lip_k_remain_bds_case_two_doneA} yield
\beq
    \label{lip_k_remains_bd_case_two_doneB}
        \frac{\left|\left| \tilde{R}^{\psi}_l(x,y) 
        \right|\right|_{\cl(V^{\otimes l};W)}}
        {||y-x||_V^{\th-l}}
        \leq
        \min \left\{ \Diam(\Gamma)^{\rho - \th} 
        + \sum_{i=q+1}^n \frac{\Diam(\Gamma)^{i-\th} }
        {(i-l)!} ~,~ H(\rho,\th,l,\Gamma) \right\}.
\eeq
The arbitrariness of $l \in \{0, \ldots , q\}$
and the points $x,y \in \Gamma$ with $x \neq y$
mean that \eqref{lip_k_remains_bd_case_two_doneB}
establishes the estimates claimed in 
\eqref{lip_k_remainder_bound_case_two} for the case
that $|| \psi ||_{\Lip(\rho,\Gamma,W)}=1$.
This completes the proof of Lemma 
\ref{lip_k_remainder_terms_pointwise_ests}.
\end{proof}

\section{Nested Embedding Property}
\label{nested_embed_sec}
In this section we establish that Lipschitz spaces 
are nested in the following sense. 
Let $V$ and $W$ be Banach spaces and assume that the tensor
powers of $V$ are equipped with admissible tensor norms 
(cf. Definition \ref{admissible_tensor_norm}).
Let $\rho \geq \th > 0$ and $\Gamma \subset V$ be a closed 
subset.
Then $\Lip(\rho,\Gamma,W) \subset \Lip(\th,\Gamma,W)$.
This nesting property is established by Stein in his original 
work \cite{Ste70}, whilst 
Theorem 1.18  in \cite{Bou15} provides a formulation in 
our particular framework.
To elaborate, if we let $\psi \in \Lip(\rho,\Gamma,W)$, 
$q \in \Z_{\geq 0}$ such that 
$\th \in (q,q+1]$, and 
$\psi_{[q]} = (\psi^{(0)},\ldots ,\psi^{(q)} )$,
then $\psi_{[q]} \in \Lip(\th,\Gamma,W)$. But
it is \textit{not} necessarily true that
$|| \psi_{[q]} ||_{\Lip(\th,\Gamma,W)} 
\leq || \psi ||_{\Lip(\rho,\Gamma,W)}$.

For example, consider $\Gamma := [-1,1] \subset \R$
and define functions $\psi^{(0)} : \Gamma \to \R$
and $\psi^{(1)} : \Gamma \to \cl(\R;\R)$ by 
$\psi^{(0)}(x) := x^2$ and $\psi^{(1)}(x)[v] := 2xv$ 
respectively. Let $\psi = (\psi^{(0)} , \psi^{(1)})$. 
Then the associated remainder terms
are $R_0^{\psi}(x,y) := \psi^{(0)}(y) - \psi^{(0)}(x) - 
\psi^{(1)}(x)[y-x] = (y-x)^2$ and 
$R^{\psi}_1(x,y)[v] 
:= \psi^{(1)}(y)[v] - \psi^{(1)}(x)[v]
= 2(y-x)v$.
It follows that $\psi \in \Lip(2,\Gamma,\R)$
with $|| \psi ||_{\Lip(2,\Gamma,\R)} = 2$.
However $R^{\psi}_1 (-1 ,1 )[v] = 4 v = 2\sqrt{2}
| 1 - (-1)|^{\frac{1}{2}} v$ 
and so $||\psi||_{\Lip(3/2,\Gamma,\R)} = 2\sqrt{2}
= \sqrt{2} || \psi ||_{\Lip(2,\Gamma,\R)}$.

In the following \textit{Lipschitz Nesting} Lemma 
\ref{lip_k_spaces_nested_lemma} we provide an 
explicit constant $C \geq 1$ for which the estimate 
$|| \cdot ||_{\Lip(\th,\Gamma,W)} \leq
C || \cdot ||_{\Lip(\rho,\Gamma,W)}$ holds.
The constant $C$ is more finely attuned to the geometry of the
domain $\Gamma$ than the corresponding constant in Theorem 
1.18 in \cite{Bou15}.

\begin{lemma}[Lipschitz Nesting]
\label{lip_k_spaces_nested_lemma}
Let $V$ and $W$ be Banach spaces, and assume that the 
tensor powers of $V$ are all equipped with admissible 
norms (cf. Definition \ref{admissible_tensor_norm}).
Assume that $\Gamma \subset V$ is closed.
Let $\rho > 0$ with $n \in \Z_{\geq 0}$ such that 
$\rho \in (n,n+1]$, and $\th \in (0, \rho)$ with 
$q \in \{0, \ldots , n\}$ such that $\th \in (q,q+1]$.
Suppose that $\psi = (\psi^{(0)} , \ldots , \psi^{(n)})
\in \Lip(\rho, \Gamma,W)$.
Then $\psi_{[q]} = (\psi^{(0)} , \ldots , \psi^{(q)})
\in \Lip(\th,\Gamma,W)$.
Further, if $\th \in (n , \rho)$ then 
we have the estimate that
\beq
    \label{lip_k_nested_norm_est_0}
        \left|\left| \psi_{[q]} 
        \right|\right|_{\Lip(\th,\Gamma,W)}
        \leq
        \max \left\{ 1 ~,~
        \min \left\{ 1 + e ~,~
        \Diam(\Gamma)^{\rho - \th} 
        \right\} \right\}
        || \psi ||_{\Lip(\rho,\Gamma,W)}.
\eeq
If $\th \in (0,n]$ then we have the 
estimate that
\beq
    \label{lip_k_nested_norm_est}
        \left|\left| \psi_{[q]} 
        \right|\right|_{\Lip(\th,\Gamma,W)}
        \leq
        \min \left\{ C_1 ~,~ C_2
        \right\} 
        || \psi ||_{\Lip(\rho, \Gamma,W)}
\eeq
where $C_1 , C_2 > 0$ are constants, depending
only on $\Diam(\Gamma)$, $\rho$, and $\th$,
defined by 
\beq
    \label{lip_k_nested_C1}
        C_1 := 
        \max \left\{ 1 ~,~
        \min \left\{ 1 + e ~,~
        \Diam(\Gamma)^{\rho - \th}
        +
        \sum_{j=q+1}^{n}
        \frac{\Diam(\Gamma)^{j-\th}}{(j-q)!}
        \right\} \right\}
\eeq
and
\beq
    \label{lip_k_nested_C2}
    		C_2 = 
    		\max \left\{ 1 ,
    		\min \left\{ 1 + e,
            \Diam(\Gamma)^{q+1-\th }
            \right\}\right\}
            \left( 1 +  \min \left\{e, 
            \Diam(\Gamma)^{\rho - n}
            \right\} \right)
            \left( 1 + \min \left\{
            e , \Diam(\Gamma) \right\}
            \right)^{n-(q+1)}.
\eeq
Finally, as a consequence of 
\eqref{lip_k_nested_norm_est_0} and 
\eqref{lip_k_nested_norm_est}, 
for any $\th \in (0, \rho)$ we have the estimate
\beq
    \label{lip_k_nested_norm_est_weak}
        \left|\left| \psi_{[q]} 
        \right|\right|_{\Lip(\th,\Gamma,W)}
        \leq 
        (1+e)||\psi||_{\Lip(\rho,\Gamma,W)}.
\eeq
\end{lemma}

\begin{remark}
The integers $n, q \in \Z_{\geq 0}$ are determined
by $\rho$ and $\th$ respectively. Consequently, 
any apparent dependence on $n$ and $q$ in 
\eqref{lip_k_nested_C1} and \eqref{lip_k_nested_C2} is 
really dependence on $\rho$ and $\th$ respectively.
\end{remark}

\begin{remark}
\label{lip_k_nested_equal_one_example}
We can have equality in \eqref{lip_k_nested_norm_est_0}. 
To see this, let $\Gamma := [-1,1] \subset \R$
and define $\psi^{(0)} : \Gamma \to \R$ by 
$\psi^{(0)}(x) := x^2$,
$\psi^{(1)} : \Gamma \to \cl(\R ; \R)$ by
$\psi^{(1)}(x)[v] := 2xv$,
$R_0(x,y) := \psi^{(0)}(y) -\psi^{(0)}(x) -\psi^{(1)}[y-x]
= (y-x)^2$ and
$R_1(x,y)[v] := \psi^{(1)}(y)[v] - \psi^{(1)}(x)[v]
= 2(y-x)v$.
Then $\psi = \left( \psi^{(0)} , \psi^{(1)} \right) 
\in \Lip(2, \Gamma, \R)$ with 
$|| \psi ||_{\Lip(2,\Gamma,\R)} = 2$.
However $R_1 (-1 ,1 )[v] = 4 v = 2\sqrt{2}
| 1 - (-1)|^{\frac{1}{2}} v$ and so 
$|| \psi ||_{\Lip(3/2 , \Gamma,\R)} = 2\sqrt{2}
= \sqrt{2} || \psi ||_{\Lip(2,\Gamma,\R)}$.
Here $\Diam(\Gamma)=2$, $\rho=2$
and $\th = 3/2$. Thus we observe that
$1 < \Diam(\Gamma)^{2-\frac{3}{2}} = 
\sqrt{2} < 1 + e$, which establishes equality in 
\eqref{lip_k_nested_norm_est_0}.
\end{remark}

\begin{remark}
\label{lip_k_nested_equal_two_example}
We can have equality in \eqref{lip_k_nested_norm_est}. 
As an example, let $\Gamma := \{0,1\} \subset \R$
and define $\psi^{(0)} : \Gamma \to \R$
by $\psi^{(0)}(0) := -A$ and $\psi^{(0)}(1) := A$ 
for some $A > 0$, and define $\psi^{(1)} : 
\Gamma \to \cl ( \R ; \R)$ by
$\psi^{(1)}(x)[v] := Av$ for every $x \in \Gamma$.
Then given $x,y \in \Gamma$ 
\beq
    \label{nest_est_sharp_1}
        \psi^{(0)}(y) - \psi^{(0)}(x) - \psi^{(1)}[y-x]
        = \threepartdef
        {A}{x=0,y=1}
        {-A}{x=1,y=0}
        {0}{x=y.} 
\eeq
It follows from \eqref{nest_est_sharp_1} that
$\psi = \left( \psi^{(0)} , \psi^{(1)} \right) 
\in \Lip(2,\Gamma,\R)$
with $|| \psi ||_{\Lip(2,\Gamma,\R)}=A$.
Moreover, we also have that
$\psi_{[0]} = \psi^{(0)} \in \Lip(1,\Gamma,\R)$ with 
$\left|\left| \psi_{[0]} 
\right|\right|_{\Lip(1,\Gamma,\R)}= 2A = 2
|| \psi ||_{\Lip(2,\Gamma,\R)}$.
Here $\Diam(\Gamma)=1$,
$n=1$, $\rho = 2$, $\th =1$ and $q = 0$.
Consequently, both $C_1$ defined in 
\eqref{lip_k_nested_C1} and $C_2$ defined in
\eqref{lip_k_nested_C2} are equal to $2$. 
Hence $\min \left\{ C_1 , C_2 \right\} = 2$,
and so we have equality in \eqref{lip_k_nested_norm_est}.
\end{remark}

\begin{proof}[Proof of Lemma 
\ref{lip_k_spaces_nested_lemma}]
Let $V$ and $W$ be Banach spaces, and assume that 
the tensor powers of $V$ are all equipped with 
admissible norms 
(cf. Definition \ref{admissible_tensor_norm}).
Assume that $\Gamma \subset V$ is closed.
Let $\rho > 0$ with $n \in \Z_{\geq 0}$ such that
$\rho \in (n,n+1]$, and let $\th \in (0,\rho)$
with $q \in \{0, \ldots , n\}$ such that
$\th \in (q,q+1]$.
To deal with the case that $\th \in (n, \rho)$, 
we first establish the following claim.

\begin{claim}
\label{lip_k_nested_claim_1}
Suppose $V$ and $W$ are Banach spaces, and that the tensor 
powers of $V$ are all equipped with admissible norms 
(cf. Definition \ref{admissible_tensor_norm}).
Assume that $\cd \subset V$ is closed.
Let $\lambda > 0$ with $m \in \Z_{\geq 0}$ such that 
$\lambda \in (m,m+1]$, and $\sigma \in (m,\lambda)$. 
If $\phi = (\phi^{(0)} , \ldots , \phi^{(m)}) 
\in \Lip(\lambda, \cd,W)$ then 
$\phi \in \Lip(\sigma, \cd,W)$, and we have the estimate
that
\beq    
    \label{lip_k_nested_claim_1_conc}
        || \phi ||_{\Lip(\sigma, \cd,W)}
        \leq
        \max \left\{ 1 ~,~ 
        \min \left\{ 1 + e ~,~
        \Diam(\cd)^{\lambda - \sigma} 
        \right\} \right\}
        || \phi ||_{\Lip(\lambda,\cd,W)}.
\eeq
\end{claim}

\begin{proof}[Proof of Claim \ref{lip_k_nested_claim_1}]
For each $l \in \{0, \ldots , m\}$ define
$R_l^{\phi} : \cd \times \cd \to 
\cl( V^{\otimes l};W)$ for 
$x,y \in \cd$ and $v \in V^{\otimes l}$ by
\beq    
    \label{lip_k_nested_remainders}
        R^{\phi}_l(x,y)[v] :=
        \phi^{(l)}(y)[v] - \sum_{s=0}^{n-l}
        \frac{1}{s!}\phi^{l+s}(x)\left[ v \otimes 
        (y-x)^{\otimes s}\right].
\eeq
Since the estimate \eqref{lip_k_nested_claim_1_conc} 
is trivial when $||\phi||_{\Lip(\lambda,\cd,W)}=0$,
we need only establish the validity of 
\eqref{lip_k_nested_claim_1_conc} when 
$||\phi||_{\Lip(\lambda,\cd,W)} \neq 0$. But in this 
case, by replacing $\phi$ by $\phi / 
|| \phi ||_{\Lip(\lambda,\cd,W)}$ it suffices to prove 
\eqref{lip_k_nested_claim_1_conc} under the 
additional assumption that 
$|| \phi ||_{\Lip(\lambda,\cd,W)} = 1$

A consequence of $\phi \in \Lip(\lambda,\cd,W)$
with $|| \phi ||_{\Lip(\lambda,\cd,W)} = 1$ is that, 
whenever $l \in \{0, \ldots , m\}$ and $x,y \in \cd$,
we have the bounds (cf. \eqref{lip_k_bdd} and 
\eqref{lip_k_remain_holder})
\beq
    \label{lip_k_nested_bdd}
        (\bI) \quad \left|\left| \phi^{(l)}(x)
        \right|\right|_{\cl(V^{\otimes l}; W)}
        \leq 1
        \qquad \text{and} \qquad
        (\bII) \quad 
        \left|\left| R^{\phi}_l(x,y) 
        \right|\right|_{\cl(V^{\otimes l};W)}
        \leq || y - x ||_V^{\lambda - l}.
\eeq
Given any $l \in \{0, \ldots , m\}$ and any point 
$x \in \cd$, we can conclude from 
\eqref{lip_k_nested_remainders} that 
$R^{\psi}_l(x,x) \equiv 0$ in $\cl(V^{\otimes l} ; W)$.
Hence controlling the $\cl(V^{\otimes l};W)$ norm of the 
remainder term $R^{\psi}_l$ is trivial on the 
diagonal of $\cd \times \cd$.

Given any $l \in \{0, \ldots , m\}$, we now estimate 
the $\cl(V^{\otimes l};W)$ norm of $R^{\psi}_l$ off the 
diagonal of $\cd \times \cd$. For any points $x,y \in \cd$
with $x \neq y$, we apply Lemma 
\ref{lip_k_remainder_terms_pointwise_ests}, recalling 
that $||\phi||_{\Lip(\lambda,\cd,W)}=1$, to conclude 
that (cf. \eqref{lip_k_remainder_bound_case_one})
\beq
    \label{lip_k_nested_case_one_remainders}
        \frac{\left|\left| R^{\phi}_l(x,y) 
        \right|\right|_{\cl(V^{\otimes l};W)}}
        {||y-x||_V^{\sigma-l}}
        \leq
        \min \left\{ \Diam(\cd)^{\lambda - \sigma} ~,~
        G(\lambda, \sigma, l, \cd) \right\}
\eeq
where $G(\lambda,\sigma,l,\cd)$ is defined by
\beq
    \label{lip_k_nested_W_def}
        G(\lambda,\sigma,l,\cd) := 
        \inf_{r \in (0, \Diam(\cd))}
        \left\{
        \max \left\{ r^{\lambda - \sigma} ~,~
        \frac{1 }{r^{\sigma-l}}
        \left( 1 + 
        \sum_{s=0}^{m-l} \frac{r^s}{s!} \right)
        \right\} \right\}.
\eeq
We now prove that
\beq
    \label{lip_k_nested_claim_1_1+e_bd}
        \min \left\{ \Diam(\cd)^{\lambda - \sigma} ,
        G(\lambda,\sigma,l,\cd) \right\}
        \leq
        1+e.
\eeq
If $\Diam(\cd) \leq 1$, then  
\eqref{lip_k_nested_claim_1_1+e_bd} is obtained 
by observing that
$$\min \left\{ \Diam(\cd)^{\lambda - \sigma} ,
G(\lambda,\sigma,l,\cd) \right\} \leq 
\Diam(\cd)^{\lambda - \sigma} \leq 1 < 1 + e.$$
If $\Diam(\cd) > 1$ then 
\eqref{lip_k_nested_claim_1_1+e_bd} is obtained 
by observing that
$$\min \left\{ \Diam(\cd)^{\lambda - \sigma} ,
    G(\lambda,\sigma,l,\cd) \right\}
    \leq
    G(\lambda,\sigma,l,\cd)
    \leq
    \max \left\{ 1 ~,~ 1 + \sum_{s=0}^{m-l}
    \frac{1}{s!} \right\}
    \leq (1+e).$$
Hence \eqref{lip_k_nested_claim_1_1+e_bd} is
proven. Together 
\eqref{lip_k_nested_case_one_remainders} and 
\eqref{lip_k_nested_claim_1_1+e_bd} yield that
\beq
    \label{lip_k_nested_case_one_remain_bd_3}
        \frac{\left|\left| R^{\phi}_l(x,y) 
        \right|\right|_{\cl(V^{\otimes l};W)}}
        {||y-x||_V^{\sigma-l}}
        \leq 1+e
\eeq
Thus we may combine 
\eqref{lip_k_nested_case_one_remainders} and 
\eqref{lip_k_nested_case_one_remain_bd_3} to 
conclude that
\beq
    \label{lip_k_nested_case_one_remain_bd_5}
        \frac{\left|\left| R^{\phi}_l(x,y) 
        \right|\right|_{\cl(V^{\otimes l};W)}}
        {||y-x||_V^{\sigma-l}}
        \leq
        \min \left\{ \Diam(\cd)^{\lambda - \sigma}
        ~,~ 1 + e \right\}.
\eeq
Both the choice of $l \in \{0, \ldots , m\}$
and the choice of points $x,y \in \cd$ with $x \neq y$ 
were arbitrary. Hence we may conclude that the 
estimate \eqref{lip_k_nested_case_one_remain_bd_5} 
is valid for every $l \in \{0, \ldots , m\}$
and all points $x,y \in \cd$ with $x \neq y$.
The pointwise bounds for the functions
$\phi^{(0)} , \ldots , \phi^{(m)}$ given in (\bI) of
\eqref{lip_k_nested_bdd} and the remainder term 
bounds \eqref{lip_k_nested_case_one_remain_bd_5}
establish that
\beq
    \label{lip_k_nested_case_one_done}
        \left|\left| \phi
        \right|\right|_{\Lip(\sigma,\cd,W)}
        \leq
        \max \left\{ 1 ~,~ 
        \min \left\{ \Diam(\cd)^{\lambda -\sigma} 
        ~,~ 1 + e \right\} \right\}. 
\eeq
The estimate \eqref{lip_k_nested_case_one_done} 
is precisely the estimate 
claimed in \eqref{lip_k_nested_claim_1_conc} for the 
case that $||\phi||_{\Lip(\lambda,\cd,W)}=1$.
This completes the proof of Claim 
\ref{lip_k_nested_claim_1}.
\end{proof}
\vskip 4pt
\noindent
The estimate claimed in the case that 
$\th \in (n,\rho)$ is an immediate consequence of 
Claim \ref{lip_k_nested_claim_1}.
Indeed, assuming that $\th \in (n,\rho)$, we
appeal to Claim \ref{lip_k_nested_claim_1}
with $\cd := \Gamma$, $m := n$, $\lambda := \rho$
and $\sigma := \th$ to conclude from 
\eqref{lip_k_nested_claim_1_conc} that
\beq
    \label{lip_k_nested_eta_norm_est_0}
        || \psi ||_{\Lip(\th,\Gamma,W)}
        \leq
        \max \left\{ 1 ~,~
        \min \left\{ 1 + e ~,~
        \Diam(\Gamma)^{\rho - \th} 
        \right\} \right\}
        || \psi ||_{\Lip(\rho,\Gamma,W)},
\eeq
which is precisely the estimate claimed in
\eqref{lip_k_nested_norm_est_0}.

It remains only to establish the estimate
claimed in \eqref{lip_k_nested_norm_est} 
for the case that $0 < \th \leq n < \rho \leq n+1$.
Observe that this requires $n \geq 1$ and 
$q \in \{0 , \ldots , n-1\}$.
We begin by establishing the following claim.

\begin{claim}
\label{lip_k_nested_claim_2}
Suppose $V$ and $W$ are Banach spaces, and that the tensor 
powers of $V$ are all equipped with admissible norms 
(cf. Definition \ref{admissible_tensor_norm}).
Assume that $\cd \subset V$ is closed.
Let $\lambda > 1$ with $m \in \Z_{\geq 1}$ such that
$\lambda \in (m,m+1]$, and $\sigma \in (0,m]$ with 
$p \in \{0 , \ldots , m-1\}$ such that
$\sigma \in (p,p+1]$.
If $\phi = (\phi^{(0)} , \ldots , \phi^{(m)}) 
\in \Lip(\lambda, \cd,W)$ then 
$\phi_{[p]} = (\phi^{(0)} , \ldots , \phi^{(p)}) 
\in \Lip(\sigma, \cd,W)$, and we have the estimate
that
\beq
    \label{lip_k_nested_claim_2_conc}
        \left|\left| \phi_{[p]} 
        \right|\right|_{\Lip(\sigma, \cd,W)}
        \leq
        \max \left\{ 1 ~,~
        \min \left\{ 1+e ~,~
        \Diam(\cd)^{\lambda - \sigma}
        +
        \sum_{j=p+1}^{m}
        \frac{\Diam(\cd)^{j-\sigma}}{(j-p)!}
        \right\}
        \right\}
        || \phi ||_{\Lip(\lambda, \cd,W)}.
\eeq
\end{claim}

\begin{proof}[Proof of Claim 
\ref{lip_k_nested_claim_2}]
For each $l \in \{0, \ldots , m\}$ define 
$R^{\phi}_l : \cd \times \cd \to \cl (V^{\otimes l} ;W)$
for $x,y \in \cd$ and $v \in V^{\otimes l}$ by 
\beq
    \label{lip_k_nested_lemma_phi_tay_exp_2}
        R^{\phi}_l(x,y)[v] :=
        \phi^{(l)}(y)[v] -
        \sum_{j=0}^{m-l} \frac{1}{j!}
        \phi^{(j + l)}(x) \left[ 
        v \otimes (y - x)^{\otimes j} 
        \right].
\eeq
Since the estimate \eqref{lip_k_nested_claim_2_conc} 
is trivial when $||\phi||_{\Lip(\lambda,\cd,W)}=0$,
we need only establish the validity of 
\eqref{lip_k_nested_claim_2_conc} when 
$||\phi||_{\Lip(\lambda,\cd,W)} \neq 0$. But in this 
case, by replacing $\phi$ by $\phi / 
|| \phi ||_{\Lip(\lambda,\cd,W)}$ it suffices to prove 
\eqref{lip_k_nested_claim_2_conc} under the 
additional assumption that 
$|| \phi ||_{\Lip(\lambda,\cd,W)} = 1$

A consequence of $\phi \in \Lip(\lambda,\cd,W)$
with $|| \phi ||_{\Lip(\lambda,\cd,W)} = 1$ is that,
whenever $l \in \{0, \ldots , m\}$ and $x,y \in \cd$, 
we have the bounds (cf. \eqref{lip_k_bdd} and 
\eqref{lip_k_remain_holder})
\beq
    \label{lip_k_nested_lemma_phi_l_bound_2}
        (\bI) \quad \left|\left| \phi^{(l)}(x) 
        \right|\right|_{\cl ( V^{\otimes l} ;W)}
        \leq 1
        \qquad \text{and} \qquad
        (\bII) \quad 
        \left|\left| R^{\phi}_l(x,y) 
        \right|\right|_{\cl ( V^{\otimes l} ; W)}
        \leq || y - x ||_V^{\lambda - l}.
\eeq
Our goal is to show that $\phi_{[p]} = 
(\phi^{(0)} , \ldots , \phi^{(p)})$ is in 
$\Lip(\sigma,\cd,W)$.
For this purpose, given $s \in \{0, \ldots , p\}$,
let $\tilde{R}^{\phi}_s : \cd \times \cd \to
\cl( V^{\otimes s} ; W)$ be defined for all
$x,y \in \cd$ and $v \in V^{\otimes s}$ by
\beq
    \label{lip_k_nested_lemma_psi_new_tay_exp}
        \tilde{R}^{\phi}_s(x,y)[v] := 
        \phi^{(s)}(y)[v] -
        \sum_{j=0}^{p-s} \frac{1}{j!}
        \phi^{(s+j)}(x) \left[ 
        v \otimes (y - x)^{\otimes j} 
        \right].
\eeq
Together, \eqref{lip_k_nested_lemma_phi_tay_exp_2} and 
\eqref{lip_k_nested_lemma_psi_new_tay_exp} yield that 
\beq
    \label{lip_k_nested_new_remainders}
        \tilde{R}^{\phi}_s(x,y)[v] =
        R^{\phi}_s(x,y)[v] + 
        \sum_{j=p+1-s}^{m-s} \frac{1}{(k-j)!}
        \phi^{(s+j)}(x) 
        \left[ v \otimes (y-x)^{\otimes j}
        \right].
\eeq
Given any $l \in \{0 , \ldots , p\}$ and any point
$x \in \cd$, we conclude from 
\eqref{lip_k_nested_new_remainders} that
$\tilde{R}^{\psi}_l(x,x) \equiv 0$ in 
$\cl(V^{\otimes l};W)$. Hence controlling the 
$\cl(V^{\otimes l};W)$ norm of $\tilde{R}^{\psi}_l$
is trivial on the diagonal of $\cd \times \cd$.

Given any $l \in \{0, \ldots , p\}$ we now 
estimate the $\cl(V^{\otimes l};W)$ norm of 
$\tilde{R}^{\psi}_l$ off of the diagonal of 
$\cd \times \cd$. The alteration in 
\eqref{lip_k_nested_new_remainders} is exactly the
same as the alteration defined in 
\eqref{lip_k_remainder_bound_alt_remainders} of
Lemma \ref{lip_k_remainder_terms_pointwise_ests}.
Consequently, recalling that 
$||\phi||_{\Lip(\lambda,\cd,W)} =1$, 
we may apply that result 
(Lemma \ref{lip_k_remainder_terms_pointwise_ests})
to deduce that for 
every $x,y \in \cd$ with $x \neq y$ we have that
(cf. \eqref{lip_k_remainder_bound_case_two})
\beq
    \label{lip_k_nested_claim_2_remain_ests}
        \frac{\left|\left| \tilde{R}^{\phi}_l(x,y) 
        \right|\right|_{\cl(V^{\otimes l};W)}}
        {|| y - x ||_V^{\sigma - l}}
        \leq
        \min \left\{ \Diam(\cd)^{\lambda - \sigma} 
        + \sum_{i=p+1}^m \frac{\Diam(\cd)^{i-\sigma} }
        {(i-l)!} ~,~ H(\lambda,\sigma,l,\cd) \right\}
\eeq
where $H(\lambda,\sigma,l,\cd)$ is defined by
(cf. \eqref{lip_k_remain_bds_case_two_H_def_state})
\beq
    \label{lip_k_nested_claim_2_H_def}
        H(\lambda,\sigma,l,\cd) :=
        \inf_{r \in (0, \Diam(\cd))}
        \left\{
        \max \left\{ r^{\lambda - \sigma} + 
        \sum_{i=p+1}^m \frac{r^{i-\sigma}}{(i-l)!} ~,~
        \frac{1}{r^{\sigma-l}} \left( 1 + 
        \sum_{s=0}^{p-l} \frac{r^s}{s!} \right)
        \right\} \right\}.
\eeq
We now prove that
\beq
    \label{lip_k_nested_claim_2_1+e_bd}
        \cH := 
        \min \left\{ \Diam(\cd)^{\lambda - \sigma} 
        + \sum_{i=p+1}^m \frac{\Diam(\cd)^{i-\sigma} }
        {(i-l)!} ~,~ H(\lambda,\sigma,l,\cd) \right\}
        \leq 
        1 + e.
\eeq
If $\Diam(\cd) \leq 1$ then we obtain 
\eqref{lip_k_nested_claim_2_1+e_bd} by observing,
for every $i \in \{p+1, \ldots , m\}$, that 
$(i-l)! \geq (i-p)!$ and hence
$$\cH \leq 
\Diam(\cd)^{\lambda - \sigma} 
+ \sum_{i=p+1}^m \frac{\Diam(\cd)^{i-\sigma} }
{(i-l)!} \leq 1 + \sum_{i=p+1}^m\frac{1}{(i-p)!}
< 1 + e.$$
If $\Diam(\cd) > 1$ then we obtain 
\eqref{lip_k_nested_claim_2_1+e_bd} by observing that
$$\cH \leq 
H(\lambda,\sigma,l,\cd) \leq
\max \left\{ 1 + \sum_{i=p+1}^m \frac{1}{(i-l)!}
, 1 + \sum_{s=0}^{p-l} \frac{1}{s!} \right\} 
\leq 
\max \left\{ 1 + \sum_{i=p+1}^m \frac{1}{(i-p)!}
, 1 + \sum_{s=0}^{p-l} \frac{1}{s!} \right\} 
< 1 + e.$$
Hence \eqref{lip_k_nested_claim_2_1+e_bd} is proven.
Together \eqref{lip_k_nested_claim_2_remain_ests} and 
\eqref{lip_k_nested_claim_2_1+e_bd} establish that
\beq
    \label{lip_k_nested_claim_2_remain_ests_B}
        \frac{\left|\left| \tilde{R}^{\phi}_l(x,y) 
        \right|\right|_{\cl(V^{\otimes l};W)}}
        {|| y - x ||_V^{\sigma - l}}
        \leq 1+e
\eeq
Thus we may combine 
\eqref{lip_k_nested_claim_2_remain_ests},  
\eqref{lip_k_nested_claim_2_remain_ests_B}, 
and the observation that for every 
$i \in \{p+1, \ldots ,m\}$ we have $(i-l)! \geq (i-p)!$ 
to conclude that
\beq
    \label{lip_k_nested_claim_2_remain_bd_5}
        \frac{\left|\left| \tilde{R}^{\phi}_l(x,y) 
        \right|\right|_{\cl(V^{\otimes l};W)}}
        {|| y - x ||_V^{\sigma - l}}
        \leq
        \min \left\{ 
        \Diam(\cd)^{\lambda - \sigma}
        + \sum_{i=p+1}^m \frac{\Diam(\cd)^{i-\sigma}}
        {(i-p)!}
        ~,~ 1+e\right\}.
\eeq
The arbitrariness of $l \in \{0, \ldots , p\}$
and the points $x,y \in \cd$ with $x \neq y$ ensure
that the estimate 
\eqref{lip_k_nested_claim_2_remain_bd_5} is valid
for every $l \in \{0, \ldots , p\}$ and every
$x,y \in \cd$ with $x \neq y$.
Together, the definitions 
\eqref{lip_k_nested_lemma_psi_new_tay_exp},
the bounds in (\bI) of 
\eqref{lip_k_nested_lemma_phi_l_bound_2}, and the
estimates 
\eqref{lip_k_nested_claim_2_remain_bd_5} allow
us to conclude that
$\phi_{[p]} = (\phi^{(0)} , \ldots , \phi^{(p)}) 
\in \Lip(\sigma, \cd,W)$, and that
\beq
    \label{lip_k_nested_lip_eta_norm_bound}
        \left|\left| \phi_{[p]} 
        \right|\right|_{\Lip(\sigma, \cd,W)}
        \leq
        \max \left\{ 1 ~,~
        \min \left\{ 1 + e ~,~
        \Diam(\cd)^{\lambda - \sigma}
        +
        \sum_{j=p+1}^{m}
        \frac{\Diam(\cd)^{j-\sigma}}{(j-p)!}
        \right\} \right\}.
\eeq
The estimate \eqref{lip_k_nested_lip_eta_norm_bound}
is precisely the estimate claimed in 
\eqref{lip_k_nested_claim_2_conc} for the case that
$||\phi||_{\Lip(\lambda,\cd,W)}=1$.
This completes the proof of Claim 
\ref{lip_k_nested_claim_2}.
\end{proof}
\vskip 4pt
\noindent
Returning to the proof of Lemma 
\ref{lip_k_spaces_nested_lemma} itself, 
suppose $\th \in (0,n]$.
A direct application of Claim \ref{lip_k_nested_claim_2}
with $\cd := \Gamma$, $m:=n$, $\lambda := \rho$ and
$\sigma := \th$ means that 
\eqref{lip_k_nested_claim_2_conc} yields that
\beq
    \label{lip_k_nested_lip_eta_norm_1&2}
        \left|\left| \psi_{[q]} 
        \right|\right|_{\Lip(\th, \Gamma,W)}
        \leq
        \max \left\{ 1 ~,~
        \min \left\{ 1 + e ~,~
        \Diam(\Gamma)^{\rho - \th}
        +
        \sum_{j=q+1}^{n}
        \frac{\Diam(\Gamma)^{j-\th}}{(j-q)!}
        \right\} \right\}
        || \psi ||_{\Lip(\rho, \Gamma,W)}
\eeq
where $q \leq n - 1$ since $\th \in (0,n]$.
By examining the definition of $C_1$ in 
\eqref{lip_k_nested_C1}, we see that 
\eqref{lip_k_nested_lip_eta_norm_1&2}
is the first part of the estimate claimed 
in \eqref{lip_k_nested_norm_est}.
To derive the remaining estimate we 
note that $\th \in (q,q+1]$.  
By appealing to Claim \ref{lip_k_nested_claim_2}, with 
$\cd := \Gamma$, $m := n$, $\lambda := \rho$ and
$\sigma := n$, we deduce via 
\eqref{lip_k_nested_claim_2_conc} that
\beq
    \label{lip_k_nested_lip_k_norm_bound}
        \left|\left| \psi_{[n-1]} 
        \right|\right|_{\Lip(n, \Gamma,W)}
        \leq
        \min \left\{ 1 + e ~,~
        1+\Diam(\Gamma)^{\rho - n}
        \right\} 
        || \psi ||_{\Lip(\rho, \Gamma,W)}.
\eeq
If we now appeal to Claim \ref{lip_k_nested_claim_2} for 
$\cd := \Gamma$, $m:=n-1$, $\lambda := n$ and
$\sigma := n-1$, then \eqref{lip_k_nested_claim_2_conc} and 
\eqref{lip_k_nested_lip_k_norm_bound} give
\begin{align*}
    \left|\left| \psi_{[n-2]} 
    \right|\right|_{\Lip(n-1, \Gamma,W)}
        &\stackrel{
        \eqref{lip_k_nested_claim_2_conc}
        }{\leq}
        \max \left\{ 1 ~,~
        \min \left\{ 1 + e ~,~
        1+ \Diam(\Gamma) \right\}
        \right\}
        \left|\left| \psi_{[n-1]} 
        \right|\right|_{\Lip(n, \Gamma,W)} \\
        &\stackrel{
        \eqref{lip_k_nested_lip_k_norm_bound}
        }{\leq}
        \min \left\{ 1 + e ~,~
        1+\Diam(\Gamma)^{\rho - n}
        \right\} 
        \min \left\{ 1 + e ~,~
        1 + \Diam(\Gamma) \right\}
        || \psi ||_{\Lip(\rho, \Gamma,W)} \\
        &= 
        \left( 1 + \min \left\{ e ~,~
        \Diam(\Gamma)^{\rho - n}
        \right\}\right) \left( 1 + 
        \min \left\{ e ~,~
        \Diam(\Gamma) \right\} \right)
        || \psi ||_{\Lip(\rho, \Gamma,W)}.
\end{align*}
We can now appeal to Claim \ref{lip_k_nested_claim_2}
for $\cd := \Gamma$, $m := n-2$, $\lambda := n -1$, 
and $\sigma := n-2$.
Proceeding inductively as $r = 0,1, \ldots ,n-1$
increases, we establish via applying Claim  
\ref{lip_k_nested_claim_2} for 
$\cd := \Gamma$, $m:= n-r$, $\lambda := n-(r-1)$, and 
$\sigma := n-r$ that for every
$r \in \{0, 1 , \ldots , n-1\}$ we have that
\beq
    \label{lip_k_nested_lip_k-r_norm_bound}
        \left|\left| \psi_{[n-r-1]} 
        \right|\right|_{\Lip(n-r, \Gamma,W)}
        \leq
        \left( 1 + \min \left\{ e ~,~
        \Diam(\Gamma)^{\rho - n}
        \right\}\right) 
        \left( 1 + \min \left\{ e ~,~
        \Diam(\Gamma) \right\} \right)^r
        || \psi ||_{\Lip(\rho, \Gamma,W)}.
\eeq
Taking $r := n - (q+1)$ in 
\eqref{lip_k_nested_lip_k-r_norm_bound} yields that
\beq
    \label{lip_k_nested_lip_q+1_norm_bound}
        \left|\left| \psi_{[q]} 
        \right|\right|_{\Lip(q+1, \Gamma,W)}
        \leq
        \left( 1 + \min \left\{ e ~,~
        \Diam(\Gamma)^{\rho - n}
        \right\}\right) \left( 1 + 
        \min \left\{ e ~,~
        \Diam(\Gamma) \right\} \right)^{n - (q+1)}
        || \psi ||_{\Lip(\rho, \Gamma,W)}.
\eeq
As $\th \in (q,q+1]$, we can appeal to 
Claim \ref{lip_k_nested_claim_1} with
$\cd := \Gamma$, $m:=q$, $\lambda := q +1$ and
$\sigma := \th$ to deduce via 
\eqref{lip_k_nested_claim_1_conc} that
\beq
    \label{lip_k_nested_lip_eta_norm_bound_q+1}
        \left|\left| \psi_{[q]} 
        \right|\right|_{\Lip(\th,\Gamma,W)}
        \leq
        \max \left\{ 1 ~,~
        \min \left\{ 1 + e ~,~
        \Diam(\Gamma)^{q+1-\th }
        \right\} \right\}
        \left|\left| \psi_{[q]} 
        \right|\right|_{\Lip(q+1, \Gamma,W)}.
\eeq
Together, \eqref{lip_k_nested_lip_q+1_norm_bound} and 
\eqref{lip_k_nested_lip_eta_norm_bound_q+1} yield that
\beq
    \label{lip_k_nested_lip_eta_norm_bound_final}
        \left|\left| \psi_{[q]} 
        \right|\right|_{\Lip(\th,\Gamma,W)}
        \leq
        C_2
        || \psi ||_{\Lip(\rho, \Gamma,W)}
\eeq
for $C_2 > 0$ defined by
$$C_2 := \max \left\{ 1 ,
    		\min \left\{ 1 + e,
            \Diam(\Gamma)^{q+1-\th }
            \right\}\right\}
            \left( 1 +  \min \left\{e, 
            \Diam(\Gamma)^{\rho - n}
            \right\} \right)
            \left( 1 + \min \left\{
            e , \Diam(\Gamma) \right\}
            \right)^{n-(q+1)}$$
as claimed in \eqref{lip_k_nested_C2}. 
The estimates 
\eqref{lip_k_nested_lip_eta_norm_1&2}
and
\eqref{lip_k_nested_lip_eta_norm_bound_final}
combine to give
$$\left|\left| \psi_{[q]} 
        \right|\right|_{\Lip(\th,\Gamma,W)}
        \leq
        \min \left\{ C_1 ~,~ C_2 \right\}
        || \psi ||_{\Lip(\rho, \Gamma,W)}$$
where
$$ C_1 = \max \left\{ 1 ~,~
        \min \left\{ 1 + e ~,~
            \Diam(\Gamma)^{\rho - \eta}
            +
            \sum_{j=q+1}^{n}
            \frac{\Diam(\Gamma)^{j-\th}}{(j-q)!}
            \right\}
            \right\}$$
and
$$C_2 := \max \left\{ 1 ,
    		\min \left\{ 1 + e,
            \Diam(\Gamma)^{q+1-\th }
            \right\}\right\}
            \left( 1 +  \min \left\{e, 
            \Diam(\Gamma)^{\rho - n}
            \right\} \right)
            \left( 1 + \min \left\{
            e , \Diam(\Gamma) \right\}
            \right)^{n-(q+1)}$$
as claimed in \eqref{lip_k_nested_norm_est}.

Finally, since $C_1 \leq 1 + e$, 
\eqref{lip_k_nested_norm_est_0} 
and \eqref{lip_k_nested_norm_est} 
combine to yield that, for any 
$\th \in (0, \rho)$, we have the estimate
$\left|\left| \psi_{[q]} 
\right|\right|_{\Lip(\th,\Gamma,W)}
\leq (1+e)||\psi||_{\Lip(\rho,\Gamma,W)}$
as claimed in \eqref{lip_k_nested_norm_est_weak}.
This completes the proof of Lemma
\ref{lip_k_spaces_nested_lemma}.
\end{proof}

\section{Local Lipschitz Bounds}
\label{local_lip_bds_sec}
In this section we establish some local estimates arising 
as consequences from knowing the Lipschitz norm of a function 
is small when the domain is taken to be a single point in a
similar spirit to Lemma 1.13 in \cite{Bou15}.
The constants appearing in our estimates are more convenient 
for our purposes.

We first record the following result relating the pointwise
properties of a Lipschitz function at one point to its 
pointwise values at another.
The precise result is the following.

\begin{lemma}[Pointwise Estimates]
\label{lip_k_growth_rates_lemma}
Let $V$ and $W$ be Banach spaces, and assume that the 
tensor powers of $V$ are all equipped with admissible 
norms (cf. Definition \ref{admissible_tensor_norm}).
Assume $\Gamma \subset V$ is closed with $p \in \Gamma$.
Let $A, \rho > 0$, $r_0 \geq 0$, and  
$n \in \Z_{\geq 0}$ such that $\rho \in (n,n+1]$.
Let $F = \left( F^{(0)} , \ldots , F^{(n)}
\right) \in \Lip(\rho, \Gamma,W)$ with 
$|| F ||_{\Lip(\rho, \Gamma,W)} \leq A$.
For every $j \in \{0, \ldots , n\}$
let $R^{F}_j : \Gamma \times \Gamma \to 
\cl(V^{\otimes j};W)$ denote the remainder term 
associated to $F^{(j)}$, defined for 
$x,y \in \Gamma$ and $v \in V^{\otimes j}$ by
\beq
    \label{lip_k_growth_rate_remainder_def} 
        R_j^{F}(x,y)[v] := F^{(j)}(y)[v]
        - \sum_{s=0}^{n-j}\frac{1}{s!}
        F^{(j+s)}(x) \left[ v \otimes 
        (y-x)^{\otimes s} \right].
\eeq
Then for for every $l \in \{0, \ldots , n\}$,
any $x , y \in \Gamma$, and any $\th \in (n, \rho)$,
we have that
\beq
	\label{lip_k_growth_rate_remainder_conc}
		\left|\left| R^{F}_l(x,y)
		\right|\right|_{\cl ( V^{\otimes l} ;W)}
		\leq
		A \left(\dist(x, p) + \dist(y,p)
		\right)^{\rho - \th}
		|| y - x ||_V^{\th - l}.
\eeq
Further, suppose $q \in \{0, \ldots , n\}$ and that
for every $s \in \{0, \ldots , q\}$ we have 
$\left|\left| F^{(s)}(p) 
\right|\right|_{\cl(V^{\otimes s};W)} \leq r_0$. 
Then for any $l \in \{0, \ldots , q\}$ and any 
$x \in \Gamma$ we have that
\beq
    \label{lip_k_growth_rate_conc}
        \left|\left| F^{(l)}(x) 
        \right|\right|_{\cl(V^{\otimes l};W)} 
        \leq 
        \min \left\{ A ~,~   
        A \left[ \dist(x,p)^{\rho - l} + 
        S_{l,q}(x,p) \right] + 
        r_0 \sum_{j=0}^{q-l}
        \frac{1}{j!} \dist(x,p)^j\right\}
\eeq
where
\beq
    \label{lip_k_growth_rate_C0_S_def}
        S_{l,q}(x,p) := \twopartdef
        {\sum_{j=q+1-l}^{n-l} \frac{1}{j!}
        \dist(x,p)^j}{q < n}
        {0}{q=n.}
\eeq
\end{lemma}

\begin{proof}[Proof of Lemma 
\ref{lip_k_growth_rates_lemma}]
Let $V$ and $W$ be Banach spaces, and assume that the 
tensor powers of $V$ are all equipped with admissible 
norms (cf. Definition \ref{admissible_tensor_norm}).
Assume that $\Gamma \subset V$ is closed and that 
$\rho > 0$ with $n \in \Z_{\geq 0}$ such that
$\rho \in (n,n+1]$. Suppose that, for 
$l \in \{0 , \ldots , n \}$, we have functions 
$F^{(l)} : \Gamma \to \cl( V^{\otimes l} ; W)$
such that $F = (F^{(0)} ,\ldots ,F^{(n)})
\in \Lip(\rho, \Gamma, W)$ with 
$|| F ||_{\Lip(\rho,\Gamma,W)} \leq A$.
For each $j \in \{0, \ldots , n\}$ let 
$R^{F}_j : \Gamma \times \Gamma \to 
\cl(V^{\otimes j};W)$ denote the remainder term 
associated to $F^{(j)}$, defined for 
$x,y \in \Gamma$ and $v \in V^{\otimes j}$ by
\beq
    \label{lip_k_growth_rate_remainder_def_pf} 
        R_j^{F}(x,y)[v] := F^{(j)}(y)[v]
        - \sum_{s=0}^{n-j}\frac{1}{s!}
        F^{(j+s)}(x) \left[ v \otimes 
        (y-x)^{\otimes s} \right].
\eeq
As a consequence of $F \in \Lip(\rho,\Gamma,W)$
with $|| F ||_{\Lip(\rho,\Gamma,W)} \leq A$, 
whenever $l \in \{0, \ldots , n\}$ and $x,y \in \Gamma$,
we have the bounds (cf. \eqref{lip_k_bdd} and 
\eqref{lip_k_remain_holder})
\beq
    \label{lip_k_growth_rate_bdd}
        (\bI) \quad \left|\left| F^{(l)}(x)
        \right|\right|_{\cl(V^{\otimes l}; W)}
        \leq A
        \qquad \text{and} \qquad
        (\bII) \quad \left|\left| R^{F}_l(x,y) 
        \right|\right|_{\cl(V^{\otimes l};W)}
        \leq A || y - x ||_V^{\rho - l}.
\eeq
If $l \in \{0, \ldots , n\}$, $x,y \in \Gamma$, 
and $\th \in (n , \rho)$, we use (\bII) of 
\eqref{lip_k_growth_rate_bdd} to compute that
\begin{align*}
	\left|\left| R^{F}_l(x,y) 
	\right|\right|_{\cl (V^{\otimes l} ; W)}
		\stackrel{(\bII) \text{ of } 
		\eqref{lip_k_growth_rate_bdd}}
		{\leq}
		A || y - x ||_V^{\rho - l}	&=
		A || y - x ||_V^{\rho - \th} 
		||y-x||_V^{\th - l} \\
		&\leq
		A \left( ||x-p||_V + ||y-p||_V \right)^{\rho - \th}
		|| y - x ||_V^{\th - l} \\
		&=
		A \left( \dist(x,p) + \dist(y,p) 
		\right)^{\rho - \th}
		|| y - x ||_V^{\th - l}
\end{align*}
as claimed in \eqref{lip_k_growth_rate_remainder_conc}.

Now suppose that $q \in \{0 , \ldots , n\}$ and that
for every $s \in \{0, \ldots , q\}$ we have 
$\left|\left| F^{(s)}(p) 
\right|\right|_{\cl(V^{\otimes s};W)} \leq r_0$. 
Given $l \in \{0 , \ldots , q\}$,
$x \in \Gamma$, and $v \in V^{\otimes l}$, 
recalling that the tensor powers
of $V$ are equipped with admissible norms (cf. 
Definition \ref{admissible_tensor_norm}), we may use 
\eqref{lip_k_growth_rate_remainder_def_pf} and (\bII)
of \eqref{lip_k_growth_rate_bdd} to obtain that
\beq
    \label{lip_k_growth_rate_C0_step1}
        \left|\left|F^{(l)}(x)[v]
        \right|\right|_W
        \leq
        \sum_{s=0}^{n - l} \frac{1}{s!}
        \left|\left|F^{(s)}(p)
        \right|\right|_{\cl(V^{\otimes s};W)}
        ||x-p||_V^s || v ||_{V^{\otimes l}}
        + A ||x-p||_V^{\rho - l}
        ||v ||_{V^{\otimes l}}.
\eeq      
If $q = n$ then \eqref{lip_k_growth_rate_C0_step1}
tells us that 
\beq
    \label{lip_k_growth_rate_C0_step2}
        \left|\left|F^{(l)}(x)[v]
        \right|\right|_W
        \leq
        \left(r_0 \sum_{s=0}^{q - l} \frac{1}{s!}
        ||x-p||_V^s
        + A ||x-p||_V^{\rho - l} \right)
        ||v||_{V^{\otimes l}}.
\eeq 
Whilst if $q < n$, we deduce from
\eqref{lip_k_growth_rate_C0_step1} that
\beq
    \label{lip_k_growth_rate_C0_step3}
        \left|\left|F^{(l)}(x)[v]
        \right|\right|_W 
        \leq
        \left(r_0\sum_{s=0}^{q - l} \frac{1}{s!}
        ||x-p||_V^s
        + A \sum_{s=q-l+1}^{n-l} \frac{1}{s!}
        ||x-p||_V^s
        + A ||x-p||_V^{\rho - l} \right) 
        ||v||_{V^{\otimes l}}.
\eeq 
If we let $S_{l,q}(x,p)$ be defined as in 
\eqref{lip_k_growth_rate_C0_S_def},
then \eqref{lip_k_growth_rate_C0_step2} and
\eqref{lip_k_growth_rate_C0_step3} combine to yield 
that 
\beq
    \label{lip_k_growth_rate_C0_step4}
        \left|\left|F^{(l)}(x)[v]
        \right|\right|_W 
        \leq
        \left(r_0 \sum_{s=0}^{q-l}
        \frac{1}{s!}\dist(x,p)^s
        + A\left[ \dist(x,p)^{\rho} + 
        S_{l,q}(x,p) \right] \right)
        ||v||_{V^{\otimes l}}.
\eeq
Taking the supremum over
$v \in V^{\otimes l}$ with unit $V^{\otimes l}$
norm in \eqref{lip_k_growth_rate_C0_step4}
yields the second estimate claimed in 
\eqref{lip_k_growth_rate_conc}. 
The first estimate claimed in 
\eqref{lip_k_growth_rate_conc} follows from (\bI) in 
\eqref{lip_k_growth_rate_bdd}.
This completes the proof of Lemma 
\ref{lip_k_growth_rates_lemma}.
\end{proof}
\vskip 4pt 
\noindent 
Our aim for the remainder of this section is to strengthen 
the pointwise estimates obtained in Lemma 
\ref{lip_k_growth_rates_lemma} to full 
Lipschitz norm bounds on a local neighbourhood of the 
point $p$. 
The first local Lipschitz norm bounds we can establish 
in a neighbourhood of a given point are recorded in the 
following result.

\begin{lemma}[Local Lipschitz Bounds I]
\label{lip_k_lip_norm_ball_bound_lemma}
Let $V$ and $W$ be Banach spaces, 
and assume that the tensor powers of $V$ are all 
equipped with admissible norms (cf. Definition 
\ref{admissible_tensor_norm}).
Assume that $\Gamma \subset V$ is non-empty and 
closed, and that $z \in \Gamma$. 
Let $A , \rho > 0$ with $n \in \Z_{\geq 0}$ 
such that $\rho \in (n,n+1]$, $r_0 \in [0,A]$, 
and $\th \in (n,\rho)$. 
Suppose that 
$F = \left( F^{(0)} , \ldots , F^{(n)} \right) 
\in \Lip(\rho, \Gamma,W)$ satisfies 
that $ ||F||_{\Lip(\rho,\Gamma,W)} \leq A$,
and that for every $j \in \{0,\ldots,n\}$ we have 
the bound $ \left|\left| F^{(j)}(z) 
\right|\right|_{\cl(V^{\otimes j};W)} \leq r_0$.
Then for any $\de \in [0,1]$ we have that 
\beq
	\label{lip_k_lip_norm_ball_bound_lemma_conc}
		|| F ||_{\Lip(\th,\Omega,W)} 
		\leq \max \left\{
		(2\de)^{\rho - \th} A ~,~
		\min \left\{ A ~,~
		A \de^{\rho -n} + r_0 e^{\de}
		\right\} \right\}
\eeq
where $\Omega := \Gamma ~\cap~ \ovB_V (z,\de)$.
\end{lemma}

\begin{proof}[Proof of Lemma 
\ref{lip_k_lip_norm_ball_bound_lemma}]
Let $V$ and $W$ be Banach spaces, 
and assume that the tensor powers of $V$ are all 
equipped with admissible norms (cf. Definition 
\ref{admissible_tensor_norm}).
Assume that $\Gamma \subset V$ is non-empty and 
closed, and that $z \in \Gamma$. 
Let $A , \rho > 0$ with $n \in \Z_{\geq 0}$ 
such that $\rho \in (n,n+1]$, $r_0 \in [0,A]$, 
and $\th \in (n,\rho)$. Suppose that 
$F = \left( F^{(0)} , \ldots , F^{(n)} \right) 
\in \Lip(\rho, \Gamma,W)$ satisfies 
that $ ||F||_{\Lip(\rho,\Gamma,W)} \leq A$,
and that for every $j \in \{0,\ldots,n\}$ we have 
the bound $ \left|\left| F^{(j)}(z) 
\right|\right|_{\cl(V^{\otimes j};W)} \leq r_0$.
For each $l \in \{0, \ldots , n\}$
let $R^{F}_l : \Gamma \times \Gamma 
\to \cl ( V^{\otimes l} ; W) $ be defined for 
$x,y \in \Gamma$ and $v \in V^{\otimes l}$ by
\beq
    \label{lip_k_extension_lemma_psi_tay_exp}
        R^{F}_l(x,y)[v] := 
        F^{(l)}(y)[v] -
        \sum_{j=0}^{n-l} \frac{1}{j!}
        F^{(j + l)}(x) \left[ 
        v \otimes (y - x)^{\otimes j} 
        \right].
\eeq
An application of Lemma 
\ref{lip_k_growth_rates_lemma},
with $A$, $r_0$, $\rho$, $n$ and $\th$ here playing 
the same roles there, yields that for each  
$l \in \{0, \ldots , n\}$ and any $x \in \Sigma$ we have
(cf. \eqref{lip_k_growth_rate_conc} for $q=n$)
\beq
    \label{lip_k_ball_lemma_psi_l_bounds}
        \left|\left| F^{(l)}(x)
        \right|\right|_{\cl (V^{\otimes l} ; W)}
        \leq
        \min \left\{ A ~,~
        A \dist(x,p)^{\rho -l}
        +
        r_0 \sum_{s=0}^{n-l} \frac{1}{s!} 
        \dist(x,p)^s \right\}
\eeq
and (cf. \eqref{lip_k_growth_rate_remainder_conc})
\beq
	\label{lip_k_ball_lemma_remainder_bounds}
		\left|\left| R^{F}_l(x,y)
		\right|\right|_{\cl ( V^{\otimes l} ; W)}
		\leq
		A \left(
		\dist(x,p) + \dist(y,p)
		\right)^{\rho - \th}
		|| y - x||^{\th - l}_V.
\eeq
Now let $\de \in [0,1]$ and define 
$\Omega := \ovB_V(z,\de) ~\cap~ \Gamma 
\subset \Gamma$. 
Then given any $l \in \{0 , \ldots , n\}$ and any
$x,y \in \Omega$,
\eqref{lip_k_ball_lemma_psi_l_bounds} tells us that 
\beq
    \label{lip_k_ball_lemma_psi_l_bounds_A}
        \left|\left| F^{(l)}(x)
        \right|\right|_{\cl (V^{\otimes l} ; W)}
        \leq
        \min \left\{ A ~,~
        A \de^{\rho -l}
        +
        r_0 \sum_{s=0}^{n-l} \frac{\de^s}{s!} 
         \right\}
         \leq 
         \min \left\{ A ~,~
         A \de^{\rho - n} + r_0 e^{\de}
         \right\},
\eeq
since $\de \in [0,1]$ means $\de^{\rho - l} 
\leq \de^{\rho - n}$
for every $l \in \{0, \ldots , n\}$,
whilst \eqref{lip_k_ball_lemma_remainder_bounds} 
tells us that
\beq
	\label{lip_k_ball_lemma_remainder_bounds_A}
		\left|\left| R^{F}_l(x,y)
		\right|\right|_{\cl ( V^{\otimes l} ; W)}
		\leq
		A \left( 2\de\right)^{\rho - \th}
		|| y - x||^{\th - l}_V.
\eeq 
The estimates \eqref{lip_k_ball_lemma_psi_l_bounds_A} 
and \eqref{lip_k_ball_lemma_remainder_bounds_A}
allow us to conclude that $F \in \Lip(\th,\Omega,W)$
with
$$ || F ||_{\Lip(\th,\Omega,W)} \leq 
\max \left\{ \left( 2\de\right)^{\rho - \th} A  ~,~ 
\min \left\{ A ~,~ A \de^{\rho - n} + r_0 e^{\de}
\right\} \right\} $$
as claimed in \eqref{lip_k_lip_norm_ball_bound_lemma_conc}.
This completes the proof of Lemma 
\ref{lip_k_lip_norm_ball_bound_lemma}.
\end{proof}
\vskip 4pt
\noindent
Extending the local Lipschitz estimates of 
Lemma \ref{lip_k_lip_norm_ball_bound_lemma}
to the setting, in the notation of 
Lemma \ref{lip_k_lip_norm_ball_bound_lemma}, that 
$\th \leq n < \rho \leq n+1$ is more challenging.
We achieve this by combining the 
\textit{Lipschitz Nesting} Lemma \ref{lip_k_spaces_nested_lemma} 
from Section \ref{nested_embed_sec} with 
Lemma \ref{lip_k_lip_norm_ball_bound_lemma}.
The resulting local Lipschitz bounds are precisely recorded in 
the following result.

\begin{lemma}[Local Lipschitz Bounds II]
\label{lip_k_func_ext_lemma_gen_eta}
Let $V$ and $W$ be Banach spaces, and assume 
that the tensor powers of $V$ are all 
equipped with admissible norms (cf. Definition 
\ref{admissible_tensor_norm}). 
Assume that $\Gamma \subset V$ is a non-empty closed
subset with $z \in \Gamma$. 
Let $A >0$,  $r_0 \in [0,A)$,  $\rho > 1$
with $n \in \Z_{\geq 1}$ such that $\rho \in (n,n+1]$, 
and $\th \in (0,n]$ with $q \in \{0, \ldots , n-1\}$
such that $\th \in (q,q+1]$. Suppose
that $F \in \Lip(\rho, \Gamma,W)$ satisfies
that $ || F||_{\Lip(\rho,\Gamma,W)} \leq A$,
and that for every $j \in \{0, \ldots , n\}$ we have 
the bound $ \left|\left| F^{(j)}(z) 
\right|\right|_{\cl(V^{\otimes j};W)} \leq r_0$. 
Given any $\de \in [0,1]$  
we have, for $\Omega := \ovB_V(z,\de) ~\cap~
\Gamma$, that
\beq
    \label{lip_k_ext_gen_eta_lemma_eta_lip_bound}
        \left|\left| F_{[q]} 
        \right|\right|_{\Lip(\th,\Omega,W)}
		\leq
		\max \left\{ (2\de)^{q+1-\th} E_{n-b_q}, 
		\min \left\{ E_{n-b_q} ~,~ \de E_{n-b_q}  + 
		r_0 e^{\de} \right\} \right\}
\eeq
where $F_{[q]} = \left( F^{(0)},\ldots ,F^{(q)}\right)$,
$b_q := n - (q+1)$, and for 
$s \in \{0, \ldots , n - 1 \}$ $E_{n-s}$ is 
inductively defined by
\beq
	\label{lip_k_ball_control_Es_statement}
		E_{n-s} := \twopartdef
		{\left( 1 + (2\de)^{\frac{\rho - n}{2}} \right)
		\max \left\{ (2\de)^{\frac{\rho - n}{2}} A ~,~ 
		\min \left\{ A ~,~ \de^{\rho - n} A + 
		r_0 e^{\de} \right\} \right\} }{s=0}
		{\left( 1 + \sqrt{2\de} \right) 
		\max \left\{ \sqrt{2 \de} E_{n-(s-1)} ~,~ 
		\min \left\{ E_{n-(s-1)} ~,~ \de E_{n-(s-1)} + 
		r_0 e^{\de} \right\} \right\}}{s \geq 1.}
\eeq
Consequently, if $r_0 = 0$ we can conclude that
\beq
	\label{lip_k_ball_control_r0=0_conc}
		\left|\left| F_{[q]} 
		\right|\right|_{\Lip(\th,\Omega,W)} \leq 
		\left( 1 + (2\de)^{\frac{\rho-n}{2}} \right) 
		\left( 1+\sqrt{2\de} \right)^{n-(q+1)}
		(2\de)^{
		\frac{\rho - \th}{2} + \frac{q+1-\th}{2}} A.
\eeq
If $0 < r_0 < A$ then there exists 
$\de_{\ast} = \de_{\ast} (A ,r_0 ,\rho) > 0$ 
such that if we additionally impose that 
$\de \in [0 , \de_{\ast}]$ then we may conclude that
\beq
	\label{lip_k_ball_control_r0>0_conc}
		\left|\left| F_{[q]} 
		\right|\right|_{\Lip(\th,\Omega,W)} \leq
		\max \left\{ (2\de)^{q+1-\th} \e
		~,~
		\min \left\{ \e ~,~ \de \e 
		+ r_0 e^{\de}
		\right\}
		\right\}
\eeq
for $\e = \e(A,r_0,\rho,\th, \de)>0$ defined by
\beq
	\label{lip_k_ball_norm_cal_e_def}
		\e :=
		\left( 1 + (2\de)^{\frac{\rho-n}{2}} \right)
		\left( 1 + \sqrt{2\de} \right)^{n - (q+1)} 
		\left( \de^{\rho - (q+1)} A
		+ r_0 \de^{n-(q+1)} e^{\de} \right)
		+
		\bX_{n-(q+1)}(\de)
\eeq
where, for $t \in \{0 , \ldots , n-1 \}$, 
the quantity $\bX_t(\de)$ is defined by
\beq
	\label{lip_k_ball_norm_control_r0>0_bX_t_statement}
		\bX_t(\de) :=
		\twopartdef
		{0}{t=0}
		{\left( 1 + \sqrt{2\de}\right)
		r_0e^{\de}
		\sum_{j=0}^{t-1} \de^{j} 
		\left(1 + \sqrt{2\de}\right)^{j}}{t \geq 1.}
\eeq
\end{lemma}

\begin{proof}[Proof of Lemma 
\ref{lip_k_func_ext_lemma_gen_eta}]
Let $V$ and $W$ be Banach spaces, and assume 
that the tensor powers of $V$ are all 
equipped with admissible norms (cf. Definition 
\ref{admissible_tensor_norm}). 
Assume that $\Gamma \subset V$ is a non-empty closed
subset with $z \in \Gamma$. 
Let $A >0$,  $r_0 \in [0,A)$,  $\rho > 1$
with $n \in \Z_{\geq 1}$ such that $\rho \in (n,n+1]$, 
and $\th \in (0,n]$ with $q \in \{0, \ldots , n-1\}$
such that $\th \in (q,q+1]$. Suppose
that $F \in \Lip(\rho, \Gamma,W)$ satisfies
that $ || F||_{\Lip(\rho,\Gamma,W)} \leq A$,
and that for every $j \in \{0, \ldots , n\}$ we have 
the bound $ \left|\left| F^{(j)}(z) 
\right|\right|_{\cl(V^{\otimes j};W)} \leq r_0$.
Fix $\de \in [0,1]$ and define 
$\Omega := \Gamma ~\cap~ \ovB_V (z , \de)$.
For $s \in \{0, \ldots , n - 1 \}$ inductively define 
\beq
	\label{lip_k_ball_control_Es}
		E_{n-s} := \twopartdef
		{\left( 1 + (2\de)^{\frac{\rho-n}{2}} \right)
		 \max \left\{ (2\de)^{\frac{\rho - n}{2}} A ~,~ 
		\min \left\{ A ~,~ A \de^{\rho - n} + 
		r_0 e^{\de} \right\} \right\} }{s=0}
		{\left( 1 + \sqrt{2\de}\right) 
		\max \left\{ \sqrt{2\de} E_{n-(s-1)} ~,~ 
		\min \left\{ E_{n-(s-1)} ~,~ \de E_{n-(s-1)} + 
		r_0 e^{\de} \right\} \right\}}{s \geq 1.}
\eeq
We first prove that each $E_{n-s}$ is bounded
from below by $r_0$. This is the content of the
following claim.

\begin{claim}
\label{lip_k_ball_control_claim_1}
For every $s \in \{0, \ldots , n-1\}$ we have 
\beq
    \label{lip_k_ball_control_claim_1_Es_lb_conc}
        E_{n-s} \geq r_0.
\eeq
\end{claim}

\begin{proof}[Proof of Claim
\ref{lip_k_ball_control_claim_1}]
The claim is proven via induction on 
$s \in \{0, \ldots , n-1\}$. For $s=0$ we have 
\beq
    \label{lip_k_ball_control_claim_1_s=0}
        E_n \stackrel{
        \eqref{lip_k_ball_control_Es}
        }{\geq} 
        \max \left\{ ( 2 \de)^{\frac{\rho - n}{2}}A
        ~,~
		\min \left\{ A ~,~ A \de^{\rho - n} + 
		r_0 e^{\de} \right\} \right\}
		\geq 
		\min \left\{ A ~,~ A \de^{\rho - n} + 
		r_0 e^{\de} \right\}
		\geq
		r_0 
\eeq
where the last inequality uses that $r_0 \leq A$.
If \eqref{lip_k_ball_control_claim_1_Es_lb_conc}
is valid for $s \in \{0 , \ldots , n-2\}$
then we compute that
\begin{align*}
    E_{n-(s+1)} &\stackrel{
    \eqref{lip_k_ball_control_Es}}{=} 
        \left( 1 + \sqrt{2\de}\right) 
		\max \left\{ \sqrt{2\de} E_{n-s} ~,~ 
		\min \left\{ E_{n-s} ~,~ \de E_{n-s} + 
		r_0 e^{\de} \right\} \right\} \\
		&\geq
		\max \left\{ \sqrt{2\de} E_{n-s} ~,~ 
		\min \left\{ E_{n-s} ~,~ \de E_{n-s} + 
		r_0 e^{\de} \right\} \right\} \\
		&\geq 
		\min \left\{ E_{n-s} ~,~ \de E_{n-s} + 
		r_0 e^{\de} \right\} \geq r_0
\end{align*}
where the last line uses that 
$E_{n-s} \geq r_0$ by the assumption that
\eqref{lip_k_ball_control_claim_1_Es_lb_conc}
is valid for $s$, and that
$\de E_{n-s} + r_0 e^{\de} \geq r_0$ since
$\de \geq 0$. Thus we have established that the
estimate \eqref{lip_k_ball_control_claim_1_Es_lb_conc}
for $s \in \{0 , \ldots , n-2\}$ yields that the
estimate \eqref{lip_k_ball_control_claim_1_Es_lb_conc}
is true for $s+1$.
Since \eqref{lip_k_ball_control_claim_1_s=0}
establishes that 
\eqref{lip_k_ball_control_claim_1_Es_lb_conc}
is true for $s=0$, we may use induction to prove
that \eqref{lip_k_ball_control_claim_1_Es_lb_conc}
is in fact true for every
$s \in \{0, \ldots , n-1\}$ as claimed.
This completes the proof of Claim
\ref{lip_k_ball_control_claim_1}.
\end{proof}
\vskip 4pt
\noindent
We now prove that, for each 
$s \in \{0, \ldots , n-1\}$, the 
$\Lip(n-s,\Omega,W)$-norm of $F_{[n-s-1]} =
\left( F^{(0)} , \ldots , F^{(n-s-1)} \right)$ is bounded
above by $E_{n-s}$. This is the content of the 
following claim.

\begin{claim}
\label{lip_k_ball_control_claim_2}
For every $s \in \{0, \ldots , n-1\}$ we have that 
\beq
	\label{lip_k_ball_control_n-s_Es}
		\left|\left| F_{[n-(s+1)]} 
		\right|\right|_{\Lip(n-s,\Omega,W)} \leq E_{n-s}.
\eeq
\end{claim}

\begin{proof}[Proof of Claim
\ref{lip_k_ball_control_claim_2}]
We will prove \eqref{lip_k_ball_control_n-s_Es}
via induction on $s \in \{0, \ldots , n-1\}$.
We begin with the base case that $s=0$. In this case,
consider $\xi := \frac{\rho -n}{2} \in (0, \rho - n)$ 
so that $n + \xi \in (n , \rho)$ with $0 < \xi \leq 1/2$.
An initial application of Lemma 
\ref{lip_k_lip_norm_ball_bound_lemma},
with $\Gamma$, $A$, $r_0$, $\rho$ and $n$ 
here playing the same role
and with the $\th$ in Lemma 
\ref{lip_k_lip_norm_ball_bound_lemma}
being $n + \xi$ here, yields that
\beq
    \label{lip_k_ball_control_n+xi_est}
		|| F ||_{\Lip(n + \xi , \Omega,W)}
		\leq
		\max \left\{ (2\de)^{\rho - n - \xi} A , 
		\min \left\{ A ~,~ A \de^{\rho - n} + 
		r_0 e^{\de} \right\} \right\}.
\eeq
We next apply Lemma \ref{lip_k_spaces_nested_lemma}, 
with the $\Gamma$, $\rho$ and $\th$ of that result 
as $\Omega$, $n+\xi$ and $n$ here 
respectively, to obtain that
\beq
    \label{lip_k_ball_control_n_est_a}
		\left|\left| F_{[n-1]} 
		\right|\right|_{\Lip(n,\Omega,W)}
		\leq
		\min \{ C_1 , C_2 \}
		|| F ||_{\Lip(n+\xi,\Omega,W)}
\eeq
where (cf. \eqref{lip_k_nested_C1} and recalling 
both that $\Diam(\Omega) \leq 2\delta \leq 2$
and that $0 < \xi \leq 1$)
\beq
    \label{lip_k_ball_control_n_C1}
		C_1 = \max \left\{ 1 ~,~
		\min \left\{ 1 + e ~,~ 
		\Diam(\Omega)^{\xi} + \sum_{j=n}^n 
		\frac{\Diam(\Omega)^{j-n}}{(j-(n-1))!}
		\right\} \right\}
		=
		1 + \Diam(\Omega)^{\xi}
		\leq 
		1 + (2\de)^{\xi},
\eeq
and (cf. \eqref{lip_k_nested_C2})
\beq
    \label{lip_k_ball_control_n_C2_A}
    C_2 = \max \left\{ 1 ,
		\min \left\{ 1 + e ,\Diam(\Omega)^{n-n} 
		\right\} \right\}
		\left( 1 + \min \left\{ e ,\Diam(\Omega)
		\right\}^{\xi} \right)
		\left(1 + 
		\min \left\{ e , \Diam(\Omega)
		\right\}\right)^{n-n} 
\eeq
so that, since $\Diam(\Omega) \leq 2\delta \leq 2$
and $0 < \xi \leq 1$, we have
\beq
    \label{lip_k_ball_control_n_C2}
		C_2 = 1 + \Diam(\Omega)^{\xi}
		\leq 
		1 + (2\de)^{\xi}
\eeq
The combination of 
\eqref{lip_k_ball_control_n+xi_est},
\eqref{lip_k_ball_control_n_est_a},
\eqref{lip_k_ball_control_n_C1}, and 
\eqref{lip_k_ball_control_n_C2} yields that
\begin{align*}
	\left|\left| F_{[n-1]} 
	\right|\right|_{\Lip(n,\Omega,W)}
		&\leq
		\left( 1 + (2\de)^{\xi} \right)
		\max \left\{ (2\de)^{\rho - n - \xi} A , 
		\min \left\{ A ~,~ A \de^{\rho - n} + 
		r_0 e^{\de} \right\} \right\} \\
		&=
		\left( 1 + (2\de)^{\frac{\rho - n}{2}} \right)
		\max \left\{ (2\de)^{\frac{\rho - n}{2}} A , 
		\min \left\{ A ~,~ A \de^{\rho - n} + 
		r_0 e^{\de} \right\} \right\} 
		\stackrel{
		\eqref{lip_k_ball_control_Es}}
		{=}
		E_n.
\end{align*}
This completes the base case of our induction by verifying 
\eqref{lip_k_ball_control_n-s_Es} when $s=0$.

Now assume that $n-1 \geq 1$, that 
$s \in \{1 , \ldots , n-1\}$, 
and that \eqref{lip_k_ball_control_n-s_Es} 
is true for $s-1$.
Consider $\xi := \frac{1}{2} \in (0, 1)$ 
so that $n-s+\xi \in (n-s , n-(s-1))$.
An initial application of Lemma 
\ref{lip_k_lip_norm_ball_bound_lemma},
with $\Gamma$, $A$, $r_0$, $\rho$ and $\th$ of 
that result as
$\Omega$, $E_{n-(s-1)}$, $r_0$, $n-(s-1)$ and 
$n-s + \xi$ here 
respectively, yields that
\beq
	\label{lip_k_ball_control_n-s+xi_est}
		\left|\left| F_{[n-s]} 
		\right|\right|_{\Lip(n - s + \xi ,\Omega,W)}
		\leq
		\max \left\{ (2\de)^{1-\xi} E_{n-(s-1)}, 
		\min \left\{ E_{n-(s-1)} ~,~ E_{n-(s-1)} \de + 
		r_0 e^{\de} \right\} \right\}.
\eeq
We next apply Lemma \ref{lip_k_spaces_nested_lemma}, 
with the $\Gamma$, $\rho$ and $\th$ of that result 
as $\Omega$, $n-s+\xi$ and $n-s$ here 
respectively, to obtain that
\beq
	\label{lip_k_ball_control_n-s_est_a}
		\left|\left| F_{[n-(s+1)]} 
		\right|\right|_{\Lip(n-s,\Omega,W)}
		\leq
		\min \{ D_1 , D_2 \}
		\left|\left| F_{[n-s]} 
		\right|\right|_{\Lip(n - s + \xi ,\Omega,W)}
\eeq
where (cf. \eqref{lip_k_nested_C1} and 
recalling that $\Diam(\Omega) \leq 2\de \leq 2$
and $\xi := \frac{1}{2} \leq 1$)
\beq
	\label{lip_k_ball_control_n_D1}
		D_1 = \max \left\{ 1 ~,~
		\min \left\{ 1 + e ,
		\Diam(\Omega)^{\xi} + \sum_{j=n-s}^{n-s} 
		\frac{\Diam(\Omega)^{j-(n-s)}}{(j-(n-s-1))!}
		\right\} \right\}
		=
		1 + \Diam(\Omega)^{\xi}
		\leq 
		1 + (2\de)^{\xi},
\eeq
and (cf. \eqref{lip_k_nested_C2})
\beq
	\label{lip_k_ball_control_n_D2_A}
		D_2 = \max \left\{ 1 ~,~ 
		\min \left\{ 1+e ,\Diam(\Omega)^{0} 
		\right\} \right\}
		\left( 1 + \min \left\{ e ,
		\Diam(\Omega) \right\}^{\xi} \right)
		\left(1 + \min \left\{ e , 
		\Diam(\Omega) \right\} \right)^{0}
\eeq
so that, since $\Diam(\Omega) \leq 2 \de \leq 2$
and $\xi := \frac{1}{2} \leq 1$, we have 
\beq
    \label{lip_k_ball_control_n_D2}
        D_2 =
		1 + \Diam(\Omega)^{\xi}
		\leq 
		1 + (2\de)^{\xi}.
\eeq
The combination of 
\eqref{lip_k_ball_control_n-s+xi_est},
\eqref{lip_k_ball_control_n-s_est_a},
\eqref{lip_k_ball_control_n_D1}, and 
\eqref{lip_k_ball_control_n_D2} yields that
\begin{align*}
    \left|\left| F_{[n-(s+1)]} 
	\right|\right|_{\Lip(n-s,\Omega,W)}
		&\leq
		\left( 1 + (2\de)^{\xi} \right)
		\max \left\{ (2\de)^{1 - \xi} E_{n-(s-1)} , 
		\min \left\{ E_{n-(s-1)} ~,~ \de
		E_{n-(s-1)} + 
		r_0 e^{\de} \right\} \right\} \\
		&=
		\left( 1 + \sqrt{2\de} \right)
		\max \left\{ \sqrt{2\de} E_{n-(s-1)} , 
		\min \left\{ E_{n-(s-1)} ~,~ \de
		E_{n-(s-1)} + 
		r_0 e^{\de} \right\} \right\} \\
		&\stackrel{
		\eqref{lip_k_ball_control_Es}}
		{=}
		E_{n-s}.
\end{align*}
This completes the proof of the inductive step by 
establishing that if \eqref{lip_k_ball_control_n-s_Es} 
is valid for
$s-1$ with $s \in \{1 ,\ldots , n-1\}$, then 
\eqref{lip_k_ball_control_n-s_Es} is in fact valid for $s$.

Using the base case and the inductive step allows us to 
conclude that \eqref{lip_k_ball_control_n-s_Es} is 
valid for every $s \in \{0, \ldots , n-1\}$.
This completes the proof of Claim 
\ref{lip_k_ball_control_claim_2}.
\end{proof}
\vskip 4pt
\noindent
By appealing to Claim \ref{lip_k_ball_control_claim_2},
we conclude that, for every $s \in \{0, \ldots , n-1\}$ 
we have that
\beq
	\label{lip_k_ball_control_n-s_Es_restate}
		\left|\left| F_{[n-(s+1)]}
		\right|\right|_{\Lip(n-s,\Omega,W)} 
		\leq E_{n-s}.
\eeq
Let $b_q := n - (q+1) \in \{0 , \ldots , n-1\}$ so that 
$q+1 = n - b_q$.
Then \eqref{lip_k_ball_control_n-s_Es_restate} 
for $s := b_q$ tells us that
\beq
	\label{lip_k_ball_control_q+1_Em_restate}
		\left|\left| F_{[q]} 
		\right|\right|_{\Lip(q+1,\Omega,W)} 
		\leq E_{n-b_q}.
\eeq
A final application of Lemma 
\ref{lip_k_lip_norm_ball_bound_lemma},
with $\Gamma$, $A$, $r_0$, $\rho$ and $\th$ 
of that result as
$\Omega$, $E_{n-b_q}$, $r_0$, $q+1$ and $\th$ here,
yields that
\beq
	\label{lip_k_ball_control_theta_est}
		\left|\left| F_{[q]} 
		\right|\right|_{\Lip(\th,\Omega,W)} 
		\leq
		\max \left\{ (2\de)^{q+1-\th} E_{n-b_q}, 
		\min \left\{ E_{n-b_q} ~,~ \de E_{n-b_q} + 
		r_0 e^{\de} \right\} \right\}
\eeq
which is precisely the bound claimed in 
\eqref{lip_k_ext_gen_eta_lemma_eta_lip_bound}.

Now suppose that $r_0 = 0$. Then from 
\eqref{lip_k_ball_control_Es}, 
for $s \in \{0, \ldots , n - 1 \}$ we have that 
\beq
	\label{lip_k_ball_control_Es_r0=0}
		E_{n-s} := \twopartdef
		{\left( 1 + (2\de)^{\frac{\rho-n}{2}} \right) 
		\max \left\{ (2\de)^{\frac{\rho - n}{2}} A ~,~ 
		\min \left\{ A ~,~ \de^{\rho - n} A
		\right\} \right\} }{s=0}
		{\left(1+\sqrt{2\de}\right) 
		\max \left\{ \sqrt{2\de} E_{n-(s-1)} ~,~ 
		\min \left\{ E_{n-(s-1)} ~,~ \de E_{n-(s-1)} 
		\right\} \right\}}{s \geq 1.}
\eeq
Since $\de \in [0,1]$ we have both that 
$\de \leq \sqrt{\de} \leq 1$ and 
$\de^{\rho - n} \leq \de^{\frac{\rho - n}{2}} \leq 1$
Consequently, 
\eqref{lip_k_ball_control_Es_r0=0} yields that
\beq
	\label{lip_k_ball_control_Es_r0=0_v2}
		E_{n-s} := \twopartdef
		{\left( 1 + (2\de)^{\frac{\rho-n}{2}} \right) 
		(2\de)^{\frac{\rho - n}{2}} A }{s=0}
		{\left(1+\sqrt{2\de}\right) 
		\sqrt{2\de} E_{n-(s-1)} }{s \geq 1.}
\eeq
Proceeding inductively via 
\eqref{lip_k_ball_control_Es_r0=0_v2}, 
we establish that for any 
$s \in \{0 , \ldots , n-1\}$ we have that
\beq
	\label{lip_k_ball_control_En-s_r0=0}
		E_{n-s} 
		=
		\left( 1 + (2\de)^{\frac{\rho-n}{2}} \right) 
		\left( 1+\sqrt{2\de} \right)^s
		(2\de)^{\frac{\rho - n +s}{2}} A.
\eeq
Observing that $\de \in [0,1]$ means that
$\de \leq \de^{q+1-\th} \leq 1$, we may combine
\eqref{lip_k_ball_control_theta_est} and
\eqref{lip_k_ball_control_En-s_r0=0} for the choice 
$s := b_q = n - (q+1)$ to obtain that
\beq
    \label{lip_k_ball_control_lip_est_r0=0_A}
	    \left|\left| F_{[q]} 
		\right|\right|_{\Lip(\th,\Omega,W)}
		\leq
		(2\de)^{q+1-\th} E_{n-b_q}
		\stackrel{
		\eqref{lip_k_ball_control_En-s_r0=0}
		}{=} 
		\left( 1 + (2\de)^{\frac{\rho-n}{2}} \right) 
		\left( 1+\sqrt{2\de} \right)^{b_q}
		(2\de)^{\frac{\rho - n +b_q}{2} + q + 1 - \th} A. 
\eeq
Since $\rho - n + b_q = \rho - (q+1)$ we see that
\eqref{lip_k_ball_control_lip_est_r0=0_A} is precisely
the estimate claimed in 
\eqref{lip_k_ball_control_r0=0_conc}.

Now assume that $r_0 \in (0, A)$. 
We first let $\de_{\ast} := 1$
To establish \eqref{lip_k_ball_control_r0>0_conc} 
we must reduce $\de_{\ast}$ to a smaller constant.
With the benefit of hindsight, it will suffice to
reduce $\de_{\ast}$, depending only on
$A$, $r_0$, and $\rho$, to ensure that
whenever $\de \in [0, \de_{\ast}]$ we have the 
estimates
\beq
	\label{lip_k_ball_control_de<de_star_conds}
		\left\{
		\begin{aligned}
		&	(\bI) \quad 
			\max \left\{ 1 + \sqrt{2\de} ~,~
			1 + \left( 2 \de \right)^{\frac{\rho - n}{2}}
			\right\}
			< 2 \quad 
			\left( \text{in particular } 
			2\de < 1 \right), \\
		&	(\bII) \quad r_0 e^{\de} \leq A 
			\left( 1 - \de^{\rho -n} \right),   \\
		&	(\bIII) \quad \left( 2^{\frac{\rho - n}{2}} - 
			\de^{\frac{\rho - n}{2}} \right) 
			\de^{\frac{\rho - n}{2}} A 
			\leq r_0 e^{\de},\\
		&	(\bIV) \quad 2\sqrt{2\de}
			\left( \de^{\rho - n} A + r_0 e^{\de} \right) 
			\leq
			r_0 e^{\de}, \quad \text{and} \\
		&	(\bV) \quad \sqrt{2\de} \left[ 2^n 
			\left( \de^{\rho - n} A
			+ r_0 \de e^{\de}
			\right) + 
			2r_0 e^{\de} \left(
			\frac{1 - (2\de)^n}{1-2\de}\right)
			\right] \leq 
			r_0 e^{\de}.
		\end{aligned}
		\right.
\eeq
We now consider a fixed choice of  
$\de \in [0,\de_{\ast}]$ and establish the estimate
claimed in \eqref{lip_k_ball_control_r0>0_conc} 
for $\Omega := \ovB_V(p,\de) ~\cap~ \Gamma$.
We begin by estimating the terms $E_{n-s}$ for
$s \in \{0, \ldots , n-1\}$. We first prove, for 
every $t \in \{0 , \ldots , n-1\}$, that
\beq
	\label{lip_k_ball_control_r0>0_E_n-t}
		E_{n-t} =
		\left( 1 + (2\de)^{\frac{\rho-n}{2}} \right)
		\left( 1 + \sqrt{2\de} \right)^t 
		\left( \de^{\rho - n + t} A
		+ r_0 \de^t e^{\de}
		\right) +
		\bX_t(\de)
\eeq
where $\bX_t(\de)$ is the quantity defined in
\eqref{lip_k_ball_norm_control_r0>0_bX_t_statement}.
That is, 
\beq
	\label{lip_k_ball_norm_control_r0>0_bX_t}
		\bX_t(\de) :=
		\twopartdef
		{0}{t=0}
		{\left( 1 + \sqrt{2\de}\right)
		r_0e^{\de}
		\sum_{j=0}^{t-1} \de^{j} 
		\left(1 + \sqrt{2\de}\right)^{j}}{t \geq 1.}
\eeq
We begin by considering $t:=0$.
From \eqref{lip_k_ball_control_Es} we have that
\beq
	\label{lip_k_ball_control_En_r0>0_a}
		E_n = \left( 1 + (2\de)^{\frac{\rho-n}{2}} \right) 
		\max \left\{ (2\de)^{\frac{\rho - n}{2}} A ~,~
		\min \left\{ A ~,~
		\de^{\rho - n} A + r_0 e^{\de}
		\right\} \right\}.
\eeq
A consequence of (\bII) in 
\eqref{lip_k_ball_control_de<de_star_conds}
is that $\de^{\rho - n} A + r_0 e^{\de} \leq A$ so 
that from \eqref{lip_k_ball_control_En_r0>0_a} we see that 
\beq
	\label{lip_k_ball_control_En_r0>0_b}
		E_n = \left( 1 + (2\de)^{\frac{\rho-n}{2}} \right) 
		\max \left\{ (2\de)^{\frac{\rho - n}{2}} A ~,~
		\de^{\rho - n} A + r_0 e^{\de} \right\}.
\eeq
A consequence of (\bIII) in 
\eqref{lip_k_ball_control_de<de_star_conds}
is that 
$( 2\de)^{\frac{\rho - n}{2}} A \leq \de^{\rho - n} A
+ r_0 e^{\de} $ so 
that from \eqref{lip_k_ball_control_En_r0>0_b} we see that 
\beq
	\label{lip_k_ball_control_En_r0>0_c}
		E_n = 
		\left( 1 + (2\de)^{\frac{\rho-n}{2}} \right)
		\left( \de^{\rho - n} A + r_0 e^{\de} \right)
\eeq
which is the estimate claimed in 
\eqref{lip_k_ball_control_r0>0_E_n-t} for $t=0$ since 
$\bX_0(\de) := 0$.

Now consider $t \geq 1$ and assume that 
\eqref{lip_k_ball_control_r0>0_E_n-t} is true 
for $t-1$.
From \eqref{lip_k_ball_control_Es} we have that
\beq
	\label{lip_k_ball_control_En-t_r0>0_a}
		E_{n-t} = \left( 1 + \sqrt{2\de}\right) 
		\max \left\{ 
		\sqrt{2\de} E_{n-(t-1)} ~,~
		\min \left\{ E_{n-(t-1)} ~,~
		\de E_{n-(t-1)} + r_0 e^{\de}
		\right\} \right\}.
\eeq
Since \eqref{lip_k_ball_control_r0>0_E_n-t}
is valid for $t-1$ we have that
\beq
	\label{lip_k_ball_control_r0>0_En-(t-1)}
		E_{n-(t-1)}
		=
		\left( 1 + (2\de)^{\frac{\rho-n}{2}} \right)
		\left( 1 + \sqrt{2\de} \right)^{t-1} 
		\left( \de^{\rho - n + t-1} A
		+ r_0 \de^{t-1} e^{\de}
		\right) +
		\bX_{t-1}(\de)
\eeq
We claim that $\de E_{n-(t-1)} + r_0 e^{\de} \leq 
E_{n-(t-1)}$. 

If $t = 1$, then 
\eqref{lip_k_ball_control_En_r0>0_c} ensures that 
$(1-\de)E_n \geq (1-\de)
\left( 1 + (2\de)^{\frac{\rho-n}{2}} \right)
r_0 e^{\de}$. If we are able to conclude that
$(1-\de)
\left( 1 + (2\de)^{\frac{\rho-n}{2}} \right) \geq 1$
then our desired estimate 
$\de E_n + r_0 e^{\de} \leq E_n$ is true.
The required lower bound  
$(1-\de)\left(1 + (2\de)^{\frac{\rho-n}{2}} \right)
\geq 1$ is equivalent to 
$(2\de)^{\frac{\rho-n}{2}} - \de -
\de(2\de)^{\frac{\rho-n}{2}} \geq 0$.
A consequence of (\bI) in 
\eqref{lip_k_ball_control_de<de_star_conds}
is that $(2\de)^{\frac{\rho-n}{2}} < 1$.
This tells us that
$\de + \de(2\de)^{\frac{\rho-n}{2}} \leq 2\de 
\leq (2\de)^{\frac{\rho-n}{2}}$
where the latter inequality is true since $2\de < 1$
and $\frac{\rho - n}{2} < 1$. Hence $(1-\de)
\left( 1 + (2\de)^{\frac{\rho-n}{2}} \right) \geq 1$
and so we have that 
$\de E_n + r_0 e^{\de} \leq E_n$ as required.

If $t > 1$ then 
\eqref{lip_k_ball_control_r0>0_En-(t-1)} ensures that
$(1-\de)E_{n-(t-1)} \geq (1-\de)\bX_{t-1}(\de)
\geq (1-\de) \left( 1 + \sqrt{2\de}\right) r_0 e^{\de}$.
If we are able to conclude that 
$(1-\de) \left( 1 + \sqrt{2\de}\right) \geq 1$ then 
our desired estimate 
$\de E_{n-(t-1)} + r_0 e^{\de} \leq E_{n-(t-1)}$ 
is true. The required lower bound
$(1-\de) \left( 1 + \sqrt{2\de}\right) \geq 1$ 
is equivalent to 
$\sqrt{2\de} - \de - \de\sqrt{2\de} \geq 0$.
A consequence of (\bI) in 
\eqref{lip_k_ball_control_de<de_star_conds}
is that $\sqrt{2\de} < 1$.
This tells us that
$\de + \de \sqrt{2\de} \leq 2\de \leq \sqrt{2\de}$
where the latter inequality is true since $2\de < 1$.
Hence $(1-\de) \left( 1 + \sqrt{2\de}\right) \geq 1$ 
and so we have that 
$\de E_{n-(t-1)} + r_0 e^{\de} \leq E_{n-(t-1)}$ 
as required.

Having established that 
$\de E_{n-(t-1)} + r_0 e^{\de} \leq E_{n-(t-1)}$,
\eqref{lip_k_ball_control_En-t_r0>0_a} tells us that  
\beq
	\label{lip_k_ball_control_En-t_r0>0_b}
		E_{n-t} = \left( 1 + \sqrt{2\de}\right) 
		\max \left\{ 
		\sqrt{2\de} E_{n-(t-1)} ~,~
		\de E_{n-(t-1)} + r_0 e^{\de}
		\right\}.
\eeq
We claim that 
$\sqrt{2\de} E_{n-(t-1)} \leq 
\de E_{n-(t-1)} + r_0 e^{\de}$.
If $t = 1$ then we compute, using
\eqref{lip_k_ball_control_r0>0_En-(t-1)} for $t=1$, 
that
\begin{align*}
	\left(\sqrt{2\de} - \de \right) E_{n} &= 
		\sqrt{\de}\left(\sqrt{2} - \sqrt{\de} \right)
		\left( 1 + (2\de)^{\frac{\rho-n}{2}} \right)
		\left( \de^{\rho - n} A + r_0 e^{\de} \right) \\
		&\stackrel{(\bI) ~\text{in}~
		\eqref{lip_k_ball_control_de<de_star_conds}
		}{\leq}
		2 \sqrt{2\de}
		\left( \de^{\rho - n} A + r_0 e^{\de} \right) 
		\stackrel{(\bIV) ~\text{in}~
		\eqref{lip_k_ball_control_de<de_star_conds}
		}
		{\leq}
		r_0 e^{\de}.
\end{align*}
Consequently we have 
$\sqrt{2\de}E_n \leq \de E_n + r_0 e^{\de}$
as claimed.
If $t > 1$ then we compute, using
\eqref{lip_k_ball_control_r0>0_En-(t-1)} for $t>1$, 
that
\begin{align*}
	\left(\sqrt{2\de} - \de \right) E_{n-(t-1)} &=
		\left(\sqrt{2\de} - \de \right) \left(
		\left( 1 + (2\de)^{\frac{\rho-n}{2}} \right)
		\left( 1 + \sqrt{2\de} \right)^{t-1} 
		\left( \de^{\rho - n + t-1} A
		+ r_0 \de^{t-1} e^{\de}
		\right) +
		\bX_{t-1}(\de) \right) \\
		&\stackrel{
		(\bI) ~\text{in}~
		\eqref{lip_k_ball_control_de<de_star_conds} 
		}{\leq} 
		2^t \sqrt{2\de}\left(
		\de^{\rho - n + t-1} A
		+ r_0 \de^{t-1} e^{\de} \right)
		+ 
		\sqrt{2\de} \bX_{t-1}(\de)
		\\
		&= 2^t \sqrt{2\de}\left(
		\de^{\rho - n + t-1} A
		+ r_0 \de^{t-1} e^{\de} \right)
		+ \sqrt{2\de}
		\left( 1 + \sqrt{2\de}\right)
		r_0e^{\de}
		\sum_{j=0}^{t-1} \de^{j} 
		\left(1 + \sqrt{2\de}\right)^{j} \\
		&\stackrel{
		(\bI) ~\text{in}~
		\eqref{lip_k_ball_control_de<de_star_conds}
		}{\leq}
		2^{t}\sqrt{2\de} 
		\left( \de^{\rho - n} A
		+ r_0 \de e^{\de}
		\right) + 2\sqrt{2\de} r_0 e^{\de} 
		\sum_{j=0}^{t-1} (2\de)^j \\
		&\stackrel{ (\bI) ~\text{in}~
		\eqref{lip_k_ball_control_de<de_star_conds}
		}{=}
		\sqrt{2\de} \left[ 2^{t} 
		\left( \de^{\rho - n} A
		+ r_0 \de e^{\de}
		\right) + 
		2r_0 e^{\de} \left(\frac{1 - 
		(2\de)^t}{1-2\de}\right)
		\right] \\
		&\leq 
		\sqrt{2\de} \left[ 2^n 
		\left( \de^{\rho - n} A
		+ r_0 \de e^{\de}
		\right) + 
		2r_0 e^{\de} \left(\frac{1 - 
		(2\de)^n}{1-2\de}\right)
		\right] 
		\stackrel{ (\bIV) ~\text{in}~
		\eqref{lip_k_ball_control_de<de_star_conds}
		}{\leq}
		r_0e^{\de}.
\end{align*}
Consequently we have that 
$\sqrt{2\de} E_{n-(t-1)} \leq \de E_{n-(t-1)} +
r_0 e^{\de}$ as claimed.

Returning our attention to 
\eqref{lip_k_ball_control_En-t_r0>0_b}, 
the inequality 
$\sqrt{2\de} E_{n-(t-1)} \leq \de E_{n-(t-1)} +
r_0 e^{\de}$ means that 
\begin{align*}
	E_{n-t} &= 
		\left( 1 + \sqrt{2\de}\right)
		\left( \de E_{n-(t-1)} + r_0 e^{\de}
		\right) \\
		&\stackrel{
		\eqref{lip_k_ball_control_r0>0_En-(t-1)}}
		{=}
		\left( 1 + \sqrt{2\de}\right)
		\left(  
		\left( 1 + (2\de)^{\frac{\rho-n}{2}} \right)
		\left( 1 + \sqrt{2\de} \right)^{t-1} 
		\left( \de^{\rho - n + t} A
		+ r_0 \de^{t} e^{\de}
		\right) +
		\de \bX_{t-1}(\de)
		+ r_0 e^{\de} \right) \\
		&=
		\left( 1 + (2\de)^{\frac{\rho-n}{2}} \right)
		\left( 1 + \sqrt{2\de} \right)^{t} 
		\left( \de^{\rho - n + t} A
		+ r_0 \de^{t} e^{\de} \right)
		+ \left( 1 + \sqrt{2\de}\right) \left( r_0e^{\de}
		+ \de  \bX_{t-1}(\de) \right).
\end{align*}
We observe that 
\beq
	\label{lip_k_ball_norm_control_r0>0_En_step1}
		\de \left( 1 + \sqrt{2\de}\right) \bX_{t-1}(\de) 
		\stackrel{
		\eqref{lip_k_ball_norm_control_r0>0_bX_t}
		}{=} 
		\twopartdef
		{0}{t-1=0}
		{\left( 1 + \sqrt{2\de}\right) r_0e^{\de}
		\sum_{j=1}^{t} \de^{j} 
		\left(1 + \sqrt{2\de}\right)^{j}}{t -1 \geq 1.}
\eeq
Via \eqref{lip_k_ball_norm_control_r0>0_En_step1}
we see that
\begin{align*}
	\left( 1 + \sqrt{2\de}\right) \left( r_0e^{\de}
	+ \de  \bX_{t-1}(\de) \right)
		&= 
		\twopartdef
		{\left( 1 + \sqrt{2\de}\right) r_0 e^{\de}}{t-1=0}
		{\left( 1 + \sqrt{2\de}\right) r_0e^{\de}
		\sum_{j=0}^{t} \de^{j} 
		\left(1 + \sqrt{2\de}\right)^{j}}{t -1 \geq 1} \\
		&\stackrel{\eqref{lip_k_ball_norm_control_r0>0_bX_t}
		}{=} 
		\twopartdef
		{\bX_1(\de)}{t=1}
		{\bX_t(\de)}{t \geq 2} \\
		&= \bX_t(\de).
\end{align*}
Therefore we have established that 
\beq
	\label{lip_k_ball_control_r0>0_En_form_gotte}
		E_{n-t} 
		=
		\left( 1 + (2\de)^{\frac{\rho-n}{2}} \right)
		\left( 1 + \sqrt{2\de} \right)^{t} 
		\left( \de^{\rho - n + t} A
		+ r_0 \de^{t} e^{\de} \right)
		+
		\bX_t(\de)
\eeq
which is the estimate claimed in 
\eqref{lip_k_ball_control_r0>0_E_n-t} for $t$.
Induction now allows us to conclude that the estimate
\eqref{lip_k_ball_control_r0>0_E_n-t} is valid for 
every $t \in \{0, \ldots , n-1\}$ as claimed.

To conclude, recall that $\th \in (q,q+1]$ and
$b_q := n -(q+1) \in \{0, \ldots , n-1\}$.
Then \eqref{lip_k_ball_control_theta_est} 
yields that
\beq
	\label{lip_k_ball_control_r0>0_lip_th_est_a}
		\left|\left| F_{[q]}
		\right|\right|_{\Lip(\th,\Omega,W)}
		\leq
		\max \left\{ (2\de)^{q+1-\th} \e ~,~
		\min \left\{ \e ~,~
		\de \e + r_0 e^{\de}
		\right\}
		\right\}
\eeq
where, via \eqref{lip_k_ball_control_r0>0_E_n-t} 
for $t := b_q$,
$\e = \e(A,r_0,\rho,\th,\de) := E_{n- b_q}$, i.e.
\beq
	\label{lip_k_ball_control_r0>0_e_def_pf}
		\e 
		\stackrel{
		\eqref{lip_k_ball_control_r0>0_E_n-t}}
		{=}
		\left( 1 + (2\de)^{\frac{\rho-n}{2}} \right)
		\left( 1 + \sqrt{2\de} \right)^{n - (q+1)} 
		\left( \de^{\rho - (q+1)} A
		+ r_0 \de^{n-(q+1)} e^{\de} \right)
		+
		\bX_{n-(q+1)}(\de)
\eeq
as claimed in \eqref{lip_k_ball_control_r0>0_conc} 
and \eqref{lip_k_ball_norm_cal_e_def}.
This completes the proof of Lemma 
\ref{lip_k_func_ext_lemma_gen_eta}.
\end{proof}

\section{Proof of the Pointwise Lipschitz Sandwich Theorem}
\label{lip_k_sand_thms_proofs}
In this section we use the local pointwise Lipschitz estimates 
established in Lemma \ref{lip_k_growth_rates_lemma} from 
Section \ref{local_lip_bds_sec} to establish 
the \textit{Pointwise Lipschitz Sandwich} Theorem 
\ref{lip_k_cover_C0_thm}

\begin{proof}[Proof of Theorem \ref{lip_k_cover_C0_thm}]
Assume that V and W are Banach spaces and that the 
tensor powers of V are all equipped with admissible 
norms (cf. Definition \ref{admissible_tensor_norm}).
Assume $\Sigma \subset V$ is a closed subset.
Let $\ep, \gamma > 0$ with $k \in \Z_{\geq 0}$
such that $\gamma \in (k,k+1]$, 
$(K_1,K_2) \in \left( \R_{\geq 0} \times \R_{\geq 0} \right) 
\setminus \{(0,0)\}$, and 
$0 \leq \ep_0 < \min \left\{ K_1 + K_2 , \ep\right\}$.
Let $l \in \{0, \ldots , k\}$ and define 
$\de_0 = \de_0(\ep, \ep_0, K_1+K_2, \gamma, l) > 0$ by
\beq
    \label{lip_k_ext_C0_thm_de0_def}
        \de_0 := \sup \left\{ \th > 0 ~:~ 
        (K_1 + K_2) \th^{\gamma - l} + \ep_0 e^{\th} 
        \leq 
        \min \left\{ K_1 + K_2 , \ep \right\}
        \right\}.
\eeq
A first consequence of \eqref{lip_k_ext_C0_thm_de0_def}
is that $\de_0 \leq 1$, and so for every 
$s \in \{0, \dots , l\}$ we have that
$\de_0^{\gamma-s} \leq \de_0^{\gamma - l}$.
A second consequence of \eqref{lip_k_ext_C0_thm_de0_def}
is that 
\beq
    \label{lip_k_C0_thm_de0_prop}
        (K_1 + K_2) \de_0^{\gamma - l} + \ep_0 e^{\de_0}
        \leq \min \left\{ K_1 + K_2 , \ep \right\}
        \leq \ep.
\eeq
Now assume that $B \subset \Sigma$ is a $\de_0$-cover
of $\Sigma$ in the sense that 
\beq
    \label{lip_k_cover_cover_prop_C0_pf}
        \Sigma \subset \bigcup_{x \in B}
        \ovB_{V} ( x , \de_0)
        = B_{\de_0} :=
        \left\{ v \in V ~:~ \exists z \in B
        \text{ such that } ||z-v||_V \leq \de_0
        \right\}.
\eeq
Suppose $\psi = \left( \psi^{(0)} , \ldots , 
\psi^{(k)} \right) , \vph = \left( \vph^{(0)} , 
\ldots , \vph^{(k)} \right) \in \Lip(\gamma,\Sigma,W)$ satisfy 
the $\Lip(\gamma,\Sigma,W)$ norm estimates 
$||\psi||_{\Lip(\gamma,\Sigma,W)} \leq K_1$ and  
$||\vph||_{\Lip(\gamma,\Sigma,W)} \leq K_2$.
Further suppose that for every $j \in \{0, \ldots , k\}$ 
and every $x \in B$ the difference 
$\psi^{(j)}(x) - \vph^{(j)}(x) \in \cl(V^{\otimes j};W)$
satisfies the bound
\beq
    \label{lip_k_C0_cover_thm_B_close_assump_pf}
        \left|\left| \psi^{(j)}(x) - \vph^{(j)}(x)
        \right|\right|_{\cl(V^{\otimes j};W)}
        \leq \ep_0.
\eeq
Define $F \in \Lip(\gamma,\Sigma,W)$ by 
$F := \psi - \vph$ so that for every 
$j \in \{0, \ldots , k\}$ we have 
$F^{(j)} := \psi^{(j)} - \vph^{(j)}$.
Then $||F||_{\Lip(\gamma,\Sigma,W)} \leq K_1 + K_2$ and, 
for every integer $j \in \{0, \ldots , k\}$ and every point
$x \in B$, \eqref{lip_k_C0_cover_thm_B_close_assump_pf}
gives that $\left|\left| F^{(j)}(x) 
\right|\right|_{\cl(V^{\otimes j};W)} \leq \ep_0$.

Now fix $x \in \Sigma$ and $s \in \{0, \ldots , l\}$.
From \eqref{lip_k_cover_cover_prop_C0_pf} 
we conclude that there exists a point $z \in B$
with $||z - x||_V \leq \de_0$. Then apply
Lemma \ref{lip_k_growth_rates_lemma}, 
with $A := K_1 + K_2$, $r_0 := \ep_0$, $\rho := \gamma$,
$p := z$, $n := k$ and $q := k$, to conclude that
(cf. \eqref{lip_k_growth_rate_conc})
\begin{multline}
    \label{lip_k_C0_cover_thm_pf_conc}
        \left|\left| F^{(s)}(x) 
        \right|\right|_{\cl(V^{\otimes s};W)} 
        \leq 
        \min \left\{ K_1 + K_2 , (K_1 + K_2) \de_0^{\gamma - s}
        + \ep_0 \sum_{j=0}^{k-s} \frac{1}{j!} \de_0^j
        \right\} \\
        \leq
        \min \left\{ K_1 + K_2 , (K_1 + K_2) \de_0^{\gamma - l}
        + \ep_0 e^{\de_0} \right\}
        \stackrel{
        \eqref{lip_k_C0_thm_de0_prop}
        }{\leq}
        \ep
\end{multline}
where we have used that $\de_0^{\gamma-s} \leq 
\de_0^{\gamma -l}$. Since $F = \psi - \vph$,
the arbitrariness of $s \in \{0, \ldots , l\}$
and of $x \in \Sigma$ ensure that 
\eqref{lip_k_C0_cover_thm_pf_conc} gives the 
bounds claimed in \eqref{lip_k_cover_C0_conc}
This completes the proof of Theorem 
\ref{lip_k_cover_C0_thm}.
\end{proof}

\section{Proof of the Single-Point Lipschitz Sandwich Theorem}
\label{single_point_lip_sand_thm_pf_sec}
In this section we prove the 
\textit{Single-Point Lipschitz Sandwich} Theorem
\ref{lip_k_ball_estimates_thm}. 
Our approach is to alter the constant $\de_0$ 
appearing in the \textit{Pointwise Lipschitz Sandwich} 
Theorem \ref{lip_k_cover_C0_thm} in order to 
strengthen the conclusions to an estimate on the full 
$\Lip(\eta)$-norm of the difference.

To be more precise, recall that $\Sigma \subset V$ 
is closed and $\gamma > 0$ with $k \in \Z_{\geq 0}$ 
such that $\gamma \in (k,k+1]$.
Let $\eta \in (0, \gamma)$, $\ep > 0$, 
$(K_1 , K_2) \in \left( \R_{\geq 0} \times \R_{\geq 0} \right) 
\setminus \{(0,0)\}$, 
and $0 \leq \ep_0 < \min \left\{ K_1 + K_2 , \ep\right\}$.
Retrieve the constant 
$\de_0 = \de_0 (K_1+K_2,\ep,\ep_0,\gamma) > 0$ arising in 
Theorem \ref{lip_k_cover_C0_thm} for the choice 
$l := k$. Given a point $p \in \Sigma$, define 
$\Omega := \Sigma ~\cap~ \ovB_V(p,\de_0)$.

Suppose
$\psi = \left(\psi^{(0)} , \ldots , 
\psi^{(k)} \right) \in \Lip(\gamma,\Sigma,W) ,
\vph = \left( \vph^{(0)} , \ldots , \vph^{(k)}
\right) \in \Lip(\gamma,\Sigma,W)$ satisfy the norm 
bounds $||\psi||_{\Lip(\gamma,\Sigma,W)} \leq K_1$ and  
$||\vph||_{\Lip(\gamma,\Sigma,W)} \leq K_2$.
Further suppose that for every $j \in \{0, \ldots ,k\}$
the difference $\psi^{(j)}(p) - \vph^{(j)}(p) \in 
\cl(V^{\otimes j};W)$ satisfies 
$\left|\left| \psi^{(j)}(p) - \vph^{(j)}(p)
\right|\right|_{\cl(V^{\otimes j};W)} \leq \ep_0$.
Then by applying Theorem \ref{lip_k_cover_C0_thm},
for the choices $l$, $\Sigma$ and $B$ there as
$k$, $\Omega$ and $\{p\}$ here respectively,
we may conclude that for every $s \in \{0, \ldots , k\}$
and every $x \in \Omega$ we have 
$\left|\left| \psi^{(s)}(x) - \vph^{(s)}(x)
\right|\right|_{\cl(V^{\otimes s};W)} \leq \ep$.

We will prove the 
\textit{Single-Point Lipschitz Sandwich Theorem} 
\ref{lip_k_ball_estimates_thm}, by establishing that
after reducing the constant $\de_0$, allowing 
it to additionally depend on $\eta$, we may strengthen 
these pointwise bounds into a bound on the full
$\Lip(\eta,\Omega,W)$ norm of $\psi - \vph$.
We do so by appealing to the local Lipschitz estimates 
established in Lemmas \ref{lip_k_lip_norm_ball_bound_lemma} 
and \ref{lip_k_func_ext_lemma_gen_eta} 
in Section \ref{local_lip_bds_sec}.

There is a natural dichotomy within this strategy 
between the case that $\eta \in (k, \gamma)$ 
and the case that $\eta \in (0,k]$. 
We first use Lemma \ref{lip_k_lip_norm_ball_bound_lemma}
to establish the \textit{Single-Point Lipschitz Sandwich}
Theorem \ref{lip_k_ball_estimates_thm} in the simpler case that 
$\eta \in (k,\gamma)$.

\begin{proof}[Proof of Theorem 
\ref{lip_k_ball_estimates_thm} when $\eta \in (k,\gamma)$]
Assume that V and W are Banach spaces and that the 
tensor powers of V are all equipped with admissible 
norms (cf. Definition \ref{admissible_tensor_norm}).
Let $\Sigma \subset V$ be non-empty and closed.
Let $\ep, \gamma > 0$ with 
$k \in \Z_{\geq 0}$ such that $\gamma \in (k,k+1]$,
$(K_1 , K_2) \in \left( \R_{\geq 0} \times \R_{\geq 0} \right) 
\setminus \{(0,0)\}$, and 
$0 \leq \ep_0 < \min \left\{ K_1 + K_2 , \ep \right\}$, and 
$\eta \in (k,\gamma)$.
For notational convenience we let $K_0 := K_1 + K_2$.
With a view to later applying Lemma 
\ref{lip_k_lip_norm_ball_bound_lemma}, define $\de_0 = 
\de_0(K_0,\ep,\ep_0, \gamma ) > 0$ by
\beq
    \label{lip_k_ball_estimates_easy_proof_de0_def}
        \de_0 := \sup \left\{ \th \in (0,1] ~:~
        K_0(2\th)^{\gamma - \eta}  \leq 
        \min \left\{ K_0 , \ep \right\}
        \text{ and }
        K_0 \th^{\gamma - k} + \ep_0 e^{\th}
        \leq \min \left\{ K_0 , \ep \right\}
        \right\}
        > 0.
\eeq
It initially appears that $\de_0$ additionally depends on $k$.
However, $k$ is determined by $\gamma$, thus any 
dependence on $k$ is really dependence on $\gamma$.
We now fix the value of $\de_0 > 0$ for the remainder 
of the proof. We record that
\eqref{lip_k_ball_estimates_easy_proof_de0_def}
ensures that $\de_0 \leq 1$ and 
\beq
    \label{lip_k_ball_estimates_easy_proof_de0_def_B}
        (\bI) \quad 
        K_0(2 \de_0)^{\gamma - \eta} 
        \leq 
        \min \left\{ K_0 , \ep \right\}
        \qquad \text{and} \qquad
        (\bII) \quad
        K_0 \de_0^{\gamma - k} + \ep_0 e^{\de_0}
        \leq
        \min \left\{ K_0 , \ep \right\}.
\eeq
Now assume that $p \in \Sigma$ and that
$\psi = \left(\psi^{(0)} ,\ldots ,\psi^{(k)}\right)$
and 
$\vph = \left(\vph^{(0)} ,\ldots ,\vph^{(k)}\right)$
are elements in $\Lip(\gamma,\Sigma,W)$ with 
$||\psi||_{\Lip(\gamma,\Sigma,W)} \leq K_1$ and
$||\vph||_{\Lip(\gamma,\Sigma,W)} \leq K_2$.
Further suppose that for every $l \in \{0, \ldots , k\}$
the difference 
$\psi^{(l)}(p) - \vph^{(l)}(p) \in \cl(V^{\otimes l};W)$
satisfies the bound
\beq
    \label{lip_k_sand_thm_B_close_assump_pf}
        \left|\left| \psi^{(l)}(p) - \vph^{(l)}(p) 
        \right|\right|_{\cl(V^{\otimes l};W)} 
        \leq \ep_0.
\eeq
Define $\Omega := \ovB_V(p,\de_0) ~\cap~ \Sigma$
and $F \in \Lip(\gamma,\Sigma,W)$ by $F := \psi - \vph$,
so that for every $j \in \{0, \ldots , k\}$ we have 
$F^{(j)} = \psi^{(j)} - \vph^{(j)}$. We apply Lemma
\ref{lip_k_lip_norm_ball_bound_lemma} to $F$,
with $A:=K_0 = K_1 + K_2$, $r_0:=\ep_0$, $\rho := \gamma$, $n:=k$,
$z := p$, $\th := \eta$ and $\de := \de_0$, to conclude 
both that $F \in \Lip(\eta,\Omega,W)$ and
(cf. \eqref{lip_k_lip_norm_ball_bound_lemma_conc}) that
\begin{align*}
        || F ||_{\Lip(\eta ,\Omega,W)} 
        \leq 
        \max \left\{
        K_0(2\de_0)^{\gamma - \eta} ~,~
        \min \left\{ K_0 ~,~ K_0 \de_0^{\gamma - k}
        + \ep_0 e^{\de_0} \right\}
        \right\} 
        &\stackrel{
        (\bII) \text{ of }
        \eqref{lip_k_ball_estimates_easy_proof_de0_def_B}
        }{\leq}
        \max \left\{
        K_0(2\de_0)^{\gamma - \eta} ~,~ \ep
        \right\} \\
	    &\stackrel{ (\bI) \text{ of }
        \eqref{lip_k_ball_estimates_easy_proof_de0_def_B}
        }{\leq}
        \max \left\{ \ep ~,~ \ep\right\}   
        =
        \ep.
\end{align*}
Since $F = \psi - \vph$ and 
$\Omega := \Sigma ~\cap~ \ovB_V(p,\de_0)$, 
this is precisely the estimate 
claimed in \eqref{lip_k_sing_point_cover_thm_conc}.
This completes the proof of Theorem 
\ref{lip_k_ball_estimates_thm} for $\eta \in (k,\gamma)$.
\end{proof}
\vskip 4pt
\noindent
We now turn our attention to using Lemma 
\ref{lip_k_func_ext_lemma_gen_eta} to establish 
the \textit{Single-Point Lipschitz Sandwich} Theorem 
\ref{lip_k_ball_estimates_thm} in the more 
challenging case that $\eta \in (0,k]$.

\begin{proof}[Proof of Theorem 
\ref{lip_k_ball_estimates_thm} for $0 < \eta \leq k$]
Assume that V and W are Banach spaces and that the 
tensor powers of V are all equipped with admissible 
norms (cf. Definition \ref{admissible_tensor_norm}).
Let $\Sigma \subset V$ be closed and non-empty.
Let $\ep , \gamma > 0$ with $k \in \Z_{\geq 1}$ 
such that $\gamma \in (k,k+1]$, 
$(K_1 , K_2) \in \left( \R_{\geq 0} \times \R_{\geq 0} \right) 
\setminus \{(0,0)\}$, and $\eta \in (0,k]$.
Observe that this requires $1 \leq k < \gamma$.
Let $q \in \{0, \ldots , k-1\}$ such that 
$\eta \in (q,q+1] \subset (0,k]$. Finally let 
$0 \leq \ep_0 < \min \left\{ K_1 + K_2 , \ep \right\}$.
For notational convenience, we set $K_0 := K_1 + K_2 > 0$ 
for the remainder of the proof

Our strategy is to establish the desired
$\Lip(\eta)$-norm bounds via an application of Lemma
\ref{lip_k_func_ext_lemma_gen_eta}. For this purpose 
we retrieve the constant $\de_{\ast}$ arising in 
Lemma \ref{lip_k_func_ext_lemma_gen_eta}
for $A := K_0 = K_1 + K_2$, $r_0 := \ep_0$, 
$\rho := \gamma$ and $\th := \eta$.
Note that we are not actually applying Lemma 
\ref{lip_k_func_ext_lemma_gen_eta}, but simply
retrieving a constant in preparation for its future application.

Let $\de_0 := \min \left\{ 1 ~,~\de_{\ast} \right\} > 0$, 
which depends only
on $K_0 = K_1 + K_2$, $\ep_0$, $\gamma$ and $\eta$.
In order to ensure that applying Lemma 
\ref{lip_k_func_ext_lemma_gen_eta} yields the  
desired $\Lip(\eta)$-norm estimate, we will allow
ourselves to (potentially) further reduce $\de_0$, 
additionally now depending on $\ep$. With the benefit of 
hindsight, it will suffice to alter $\de_0$ to ensure that 
\beq
	\label{lip_k_ball_estimates_thm_de0}
		\left\{
		\begin{aligned}
		&	(\bA) \quad 
			\max \left\{ 1 + \left( 2 \de_0 
			\right)^{\frac{\gamma - k}{2}}
			~,~ 1 + \sqrt{2\de_0} \right\}
			< 2, \quad \left( \text{In particular, } 
			2\de_0 < 1\right),\\
		&	(\bB) \quad 
			\left( 2 \de_0 
			\right)^{\frac{\gamma - \eta}{2} 
			+ \frac{q+1 - \eta}{2}}
			\leq
			\frac{\ep}{2^{k-q}K_0}, \\
		&	(\bC) \quad
			\left( 1 + (2\de_0)^{\frac{\gamma - k}{2}}
			\right) 
			\left( \de_0^{\gamma - k} K_0 + 
			\ep_0 e^{\de_0}
			\right) \leq \ep, \\
		&	(\bD) \quad
			2^{k-q} \left( \de_0^{\gamma - k} K_0
			+ \ep_0 \de_0 e^{\de_0} \right)
			+ \left(\frac{1 + \sqrt{2\de_0}}{1 - 2\de_0}
			\right)
			\ep_0 e^{\de_0}
			\leq \ep, \quad \text{and} \\
		&	(\bE) \quad
			\ep_0 e^{\de_0} \leq 
			\left( 1 - \de_0 \right) \ep.
		\end{aligned}
		\right.
\eeq
We now fix the value of 
$\de_0 = \de_0(K_0,\gamma,\eta,\ep_0 ,\ep) > 0$
for the remainder of the proof.
Since $K_0 := K_1 + K_2$ we note that $\de_0$ has the 
claimed dependencies.

Now let $p \in \Sigma$ and assume that $\psi = 
\left( \psi^{(0)} , \ldots , \psi^{(k)} \right)$
and $\vph = \left( \vph^{(0)} , \ldots , \vph^{(k)}
\right)$ are in $\Lip(\gamma,\Sigma,W)$ 
with 
$|| \psi ||_{\Lip(\gamma,\Sigma,W)} \leq K_1$ and 
$|| \vph ||_{\Lip(\gamma,\Sigma,W)} \leq K_2$. 
Suppose that for every integer $l \in \{0, \ldots , k\}$
the difference $\psi^{(l)}(p) - \vph^{(l)}(p) 
\in \cl(V^{\otimes l};W)$ satisfies the bound
\beq
    \label{lip_k_sing_point_sand_thm_p_close_assump}
        \left|\left| \psi^{(l)}(p) - \vph^{(l)}(p)
        \right|\right|_{\cl(V^{\otimes l};W)}
        \leq \ep_0.
\eeq
Define $\Omega := \ovB_V(p,\de_0) ~\cap~ \Sigma$
and $F \in \Lip(\gamma,\Sigma,W)$ by 
$F := \psi - \vph$ so that for every 
$j \in \{0, \ldots , k\}$ we have 
$F^{(j)} = \psi^{(j)} - \vph^{(j)}$.

We begin with the case that $\ep_0 = 0$.
Since $\de_0 \leq 1$, the bounds
\eqref{lip_k_sing_point_sand_thm_p_close_assump} 
allow us to apply
Lemma \ref{lip_k_func_ext_lemma_gen_eta} to $F$, with
$A := K_0 (=K_1+K_2)$, $r_0 := \ep_0$, $\rho := \gamma$, 
$\th := \eta$ and $\de := \de_0$, to conclude that 
(cf. \eqref{lip_k_ball_control_r0=0_conc})
\beq
	\label{lip_k_ball_estimates_thm_ep0=0_est_A}
		\left|\left| F_{[q]} 
		\right|\right|_{\Lip(\eta,\Omega,W)} 
		\leq 
		\left( 1 + (2\de_0)^{\frac{\gamma - k}{2}} \right) 
		\left( 1+\sqrt{2\de_0} \right)^{k-(q+1)}
		(2\de_0)^{\frac{\gamma - \eta}{2} + 
		\frac{q+1-\eta}{2}} K_0.
\eeq
We compute that
\begin{align*} 
	\left|\left| F_{[q]} 
	\right|\right|_{\Lip(\eta,\Omega,W)}
	&\stackrel{
	\eqref{lip_k_ball_estimates_thm_ep0=0_est_A}
	}{\leq}
	\left( 1 + (2\de_0)^{\frac{\gamma - k}{2}} \right) 
	\left( 1+\sqrt{2\de_0} \right)^{k-(q+1)}
	(2\de_0)^{\frac{\gamma - \eta}{2} + 
	\frac{q+1-\eta}{2}} K_0 \\
	&\stackrel{
	(\bA) ~\text{in}~\eqref{lip_k_ball_estimates_thm_de0}
	}{\leq} 2^{k-q}
	(2\de_0)^{\frac{\gamma - \eta}{2} + 
	\frac{q+1-\eta}{2}} K_0
	\stackrel{
	(\bB) ~\text{in}~
	\eqref{lip_k_ball_estimates_thm_de0}
	}{\leq}
	2^{k-q}
	\frac{\ep}{2^{k-q}K_0} K_0 = \ep.
\end{align*}
Recalling that 
$\Omega := \Sigma ~\cap~ \ovB_V(p,\de_0)$
and $F := \psi - \vph$, this is precisely 
the estimate claimed in 
\eqref{lip_k_sing_point_cover_thm_conc}, and 
our proof is complete for the case that $\ep_0 = 0$.

Now consider the case that $\ep_0 > 0$. 
Recalling how we chose $\de_0$, the bounds
\eqref{lip_k_sing_point_sand_thm_p_close_assump} 
allow us to apply Lemma
\ref{lip_k_func_ext_lemma_gen_eta} to $F$, with 
$A := K_0(=K_1+K_2)$, $r_0 := \ep_0$, $\rho := \gamma$, 
$\th := \eta$ and $\de := \de_0$, to conclude 
via \eqref{lip_k_ball_control_r0>0_conc} that
\beq
	\label{lip_k_ball_estimate_thm_r0>0_est_A}
		\left|\left| F_{[q]} 
		\right|\right|_{\Lip(\eta,\Omega,W)} 
		\leq
		\max \left\{ (2\de_0)^{q+1-\eta} \e
		~,~
		\min \left\{ \e ~,~ \de_0 \e 
		+ \ep_0 e^{\de_0}
		\right\} \right\}
\eeq
for $\e = \e(K_0,\gamma,\eta,\ep_0)>0$ defined by
(cf. \eqref{lip_k_ball_norm_cal_e_def})
\beq
	\label{lip_k_ball_estimate_thm_cal_e_def}
		\e := \left( 
		1 + (2\de_0)^{\frac{\gamma - k}{2}}\right)
		\left( 1 + \sqrt{2\de_0} \right)^{k - (q+1)} 
		\left( \de_0^{\gamma - (q+1)} K_0
		+ \ep_0 \de_0^{k-(q+1)} e^{\de_0} \right)
		+
		\bX_{k-(q+1)}(\de_0)
\eeq
where, for $t \in \{0 , \ldots , k-1 \}$, 
the quantity $\bX_t(\de)$ is defined by (cf. 
\eqref{lip_k_ball_norm_control_r0>0_bX_t_statement})
\beq
	\label{lip_k_ball_estimate_thm_r0>0_bX_t}
		\bX_t(\de) := \twopartdef {0}{t=0}
		{\left( 1 + \sqrt{2\de_0}\right)
		\ep_0e^{\de_0}
		\sum_{j=0}^{t-1} \de_0^{j} 
		\left(1 + \sqrt{2\de_0}\right)^{j}}{t \geq 1.}
\eeq
We first prove that 
$\left( 1 - \de_0 \right) \e \geq \ep_0 e^{\de_0}$.

If $k = q+1$, then
\eqref{lip_k_ball_estimate_thm_cal_e_def} ensures
that $\left(1-\de_0\right) \e \geq 
\left( 1 - \de_0 \right) 
\left( 1 + (2\de_0)^{\frac{\gamma -k}{2}}\right)
\ep_0 e^{\de_0}$.
If we are able to conclude that 
$\left( 1 - \de_0 \right) 
\left( 1 + (2\de_0)^{\frac{\gamma -k}{2}}\right)
\geq 1$ then our desired estimate 
$\left( 1 - \de_0 \right) \e \geq \ep_0 e^{\de_0}$
is true. The required lower bound 
$\left( 1 - \de_0 \right) 
\left( 1 + (2\de_0)^{\frac{\gamma -k}{2}}\right)
\geq 1$ is equivalent to 
$(2\de_0)^{\frac{\gamma -k}{2}} - \de_0 
- \de_0 (2\de_0)^{\frac{\gamma -k}{2}} \geq 0$.
A consequence of (\bA) in 
\eqref{lip_k_ball_estimates_thm_de0} is that 
$(2\de_0)^{\frac{\gamma -k}{2}} < 1$.
This tells us that 
$\de_0 + \de_0 (2\de_0)^{\frac{\gamma -k}{2}}
\leq 2 \de_0 \leq (2\de_0)^{\frac{\gamma -k}{2}}$
where the latter inequality is true since 
$2\de_0 < 1$ and $\frac{\gamma-k}{2} < 1$.
Hence 
$\left( 1 - \de_0 \right) 
\left( 1 + (2\de_0)^{\frac{\gamma -k}{2}}\right)
\geq 1$ and so we have 
$\left( 1 - \de_0 \right) \e \geq \ep_0 e^{\de_0}$
as required.

If $k > q +1$, then  
\eqref{lip_k_ball_estimate_thm_cal_e_def} 
and \eqref{lip_k_ball_estimate_thm_r0>0_bX_t} 
yield $\left(1-\de_0\right) \e \geq 
\left( 1 - \de_0 \right) \bX_{k-(q+1)}(\de_0)
\geq \left( 1 - \de_0 \right) 
\left( 1 + \sqrt{2\de_0} \right) \ep_0 e^{\de_0}$.
If we are able to conclude that 
$\left( 1 - \de_0 \right) 
\left( 1 + \sqrt{2\de_0} \right)
\geq 1$ then our desired estimate 
$\left( 1 - \de_0 \right) \e \geq \ep_0 e^{\de_0}$
is true. The required lower bound 
$\left( 1 - \de_0 \right) 
\left( 1 + \sqrt{2\de_0} \right)
\geq 1$ is equivalent to 
$\sqrt{2\de_0} - \de_0 - \de_0 \sqrt{2\de_0} \geq 0$.
A consequence of (\bA) in 
\eqref{lip_k_ball_estimates_thm_de0} is that 
$\sqrt{2\de_0} < 1$.
This tells us that $\de_0 + \de_0 \sqrt{2\de_0} 
\leq 2 \de_0 \leq \sqrt{2\de_0}$
where the latter inequality is true since $2\de_0 < 1$.
Hence 
$\left( 1 - \de_0 \right) \left( 1 + \sqrt{2\de_0}\right)
\geq 1$ and so we have 
$\left( 1 - \de_0 \right) \e \geq \ep_0 e^{\de_0}$
as required.

Having established that 
$\left(1-\de_0\right) \e \geq \ep_0 e^{\de_0}$
we observe that 
\eqref{lip_k_ball_estimate_thm_r0>0_est_A}
becomes
\beq
	\label{lip_k_ball_estimate_thm_r0>0_est_B}
		\left|\left| F_{[q]} 
		\right|\right|_{\Lip(\eta,\Omega,W)}
		\leq
		\max \left\{ (2\de_0)^{q+1-\eta} \e
		~,~ \de_0 \e + \ep_0 e^{\de_0}
		\right\}
\eeq
We now prove the upper bound for $\e$ that 
$\e \leq \ep$. For this purpose note that when 
$k = q+1$ \eqref{lip_k_ball_estimate_thm_cal_e_def} 
yields that
\beq
	\label{lip_k_ball_estimate_thm_cal_e_ub_A}
		\e =
		\left( 1 + (2\de_0)^{\frac{\gamma - k}{2}}
		\right) 
		\left( \de_0^{\gamma - k} K_0 + 
		\ep_0 e^{\de_0} \right)
		\stackrel{(\bC) ~\text{in}~
		\eqref{lip_k_ball_estimates_thm_de0}
		}{\leq}
		\ep
\eeq
since $\bX_0(\de_0)=0$ from 
\eqref{lip_k_ball_estimate_thm_r0>0_bX_t}.
If, however, $k > q +1$ then $k - (q+1) \geq 1$ and so,
recalling that (\bA) in 
\eqref{lip_k_ball_estimates_thm_de0} ensures that 
that $\de_0(1+\sqrt{2\de_0}) < 2 \de_0 < 1$, 
we have 
\beq
	\label{lip_k_ball_estimate_thm_cal_e_ub_k>q+1_est_A}
		\bX_{k - (q+1)}(\de_0)
		\stackrel{
		\eqref{lip_k_ball_estimate_thm_r0>0_bX_t}}{=}
		\left( 1 + \sqrt{2\de_0} \right) \ep_0 e^{\de_0}
		\sum_{j=0}^{k-(q+1)} \de_0^j
		\left( 1 + \sqrt{2\de_0} \right)^j
		\leq 
		\frac{\left( 1 + \sqrt{2\de_0} \right)} 
		{1 - 2\de_0}
		\ep_0 e^{\de_0}.
\eeq
Moreover, $2\de_0 < 1$ ensures that 
$\de_0^{\gamma - (q+1)} \leq \de_0^{\gamma - k}$ 
and
$\de_0^{k-(q+1)} \leq \de_0$. 
Hence we can 
combine \eqref{lip_k_ball_estimate_thm_cal_e_def} 
and 
\eqref{lip_k_ball_estimate_thm_cal_e_ub_k>q+1_est_A} 
to obtain that
$$\e \leq 2^{k-q} \left(
	\de_0^{\gamma - k} K_0 + 
	\ep_0 \de_0 e^{\de_0}
	\right)
	+
	\frac{\left( 1 + \sqrt{2\de_0} \right)} 
	{1 - 2\de_0}
	\ep_0 e^{\de_0} 
	\stackrel{(\bD) ~\text{in}~
	\eqref{lip_k_ball_estimates_thm_de0}}
	{\leq} \ep.$$
Therefore in both the case that $k=q+1$ and 
the case that $k > q+1$ we obtain that 
\beq
	\label{lip_k_ball_estimate_thm_cal_e_ub}
		\e \leq \ep.
\eeq
We complete the proof by using the upper bound
in \eqref{lip_k_ball_estimate_thm_cal_e_ub} to 
control $\left|\left| F_{[q]} 
\right|\right|_{\Lip(\eta,\Omega,W)}$.
Recalling that (\bA) in 
\eqref{lip_k_ball_estimates_thm_de0}
means that $2 \de_0 < 1$, we have that
\beq
	\label{lip_k_ball_estimate_thm_eta<q+1_A}
		(\bI) \quad
		\left( 2 \de_0 \right)^{q+1-\eta} \e
		\leq
		\e
		\stackrel{
		\eqref{lip_k_ball_estimate_thm_cal_e_ub}
		}{\leq}
		\ep
		\qquad \text{and} \qquad
		(\bII) \quad 
		\de_0 \e + \ep_0 e^{\de_0}
		\stackrel{
		\eqref{lip_k_ball_estimate_thm_cal_e_ub}
		}{\leq}
		\de_0 \ep + \ep_0 e^{\de_0}
		\stackrel{(\bE) ~\text{in}~
		\eqref{lip_k_ball_estimates_thm_de0}}{\leq}
		\ep.
\eeq
Thus
\beq
	\label{lip_k_ball_estimate_thm_r0>0_est_C}
		\left|\left| F_{[q]} 
		\right|\right|_{\Lip(\eta,\Omega,W)} 
		\stackrel{
		\eqref{lip_k_ball_estimate_thm_r0>0_est_B}
		}{\leq}
		\max \left\{ (2\de_0)^{q+1-\eta} \e
		, \de_0 \e + \ep_0 e^{\de_0}
		\right\}
		\stackrel{(\bI) \text{ in }
		\eqref{lip_k_ball_estimate_thm_eta<q+1_A}
		}{\leq}
		\max \left\{  \ep
		, \de_0 \e + \ep_0 e^{\de_0}
		\right\}
		\stackrel{(\bII) \text{ in }
		\eqref{lip_k_ball_estimate_thm_eta<q+1_A}
		}{=} \ep.
\eeq
Recalling that 
$\Omega := \Sigma ~\cap~ \ovB_V(p,\de_0)$
and $F := \psi - \vph$, 
\eqref{lip_k_ball_estimate_thm_r0>0_est_C}
is precisely the estimate claimed in 
\eqref{lip_k_sing_point_cover_thm_conc}, and 
our proof is complete for the case that $\ep_0 > 0$.
Having already established the conclusion for the 
case that $\ep_0 = 0$, this completes the proof of 
Theorem \ref{lip_k_ball_estimates_thm} for the case 
that $\eta \in (0,k]$.
\end{proof}

\section{Proof of the Lipschitz Sandwich Theorem 
\ref{lip_k_cover_thm}}
\label{lip_sand_thm_pf_sec}
In this section we establish the full
\textit{Lipschitz Sandwich Theorem} \ref{lip_k_cover_thm}.
Our strategy to prove this result is to 
patch together the local Lipschitz bounds achieved
by the \textit{Single-Point Lipschitz Sandwich}
Theorem \ref{lip_k_ball_estimates_thm} in a similar 
spirit to the patching of local Lipschitz bounds 
in Lemma 1.16 in \cite{Bou15}.
We do not necessarily have local Lipschitz bounds on a 
small ball centred at \textit{any} point in $\Sigma$; 
we only have such estimates for points in the closed subset
$B \subset \Sigma$, and we do \textit{not} require that $B = \Sigma$.
Consequently, our patching is more complicated than the patching 
used in Lemma 1.16 in \cite{Bou15}.

To be 
more precise, recall that $\Sigma \subset V$ is closed
and $\gamma > 0$ with $k \in \Z_{\geq 0}$ such that
$\gamma \in (k,k+1]$. Let $\eta \in (0,\gamma)$, 
$\ep > 0$, 
$(K_1 , K_2) \in \left( \R_{\geq 0} \times \R_{\geq 0} \right) 
\setminus \{(0,0)\}$, 
and $0 \leq \ep_0 < 
\min \left\{ K_1 + K_2 , \ep \right\}$. Retrieve the 
constant $\de_0 = \de_0(K_1+K_2,\ep,\ep_0,\gamma,\eta) >0$
arising in the \textit{Single-Point Lipschitz Sandwich 
Theorem} \ref{lip_k_ball_estimates_thm}.
Assume that $B \subset \Sigma$ is a $\de_0$-cover of
$\Sigma$ in the sense that the $\de_0$-fattening of
$B$ contains $\Sigma$.

Suppose $\psi = \left( \psi^{(0)} , \ldots , 
\psi^{(k)} \right) \in \Lip(\gamma,\Sigma,W)$
and $\vph \in \left( \vph^{(0)} , \ldots , \vph^{(k)}
\right) \in \Lip(\gamma,\Sigma,W)$ both satisfy the 
norm bounds 
$||\psi||_{\Lip(\gamma,\Sigma,W)} \leq K_1$ and 
$||\vph||_{\Lip(\gamma,\Sigma,W)} \leq K_2$. 
Further suppose that for every $j \in \{0 , \ldots , k\}$
and every $x \in B$ the difference 
$\psi^{(j)}(x) - \vph^{(j)}(x) \in \cl(V^{\otimes j};W)$
satisfies $\left|\left| \psi^{(j)}(x) - \vph^{(j)}(x)
\right|\right|_{\cl(V^{\otimes j};W)} \leq \ep_0$.
Then given any point $p \in B$, we can apply the
\textit{Single-Point Lipschitz Sandwich Theorem} 
\ref{lip_k_ball_estimates_thm} to conclude that 
$|| \psi_{[q]} - \vph_{[q]} 
||_{\Lip(\eta,\Omega_p,W)} \leq \ep$
for $\Omega_p := \Sigma ~\cap~ \ovB_V(p,\de_0)$
and $q \in \{0, \ldots , k\}$ such that 
$\eta \in (q,q+1]$.

It may initially appear that since 
$\Sigma = \cup_{p \in B} \Omega_p$ these local 
$\Lip(\eta)$-norm bounds should combine together to
yield $||\psi_{[q]} - \vph_{[q]}
||_{\Lip(\eta,\Sigma,W)} \leq \ep$.
However, this is not necessarily true. 
For example, given any $\al \in (0,1)$, 
consider the function 
$F : [0,1] \cup [1+\al,2] \to \R$ defined by
$F(x) := 0$ if $x \in [0,1]$ and $F(x) := \al$
if $x \in [1+\al,2]$. Then $F \in \Lip(1,[0,1] \cup 
[1+\al,2],\R)$ and we have that 
$||F||_{\Lip(1,[0,1],\R)} = 0$ and 
$||F||_{\Lip(1,[1+\al,2],\R)} = \al$.
But $|F(1+\al)-F(1)| = \al = |1+\al - 1|$ and so 
$||F||_{\Lip(1,[0,1]\cup[1+\al,2],\R)} = 1 > \al$.

The main content of our proof of the 
\textit{Lipschitz Sandwich Theorem} \ref{lip_k_cover_thm}
is to overcome this problem. We prove that, by 
requiring the constant $\ep_0$ to be sufficiently 
small, depending only on $\ep, K_1+K_2, \gamma$ 
and $\eta$, rather than an arbitrary real number in 
the interval $[0, \min \left\{ K_1 + K_2 , \ep\right\})$,
we \textit{can} patch together local Lipschitz 
estimates resulting from an application of the 
\textit{Single-Point Lipschitz Sandwich Theorem} 
\ref{lip_k_ball_estimates_thm} to yield
global Lipschitz estimates throughout $\Sigma$.
A key point is to ensure that the sets $\Omega_p$ on 
which the \textit{Single-Point Lipschitz Sandwich Theorem} 
\ref{lip_k_ball_estimates_thm} yields local Lipschitz estimates
are not pairwise disjoint; that is, for each $p \in B$
there must be 
some $q \in B \setminus \{p\}$ such that the intersection 
$\Omega_p \cap \Omega_q$ is non-empty.

\begin{proof}[Proof of Theorem 
\ref{lip_k_cover_thm}]
Let $V$ and $W$ be Banach spaces, and assume  
that the tensor powers of 
$V$ are all equipped with admissible norms (cf. 
Definition \ref{admissible_tensor_norm}).
Assume that $\Sigma \subset V$ is non-empty and closed.
Let $\ep , \gamma > 0$ with $k \in \Z_{\geq 0}$
such that $\gamma \in (k,k+1]$, and 
$(K_1 , K_2) \in \left( \R_{\geq 0} \times \R_{\geq 0} \right) 
\setminus \{(0,0)\}$.
Further let $\eta \in (0, \gamma)$ with 
$q \in \{0, \ldots , k\}$ such that $\eta \in (q,q+1]$.
It suffices to prove the theorem under the additional
assumption that $\ep \leq K_0 := K_1 + K_2$; the conclusion 
\eqref{lip_k_cover_thm_conc} being valid for 
$\ep$ immediately means it is also valid for any
constant $\ep' \geq \ep$.

Define $\th := \frac{1}{2(1+e)} > 0$ and 
retrieve the constant $\de_0 >0$ arising from Theorem 
\ref{lip_k_ball_estimates_thm} for the same constants 
$K_1$, $K_2$, $\gamma$ and $\eta$ as here respectively,
and with the choices of $\th \ep$ and $\frac{\th}{2}\ep$ 
here as the constants $\ep$ and $\ep_0$ in 
Theorem \ref{lip_k_ball_estimates_thm} respectively.
Note that we are not actually applying 
Theorem \ref{lip_k_ball_estimates_thm}, but simply 
retrieving a constant in preparation for its future 
application. Examining the dependencies in Theorem 
\ref{lip_k_ball_estimates_thm} reveals that $\de_0 > 0$ 
depends only on $\ep$, $K_0$, $\gamma$ and $\eta$. 
If necessary, we reduce $\de_0$, without additional 
dependencies, so that $\de_0 \leq 1$. Finally, 
we replace the resulting constant $\de_0$ by $\de_0/2$. 

Our choice of $\de_0 > 0$ means that
if $\psi = \left(\psi^{(0)} , \ldots , \psi^{(k)}
\right)~,~ \vph = \left( \vph^{(0)} , \ldots , 
\vph^{(k)} \right) \in \Lip(\gamma,\Sigma,W)$ with
$|| \psi ||_{\Lip(\gamma,\Sigma,W)} \leq K_1$ and 
$|| \vph ||_{\Lip(\gamma,\Sigma,W)} \leq K_2$, and
if for a point $p \in \Sigma$ and every 
$l \in \{0, \ldots , k\}$ we have the estimate
$\left|\left| \psi^{(l)}(p) - \vph^{(l)}(p) 
\right|\right|_{\cl(V^{\otimes l};W)} \leq 
\frac{\th}{2}\ep$, then an application of 
Theorem \ref{lip_k_ball_estimates_thm} would allow 
us to conclude the estimate that
$\left|\left| \psi_{[q]} - \vph_{[q]} 
\right|\right|_{\Lip(\eta,\Omega_p,W)} \leq \th \ep$ 
for $\Omega_p := \ovB_V(p,2\de_0) ~\cap~ \Sigma$.

We now fix the value of $\de_0 > 0$ for the remainder
of the proof. Having done so, we define $\ep_0 > 0$ by 
\beq
	\label{lip_k_cover_thm_ep0_def}
		\ep_0 :=
		\min \left\{ \th ~,~
		\frac{\de_0^{\eta}}
		{e^{\de_0} \left( 1 + e^{\de_0}
		\right)} \right\}
		\frac{\ep}{2}> 0.
\eeq
Examining the dependencies in 
\eqref{lip_k_cover_thm_ep0_def} reveals
that $\ep_0$ depends only on $\ep$, $K_0 = K_1 + K_2$,
$\gamma$ and $\eta$. We may now fix the value of
$\ep_0 > 0$ for the the remainder of the proof.

Let $B \subset \Sigma$ satisfy that
\beq
    \label{lip_k_cover_thm_cover_prop}
        \Sigma \subset \bigcup_{x \in B}
        \ovB_{V} ( x , \de_0).
\eeq
Suppose 
$\psi = \left(\psi^{(0)} , \ldots , \psi^{(k)}
\right) , \vph = \left( \vph^{(0)} , \ldots , 
\vph^{(k)} \right) \in \Lip(\gamma,\Sigma,W)$ satisfy the 
$\Lip(\gamma,\Sigma,W)$ norm estimates
$|| \psi ||_{\Lip(\gamma,\Sigma,W)} \leq K_1$ and 
$|| \vph ||_{\Lip(\gamma,\Sigma,W)} \leq K_2$.
Further assume that whenever $l \in \{0,\ldots,k\}$
and $x \in B$ we have the estimate
$\left|\left| \psi^{(l)}(x) - \vph^{(l)}(x) 
\right|\right|_{\cl(V^{\otimes l};W)} \leq \ep_0$.
Let $p \in B$.
Recalling how we chose the constant $\de_0 > 0$ and 
that \eqref{lip_k_cover_thm_ep0_def} 
means that $\ep_0 \leq \frac{\th}{2}\ep$,
we may appeal to Theorem \ref{lip_k_ball_estimates_thm}
to conclude that
\beq
    \label{lip_k_cover_thm_bd_1}
        \left|\left| \psi_{[q]} - \vph_{[q]} 
        \right|\right|_{\Lip(\eta, \Omega_p,W)}
        \leq
        \th \ep
\eeq
where $\Omega_p := \Sigma ~\cap~ \ovB_V (p,2\de_0)$.
The arbitrariness of $p \in B$ allows us to conclude 
that the estimate \eqref{lip_k_cover_thm_bd_1} is valid 
for every $p \in B$.

We complete the proof of Theorem 
\ref{lip_k_cover_thm} by establishing 
that having the bounds 
\eqref{lip_k_cover_thm_bd_1} for every
$p \in B$ allows us to conclude that
$\left|\left| \psi_{[q]} - \vph_{[q]} 
\right|\right|_{\Lip(\eta,\Sigma,W)} \leq \ep$.
This is proven in the following claim.

\begin{claim}
\label{lip_k_cover_thm_claim}
If $F = \left( F^{(0)} , \ldots , F^{(k)} \right) 
\in \Lip(\gamma,\Sigma,W)$ satisfies, for every 
$l \in \{0,\ldots,k\}$ and every $z \in B$, that 
$\left|\left| F^{(l)}(z)  
\right|\right|_{\cl(V^{\otimes l};W)} \leq \ep_0$
and
$\left|\left| F_{[q]} 
\right|\right|_{\Lip(\eta, \Omega_z,W)} 
\leq \th \ep$, where 
$\Omega_z := \Sigma ~\cap~
\ovB_V (z,2\de_0)$, then we have 
\beq
    \label{lip_k_cover_thm_claim_conc}
        \left|\left|F_{[q]} 
        \right|\right|_{\Lip(\eta,\Sigma,W)}
        \leq \ep. 
\eeq
\end{claim}

\begin{proof}[Proof of Claim
\ref{lip_k_cover_thm_claim}]
For each $l \in \{0, \ldots , k\}$ let
$R^{F}_l : \Sigma \times \Sigma \to
\cl( V^{\otimes l} ; W) $ denote the remainder term
associated to $F^{(l)}$. 
Therefore whenever $l \in \{0, \ldots , k \}$,
$x,y \in \Sigma$ and $v \in V^{\otimes l}$, we have that
(cf. \eqref{lip_k_tay_expansion})
\beq
    \label{lip_k_cover_thm_claim_expansions_1}
        R^F_l(x,y)[v] := 
        F^{(l)}(y)[v]
        -
        \sum_{s=0}^{k-l} \frac{1}{s!}
        F^{(l+s)}(x) \left[ 
        v \otimes (y-x)^{\otimes j}
        \right].
\eeq
If $q = k$ then we may work with the unaltered 
remainder terms defined in 
\eqref{lip_k_cover_thm_claim_expansions_1}.
But if $q < k$ then we must first appropriately
alter the remainder terms. 
For this purpose, for each $l \in \{0, \ldots , q\}$
we define $\hat{R}^F_l : 
\Sigma \times \Sigma \to \cl ( V^{\otimes l} ; W) $
for $x,y \in \Sigma$ and $v \in V^{\otimes l}$ by
\beq
    \label{lip_k_cover_thm_claim_alt_remainders}
        \hat{R}^{F}_l(x,y)[v]
        :=
        \twopartdef{R^{F}_l(x,y)[v]}{q=k}
        {R^{F}_l(x,y)[v]
        +
        \sum_{s=q+1-l}^{k-l} \frac{1}{s!}
        F^{(l+s)}(x) 
        \left[ v \otimes (y-x)^{\otimes s}
        \right]} {q < k.}
\eeq
It follows from 
\eqref{lip_k_cover_thm_claim_expansions_1}
and 
\eqref{lip_k_cover_thm_claim_alt_remainders}
that whenever $l \in \{0, \ldots , q\}$, 
$x,y \in \Sigma$ and $v \in V^{\otimes l}$ we have 
\beq
    \label{lip_k_cover_thm_claim_expansions_2}
        F^{(l)}(y)[v]
        =
        \sum_{s=0}^{q-l}
        F^{(l+s)}(x) \left[ 
        v \otimes (y-x)^{\otimes s} \right] 
        +
        \hat{R}^{F}_l(x,y)[v].
\eeq
For each $z \in B$, the assumption that
$||F_{[q]}||_{\Lip(\eta, \Omega_z,W)} 
\leq \th \ep$ for $\Omega_z := \Sigma ~\cap~
\ovB_V (z,2\de_0)$ tells us that
for every $l \in \{0,\ldots,q\}$ and any
$x , y \in \Sigma ~\cap~\ovB_V (z,2\de_0)$
we have
\beq
    \label{lip_k_cover_thm_claim_l_boundsB}
        (\bI) \quad \left|\left| F^{(l)}(x)
        \right|\right|_{\cl (V^{\otimes l} ; W)}
        \leq \th \ep
        \qquad \text{and} \qquad 
        (\bII) \quad 
        \left|\left| \hat{R}^{F}_l(x,y)
        \right|\right|_{\cl (V^{\otimes l} ; W)}
        \leq
        \th \ep || y - x ||_V^{\eta - l}.
\eeq
Consider $p \in \Sigma$. From 
\eqref{lip_k_cover_thm_cover_prop} we know that
$p \in \Sigma ~\cap~\ovB_V (z,\de_0)$ for some 
$z \in B$. Consequently the bound (\bI) in
\eqref{lip_k_cover_thm_claim_l_boundsB} holds
for $x := p$. Since $p \in \Sigma$ was arbitrary, we
conclude that for any $p \in \Sigma$ we have
\beq
    \label{lip_k_cover_thm_claim_global_l_bound}
        \left|\left| F^{(l)}(p)
        \right|\right|_{\cl (V^{\otimes l} ; W)}
        \leq
        \th \ep.
\eeq
Consider $l \in \{0, \ldots , q\}$ and  
$p,w \in \Sigma$. If there exists $z \in B$ for which 
$p,w \in \ovB_V (z, 2\de_0)$ then (\bII) in 
\eqref{lip_k_cover_thm_claim_l_boundsB} yields that
\beq
    \label{lip_k_cover_thm_claim_global_remain_bound_A}
        \left|\left| \hat{R}^{F}_l(p,w)
        \right|\right|_{\cl (V^{\otimes l} ; W)}
        \leq
        \th \ep || w - p ||_V^{\eta - l}. 
\eeq
Now suppose that no
single ball $\ovB_V(z , 2\de_0)$ contains both
$p$ and $w$. From 
\eqref{lip_k_cover_thm_cover_prop} we know that
$p \in \Sigma ~\cap~\ovB_V (z_i, \de_0)$ and 
$w \in \Sigma ~\cap~\ovB_V (z_j,\de_0)$ for some 
$z_i, z_j \in B$ which must be distinct.
In fact, since $w \notin 
\Sigma ~\cap~\ovB_V (z_i,2\de_0)$, we can conclude that
\beq
    \label{lip_k_cover_thm_claim_de0_sep}
        || w - p ||_V 
        \geq
        \de_0.
\eeq
Observe that from 
\eqref{lip_k_cover_thm_claim_expansions_2}
we have, for any $v \in V^{\otimes l}$, that
\beq
	\label{lip_k_cover_thm_claim_remain_def}
		\hat{R}^{F}_l(p,w)[v]
		=
		F^{(l)}(w)[v]
		-
		\sum_{s=0}^{q-l} \frac{1}{s!}
		F^{(l+s)}(p) \left[
		v \otimes (w - p)^{\otimes s}
		\right].
\eeq
We may further use 
\eqref{lip_k_cover_thm_claim_expansions_2}
to compute that
\beq
    \label{lip_k_cover_thm_claim_varphi_l_expand}
        F^{(l)}(w)[v]
        =
        \sum_{u=0}^{q-l} \frac{1}{u!}
        F^{(l+u)}(z_j) \left[ 
        v \otimes (w - z_j)^{\otimes u} 
        \right]
        +
        \hat{R}^{F}_l(z_j,w)[v].
\eeq
Since $w \in \ovB_V(z_j,\de_0)$ we may use
\eqref{lip_k_cover_thm_claim_global_remain_bound_A}
to conclude that
\beq
    \label{lip_k_cover_thm_claim_w_zj_remain_bound}
        \left|\left| \hat{R}^{F}_l(z_j,w)
        \right|\right|_{\cl (V^{\otimes l} ; W)}
        \leq
        \th \ep || w - z_j ||_V^{\eta - l}
        \leq
        \th \ep \de_0^{\eta - l}. 
\eeq
Additionally, since $z_j \in B$ we may compute that
\beq
	\label{lip_k_cover_thm_claim_w_zj_varphi_l_bound}
		\sum_{u=0}^{q-l} \frac{1}{u!}
        	\left|\left|F^{(l+u)}(z_j) \left[ 
        	v \otimes (w - z_j)^{\otimes u} 
        	\right] \right|\right|_W
        	\leq
        	\ep_0
        	\sum_{u=0}^{q-l} \frac{1}{u!}
        	||w - z_j||_V^u || v ||_{V^{\otimes l}}
        	\leq
        	\ep_0 e^{\de_0}
        	|| v ||_{V^{\otimes l}}.
\eeq
Combining 
\eqref{lip_k_cover_thm_claim_varphi_l_expand},
\eqref{lip_k_cover_thm_claim_w_zj_remain_bound},
and 
\eqref{lip_k_cover_thm_claim_w_zj_varphi_l_bound}
yields the estimate
\beq
	\label{lip_k_cover_thm_caim_varphi^l(w)_bound_a}
		\left| \left| F^{(l)}(w)[v]
		\right|\right|_W
		\leq
		\left( \th \ep \de_0^{\eta - l} 
		+ \ep_0 e^{\de_0} \right)
		|| v ||_{V^{\otimes l}}.
\eeq
Turning our attention to the second term in 
\eqref{lip_k_cover_thm_claim_remain_def},
note that for any $s \in \{0, \ldots , q - l\}$
we have via
\eqref{lip_k_cover_thm_claim_expansions_2}, 
for $v' := v \otimes (w-p)^{\otimes s} \in 
V^{\otimes (l+s)}$, that
\beq
    \label{lip_k_cover_thm_claim_varphi^l+s(p)_expand}
        F^{(l+s)}(p)
        \left[ v'
        \right]
        =
        \sum_{u=0}^{q-l-s} \frac{1}{u!}
        F^{(l+s+u)}(z_i) \left[ v'
        \otimes (p - z_i)^{\otimes u} 
        \right]
        +
        \hat{R}^{F}_{l+s}(z_i,p)
        \left[ v'
        \right].
\eeq
Since $p \in \ovB_V(z_i,\de_0)$ we may use
\eqref{lip_k_cover_thm_claim_global_remain_bound_A}
to conclude that
\beq
    \label{lip_k_cover_thm_claim_p_zi_remain_bound}
        \left|\left| \hat{R}^{F}_{l+s}(z_i,p)
        \right|\right|_{\cl (V^{\otimes l} ; W)}
        \leq
        \th \ep || p - z_i ||_V^{\eta - (l+s)}
        \leq
        \th \ep \de_0^{\eta - (l+s)}. 
\eeq
Via similar computations to those used to establish
\eqref{lip_k_cover_thm_claim_w_zj_varphi_l_bound},
the fact that $z_i \in B$ allows us to compute that
\beq
	\label{lip_k_cover_thm_claim_p_zi_varphi_l_bound}
		\sum_{u=0}^{q-l-s} \frac{1}{u!}
        	\left|\left|F^{(l+u+s)}(z_j) \left[ 
        	v' \otimes (p - z_i)^{\otimes u} 
        	\right] \right|\right|_W
        	\leq
        	\ep_0 e^{\de_0}
        	|| v' ||_{V^{\otimes (l+s)}}.
\eeq
Combining 
\eqref{lip_k_cover_thm_claim_varphi^l+s(p)_expand},
\eqref{lip_k_cover_thm_claim_p_zi_remain_bound}, and
\eqref{lip_k_cover_thm_claim_p_zi_varphi_l_bound}
yields that
\beq
    \label{lip_k_cover_thm_claim_varphi^l+s(p)_bound}
        \left|\left|F^{(l+s)}(p)
        \left[ v'
        \right]\right|\right|_W
        \leq
        \left( \ep_0 e^{\de_0}
        + \th \ep \de_0^{\eta - (l+s)}
        \right)
        || v' ||_{V^{\otimes (l+s)}}
        =
        \left( \ep_0 e^{\de_0}
        + \th \ep \de_0^{\eta - (l+s)}
        \right)
        || w - p ||_V^s
        || v ||_{V^{\otimes l}}
\eeq
where the last equality uses that
$v' := v \otimes (w-p)^{\otimes s}$
and that the tensor powers of $V$ are equipped
with admissible norms (cf. Definition 
\ref{admissible_tensor_norm}).
A consequence of 
\eqref{lip_k_cover_thm_claim_varphi^l+s(p)_bound}
is that
\beq
	\label{lip_k_cover_thm_claim_sum_varphi^l+s(p)_bound_a}
		\sum_{s=0}^{q-l} \frac{1}{s!}
		\left|\left|F^{(l+s)}(p)
        	\left[ v'
        	\right]\right|\right|_W
        	\leq
        	\sum_{s=0}^{q-l} \frac{1}{s!}
        	\left( \ep_0 e^{\de_0}
        	+ \th \ep \de_0^{\eta - (l+s)}
        	\right)
        	|| w - p ||_V^s
        	|| v ||_{V^{\otimes l}}.
\eeq
Since from \eqref{lip_k_cover_thm_claim_de0_sep}
we have that $\de_0 \leq || w - p ||_V$, we may 
multiply each term in the sum on the RHS of
\eqref{lip_k_cover_thm_claim_sum_varphi^l+s(p)_bound_a}
by $|| w - p||_V^{\eta - (l+s)} \de_0^{-(\eta - (l+s))}
\geq 1$ to conclude that
\beq
	\label{lip_k_cover_thm_claim_sum_varphi^l+s(p)_bound_b}
		\sum_{s=0}^{q-l} \frac{1}{s!}
		\left|\left|F^{(l+s)}(p)
        	\left[ v'
        	\right]\right|\right|_W
        	\leq
        	\sum_{s=0}^{q-l} \frac{1}{s!}
        	\left( \ep_0 e^{\de_0}
        	+ \th \ep \de_0^{\eta - (l+s)}
        	\right)
        	\de_0^{-(\eta-l-s)}
        	|| w - p ||_V^{\eta - l} 
        	|| v ||_{V^{\otimes l}}.
\eeq
Combining \eqref{lip_k_cover_thm_claim_remain_def},
\eqref{lip_k_cover_thm_caim_varphi^l(w)_bound_a}, and
\eqref{lip_k_cover_thm_claim_sum_varphi^l+s(p)_bound_b}
allows us to deduce that
\beq
	\label{lip_k_cover_thm_remain_pw_v_bound}
		\left|\left| \hat{R}^F_l(p,w)[v]
		\right|\right|_W
		\leq
		\left( \th \ep \de_0^{\eta - l} 
		+ \ep_0 e^{\de_0} +
		\sum_{s=0}^{q-l} \frac{1}{s!}
        	\left( \ep_0 e^{\de_0}
        	\de_0^{- (\eta - l - s)}
        	+ \th \ep
        	\right)
        	|| w - p ||_V^{\eta - l}
        	\right)
        	|| v ||_{V^{\otimes l}}.
\eeq
Observe that
\beq
	\label{lip_k_cover_thm_remain_pw_v_term1}
		\th \ep \de_0^{\eta - l} 
		\stackrel{
		\eqref{lip_k_cover_thm_claim_de0_sep}
		}{\leq}
		\th \ep || w - p||_V^{\eta - l},
\eeq
\beq
	\label{lip_k_cover_thm_remain_pw_v_term2}
		\ep_0 e^{\de_0} 
		\stackrel{
		\eqref{lip_k_cover_thm_claim_de0_sep}
		}{\leq}
		\ep_0 \de_0^{-(\eta - l)} e^{\de_0} 
		|| w - p||_V^{\eta - l},
\eeq
\beq
	\label{lip_k_cover_thm_remain_pw_v_term3}
        	\ep_0 e^{\de_0}
        	\de_0^{- (\eta - l)}
        	\sum_{s=0}^{q-l} \frac{1}{s!} \de_0^s
		\leq
		\ep_0 \de_0^{-(\eta-l)} e^{2\de_0}, \quad \text{and}
\eeq
\beq
	\label{lip_k_cover_thm_remain_pw_v_term4}
		\th \ep \sum_{s=0}^{q-l} \frac{1}{s!} 
		\leq
		\th e \ep.
\eeq
Combining 
\eqref{lip_k_cover_thm_remain_pw_v_term1},
\eqref{lip_k_cover_thm_remain_pw_v_term2},
\eqref{lip_k_cover_thm_remain_pw_v_term3}, and 
\eqref{lip_k_cover_thm_remain_pw_v_term4} with 
\eqref{lip_k_cover_thm_remain_pw_v_bound} yields 
\beq
	\label{lip_k_cover_thm_remain_pw_v_bound_2}
		\left|\left| \hat{R}^{\varphi}_l(p,w)[v]
		\right|\right|_W
		\leq
		\left( \th \ep (1 + e) 
		+ \ep_0 \de_0^{-(\eta - l)}
		e^{\de_0} \left(
		1 + e^{\de_0} \right) 
        	\right)
        	|| w - p ||_V^{\eta - l}
        	|| v ||_{V^{\otimes l}}.
\eeq
Taking the supremum over $v \in V^{\otimes l}$
with unit $V^{\otimes l}$-norm in 
\eqref{lip_k_cover_thm_remain_pw_v_bound_2}
yields
\beq
	\label{lip_k_cover_thm_remain_pw_bound}
		\left|\left| \hat{R}^{F}_l(p,w)
		\right|\right|_{ \cl (V^{\otimes l} ; W)}
		\leq
		\left( \th \ep (1 + e) 
		+ \ep_0 \de_0^{-\eta }
		e^{\de_0} \left(
		1 + e^{\de_0} \right) \right)
        || w - p ||_V^{\eta - l}
\eeq
since $\de_0 \leq 1$ means
$\de_0^{-(\eta -l)} \leq \de_0^{-\eta}$ 
for every $l \in \{0, \ldots , q\}$.

Together
\eqref{lip_k_cover_thm_claim_global_remain_bound_A},
\eqref{lip_k_cover_thm_remain_pw_bound} and the inequality
$\th \ep < \th \ep (1 + e) + \ep_0 \de_0^{-\eta }
e^{\de_0} \left(1 + e^{\de_0} \right)$ mean that 
for any $l \in \{0, \ldots , q\}$ and 
any $p,w \in \Sigma$ we have
\beq
	\label{lip_k_cover_thm_remain_bound}
		\left|\left| \hat{R}^{F}_l(p,w)
		\right|\right|_{ \cl (V^{\otimes l} ; W)}
		\leq
		\left( \th \ep (1 + e) 
		+ \ep_0
		\frac{e^{\de_0} \left(
		1 + e^{\de_0} \right)}
		{\de_0^{\eta }} \right)
        || w - p ||_V^{\eta - l}.
\eeq
The definitions 
\eqref{lip_k_cover_thm_claim_expansions_2},
the bounds 
\eqref{lip_k_cover_thm_claim_global_l_bound}, and 
the H\"{o}lder estimates 
\eqref{lip_k_cover_thm_remain_bound} 
tell us that
\begin{align*}
	\left|\left| F_{[q]} 
	\right|\right|_{\Lip(\eta,\Sigma,W)}
		&\leq \th \ep (1 + e) 
		+ \ep_0 \frac{e^{\de_0} \left(
		1 + e^{\de_0} \right)}
		{\de_0^{\eta }}  \\
		&\stackrel{
		\eqref{lip_k_cover_thm_ep0_def}
		}{\leq}
		\frac{1}{2(1+e)}\ep(1+e)
		+
		\frac{\ep}{2}\frac{\de_0^{\eta}}
		{e^{\de_0}
		\left( 1 + e^{\de_0}\right)}
		\frac{e^{\de_0} \left(
		1 + e^{\de_0} \right)}
		{\de_0^{\eta }}
		=
		\frac{\ep}{2} + \frac{\ep}{2}
		=
		\ep
\end{align*}
as claimed in \eqref{lip_k_cover_thm_claim_conc}.
This completes the proof of Claim 
\ref{lip_k_cover_thm_claim}
\end{proof}
\vskip 4pt
\noindent
Returning to the proof of Theorem 
\ref{lip_k_cover_thm} itself, we define 
$F := \psi - \vph \in \Lip(\gamma,\Sigma,W)$
so that for every $j \in \{0, \ldots , k\}$
we have $F^{(j)} = \psi^{(j)} - \vph^{(j)}$.
Then, by assumption, we have 
for every $j \in \{0,\ldots,k\}$ and every $z \in B$ that 
$\left|\left|F^{(j)}(z) 
\right|\right|_{\cl(V^{\otimes j};W)} \leq \ep_0$.
Moreover, \eqref{lip_k_cover_thm_bd_1} tells us 
that whenever $z \in B$ we have that 
$\left|\left|F_{[q]}
\right|\right|_{\Lip(\eta,\Omega_z,W)}\leq \th \ep$
for $\Omega_z := \Sigma ~\cap~ \ovB_V(z,2\de_0)$.
Therefore we can apply Claim \ref{lip_k_cover_thm_claim}
to $F$ and conclude that 
$\left|\left| F_{[q]} 
\right|\right|_{\Lip(\eta, \Sigma,W)} \leq \ep$.
Since $F := \psi - \vph$ this gives the estimate 
claimed in \eqref{lip_k_cover_thm_conc} and 
completes the proof of Theorem \ref{lip_k_cover_thm}.
\end{proof}

\bibliographystyle{plain}
\bibliography{HOLST_bibliography_database}

\begin{thebibliography}{10}

\bibitem{Bou15}
Youness Boutaib.
\newblock On lipschitz maps and the h{\"o}lder regularity of flows.
\newblock {\em Revue Roumaine de Mathematiques Pures et Appliquees}, 65(2):129--175, 2020.

\bibitem{Bou22}
Youness Boutaib.
\newblock The accessibility problem for geometric rough differential equations.
\newblock {\em Journal of Dynamical and Control Systems}, 29(4):1669--1693, 2023.

\bibitem{BL22}
Youness Boutaib and Terry Lyons.
\newblock A new definition of rough paths on manifolds.
\newblock In {\em Annales de la Facult{\'e} des sciences de Toulouse: Math{\'e}matiques}, volume~31, pages 1223--1258, 2022.

\bibitem{BS94}
Yuri Brudnyi and Pavel Shvartsman.
\newblock Generalizations of whitney's extension theorem.
\newblock {\em International Mathematics Research Notices}, 1994(3):129--139, 1994.

\bibitem{BS01}
Yuri Brudnyi and Pavel Shvartsman.
\newblock Whitney’s extension problem for multivariate $c^{1,\omega}$-functions.
\newblock {\em Transactions of the American Mathematical Society}, 353(6):2487--2512, 2001.

\bibitem{CLL12}
Thomas Cass, Christian Litterer, and Terry Lyons.
\newblock Rough paths on manifolds.
\newblock In {\em New Trends In Stochastic Analysis And Related Topics: A Volume in Honour of Professor KD Elworthy}, pages 33--88. World Scientific, 2012.

\bibitem{Sch50}
Mahlon~M Day.
\newblock Robert schatten, a theory of cross-spaces.
\newblock {\em Bull. Amer. Math. Soc.}, 57(6):326--327, 1951.

\bibitem{Fef06}
Charles Fefferman.
\newblock Whitney's extension problem for $c^m$.
\newblock {\em Annals of mathematics}, pages 313--359, 2006.

\bibitem{Fef07}
Charles Fefferman.
\newblock $c^{m}$ extension by linear operators.
\newblock {\em Annals of Mathematics}, 166(3):779--835, 2007.

\bibitem{Fef09-I}
Charles Fefferman.
\newblock Extension of $c^{m,\omega}$-smooth functions by linear operators.
\newblock {\em Revista Matem{\'a}tica Iberoamericana}, 25(1):1--48, 2009.

\bibitem{Fef09-II}
Charles Fefferman.
\newblock Fitting a-smooth function to data, iii.
\newblock {\em Annals of mathematics}, pages 427--441, 2009.

\bibitem{FIL16}
Charles Fefferman, Arie Israel, and Garving~K Luli.
\newblock Finiteness principles for smooth selection.
\newblock {\em Geometric and Functional Analysis}, 26:422--477, 2016.

\bibitem{FIL17}
Charles Fefferman, Arie Israel, and Garving~K Luli.
\newblock Interpolation of data by smooth nonnegative functions.
\newblock {\em Revista matem{\'a}tica iberoamericana}, 33(1):305--324, 2017.

\bibitem{FK09-II}
Charles Fefferman and Bo'az Klartag.
\newblock Fitting a $c^m$-smooth function to data ii.
\newblock {\em Revista Matem{\'a}tica Iberoamericana}, 25(1):49--273, 2009.

\bibitem{FK09-I}
Charles Fefferman and Bo'az Klartag.
\newblock Fitting a-smooth function to data i.
\newblock {\em Annals of mathematics}, pages 315--346, 2009.

\bibitem{Fef05}
Charles~L Fefferman.
\newblock A sharp form of whitney's extension theorem.
\newblock {\em Annals of mathematics}, 161(1):509--577, 2005.

\bibitem{FSS01}
Bruno Franchi, Raul Serapioni, and Francesco~Serra Cassano.
\newblock Rectifiability and perimeter in the heisenberg group.
\newblock {\em Mathematische Annalen}, 321(3), 2001.

\bibitem{JS17}
Nicolas Juillet and Mario Sigalotti.
\newblock Pliability, or the whitney extension theorem for curves in carnot groups.
\newblock {\em Analysis \& PDE}, 10(7):1637--1661, 2017.

\bibitem{Lyo98}
Terry~J Lyons.
\newblock Differential equations driven by rough signals.
\newblock {\em Revista Matem{\'a}tica Iberoamericana}, 14(2):215--310, 1998.

\bibitem{CLL04}
Terry~J Lyons, Michael Caruana, and Thierry L{\'e}vy.
\newblock {\em Differential equations driven by rough paths}.
\newblock Springer, 2007.

\bibitem{LY15}
Terry~J Lyons and Danyu Yang.
\newblock The theory of rough paths via one-forms and the extension of an argument of schwartz to rough differential equations.
\newblock {\em Journal of the Mathematical Society of Japan}, 67(4):1681--1703, 2015.

\bibitem{McS34}
EJ~McShane.
\newblock Extension of range of functions.
\newblock {\em Bull. Amer. Math. Soc.}, 40(12):837--842, 1934.

\bibitem{Nej18}
Sina Nejad.
\newblock {\em Lipschitz functions on unparameterised rough paths and the Brownian motion associated to the bilaplacian}.
\newblock PhD thesis, University of Oxford, 2018.

\bibitem{PSZ19}
Andrea Pinamonti, Gareth Speight, and Scott Zimmerman.
\newblock A $c^m$ whitney extension theorem for horizontal curves in the heisenberg group.
\newblock {\em Transactions of the American Mathematical Society}, 371(12):8971--8992, 2019.

\bibitem{PV06}
IM~Pupyshev and Sergey~Konstantinovich Vodop’yanov.
\newblock Whitney-type theorems on the extension of functions on carnot groups.
\newblock {\em Sibirsk. Mat. Zh.}, 47(4):731--752, 2006.

\bibitem{Rya02}
Raymond~A Ryan.
\newblock {\em Introduction to tensor products of Banach spaces}, volume~73.
\newblock Springer, 2002.

\bibitem{SS18}
Ludovic Sacchelli and Mario Sigalotti.
\newblock On the whitney extension property for continuously differentiable horizontal curves in sub-riemannian manifolds.
\newblock {\em Calculus of Variations and Partial Differential Equations}, 57(2):59, 2018.

\bibitem{Ste70}
Elias~M Stein.
\newblock {\em Singular integrals and differentiability properties of functions}.
\newblock Princeton university press, 1970.

\bibitem{Whi34-I}
Hassler Whitney.
\newblock Analytic extensions of differentiable functions defined in closed sets.
\newblock {\em Transactions of the American Mathematical Society}, 36(1):63--89, 1934.

\bibitem{Whi34-II}
Hassler Whitney.
\newblock Differentiable functions defined in closed sets. i.
\newblock {\em Transactions of the American Mathematical Society}, 36(2):369--387, 1934.

\bibitem{Whi44}
Hassler Whitney.
\newblock On the extension of differentiable functions.
\newblock {\em Bulletin of the American Mathematical Society}, 50(2):76--81, 1944.

\bibitem{Yan16}
Danyu Yang.
\newblock Integration of geometric rough paths.
\newblock {\em arXiv preprint arXiv:1611.06144}, 2016.

\bibitem{Zim18}
Scott Zimmerman.
\newblock The whitney extension theorem for $c^1$, horizontal curves in the heisenberg group.
\newblock {\em The Journal of Geometric Analysis}, 28(1):61--83, 2018.

\bibitem{Zim21}
Scott Zimmerman.
\newblock Whitney's extension theorem and the finiteness principle for curves in the heisenberg group.
\newblock {\em Revista Mathematica Iberoamericana}, 39(2), 2023.

\end{thebibliography}
\vskip 4pt 
\noindent
University of Oxford, Radcliffe Observatory,
Andrew Wiles Building, Woodstock Rd, Oxford, 
OX2 6GG, UK.
\vskip 4pt
\noindent
TL: tlyons@maths.ox.ac.uk \\
\url{https://www.maths.ox.ac.uk/people/terry.lyons}
\vskip 4pt
\noindent
AM: andrew.mcleod@maths.ox.ac.uk \\
\url{https://www.maths.ox.ac.uk/people/andrew.mcleod}
\end{document}